\documentclass[12pt]{amsart}
\usepackage{amsmath,amsfonts,amssymb,graphicx,hyperref, enumitem,amsthm,mathabx}
\usepackage{amsrefs}

\setlist[enumerate,1]{label={\arabic*.}}

\newcommand{\Z}{\mathbb{Z}}
\newcommand{\R}{\mathbb{R}}
\newcommand{\N}{\mathbb{N}}
\newcommand{\Q}{\mathbb{Q}}
\newcommand{\C}{\mathbb{C}}

\newcommand{\paren}[1]{\ensuremath{\left( #1 \right)}}
\newcommand{\set}[1]{\ensuremath{\left\{ #1 \right\}}}
\newcommand{\innerprod}[1]{\ensuremath{\left< #1 \right>}}
\newcommand{\norm}[1]{\ensuremath{\left\| #1 \right\|}}
\newcommand{\abs}[1]{\ensuremath{\left| #1 \right|}}
\newcommand{\setdiv}{\,\middle|\,}
\newcommand{\e}[1]{e\paren{#1}}
\newcommand{\summod}[1]{\ensuremath{\,(\mathrm{mod}\,#1)}}

\newcommand{\Matrix}[1]{\begin{pmatrix}#1\end{pmatrix}}

\newcommand\revdots{\mathinner{\mkern1mu\raise1pt\vbox{\kern7pt\hbox{.}}\mkern2mu\raise4pt\hbox{.}\mkern2mu \raise7pt\hbox{.}\mkern1mu}}

\newcommand{\Max}[1]{\ensuremath{\max \set{#1}}}

\newcommand{\floor}[1]{\ensuremath{\left\lfloor #1 \right\rfloor}}

\newcommand{\piecewise}[1]{\left\{\begin{matrix}#1\end{matrix}\right.}
\newcommand{\If}{\mbox{if }}
\newcommand{\Otherwise}{\mbox{otherwise}}

\renewcommand{\Re}{{\mathop{\mathgroup\symoperators Re}}}

\newcommand{\sgn}{{\mathop{\mathgroup\symoperators \,sgn}}}

\newcommand{\wbar}[1]{\overline{#1}}
\newcommand{\wtilde}[1]{\widetilde{#1}}
\newcommand{\what}[1]{\widehat{#1}}
\newcommand{\wcheck}[1]{\widecheck{#1}} 

\newcommand{\trans}[1]{{#1}^T}

\theoremstyle{plain} 
\newtheorem{thm}{Theorem}
\newtheorem{cor}[thm]{Corollary}
\newtheorem{lem}[thm]{Lemma}
\newtheorem{prop}[thm]{Proposition}
\newtheorem{conj}[thm]{Conjecture}

\theoremstyle{remark}

\DeclareMathAlphabet{\mathcalligra}{T1}{calligra}{m}{n}
\newcommand{\WigDName}{\mathcal{D}}
\newcommand{\WigDMat}[1]{\WigDName^{#1}}
\newcommand{\WigDRow}[2]{\WigDMat{#1}_{#2}}
\newcommand{\WigD}[3]{\WigDMat{#1}_{#2,#3}}
\newcommand{\WigdName}{\mathcalligra{d}}
\newcommand{\Wigd}[3]{\WigdName^{#1}_{#2,#3}}
\newcommand{\JacobiP}[3]{P^{(#2,#3)}_{#1}}
\DeclareMathOperator{\Tr}{Tr}

\DeclareMathOperator{\diag}{diag}

\newcommand{\tildek}[1]{\wtilde{k}\paren{#1}}
\newcommand{\Dtildek}[2]{\mathcal{R}^{#1}\paren{#2}}

\newcommand{\vpmpm}[1]{v_{_{#1}}}
\newcommand{\chipmpm}[1]{\chi_{_{#1}}}
\newcommand{\tildeWigP}[2]{\wtilde{\mathcal{P}}^{#1}_{#2}}
\newcommand{\tildeWigPd}[1]{\wtilde{\mathcal{P}}^{#1}}
\newcommand{\WigP}[2]{\mathcal{P}^{#1}_{#2}}

\newcommand{\CG}[2]{\frak{A}^{#1}_{#2}}
\newcommand{\CGB}[2]{\frak{B}^{#1}_{#2}}
\DeclareMathOperator{\proj}{proj}
\DeclareMathOperator{\Span}{span}
\newcommand{\bv}{\frak{v}}
\newcommand{\bu}{\frak{u}}
\newcommand{\KLeft}[1]{K^\text{Left}_{#1}}

\title{Higher weight on $GL(3)$, II:\\The cusp forms.}
\author{Jack Buttcane}
\date{14 April 2018}
\address{Mathematics Department, 244 Mathematics Building, Buffalo, NY 14260, USA}
\email{buttcane@buffalo.edu}
\thanks{During the time of this research, the author was supported by NSF grant DMS-1601919.}

\begin{document}

\begin{abstract}
The purpose of this paper is to collect, extend, and make explicit the results of Gel$'$fand, Graev and Piatetski-Shapiro
and Miyazaki
for the $GL(3)$ cusp forms which are non-trivial on $SO(3,\R)$.
We give new descriptions of the spaces of cusp forms of minimal $K$-type and from the Fourier-Whittaker expansions of such forms give a complete and completely explicit spectral expansion for $L^2(SL(3,\Z)\backslash PSL(3,\R))$, accounting for multiplicities, in the style of Duke, Friedlander and Iwaniec's paper.
We do this at a level of uniformity suitable for Poincar\'e series which are not necessarily $K$-finite.
We directly compute the Jacquet integral for the Whittaker functions at the minimal $K$-type, improving Miyazaki's computation.
These results will form the basis of the non-spherical spectral Kuznetsov formulas and the arithmetic/geometric Kuznetsov formulas on $GL(3)$.
The primary tool will be the study of the differential operators coming from the Lie algebra on vector-valued cusp forms.
\end{abstract}


\maketitle

\section{Introduction}

In the previous paper \cite{PartI}, we worked out the continuous and residual spectra in the Langlands decomposition in the case of $L^2(SL(3,\Z)\backslash PSL(3,\R))$.
We turn now specifically to the cuspidal spectral decomposition.
The initial decomposition into eigenspaces of the Casimir operators is originally due, in much greater generality, to Gel$'$fand, Graev and Piatetski-Shapiro \cite{GGPS}, who described the decomposition in terms of integral operators (in the style of Selberg \cite{Sel01}) operating on representation spaces in the $L^2$-space.
In modern terminology, this is the study of irreducible unitary representations and $(\frak{g},K)$ modules, and in the particular case of $PSL(3,\R)$, the former was initiated by Vahutinski\u{\i} \cite{Vahu} and the latter by Howe \cite{Howe} and Miyazaki \cite{Miyagk}.

From the representation-theoretic description, this paper will analyze the structure of the cuspidal part of the principal series representations as they decompose over the entries of the Wigner $\WigDName$-matrices, the intertwining operators, and the Jacquet integral as an operator from the principal series representation to its Whittaker model.
From the analytic perspective, we are decomposing the eigenspaces of the Casimir operators on the $L^2$-space via raising and lowering operators (in the style of Maass \cite{Maass01}) obtained from the Lie algebra; we show the equivalence of several descriptions of the minimal $K$-types, and compute Mellin-Barnes integrals for the Whittaker functions attached to those $K$-types.
We will largely avoid the representation-theoretic description outside of the introduction, section \ref{sect:MainThms}, and appendix \ref{app:MultOne}; partial descriptions in that language maybe found in \cite[appendix A]{MillerSchmid} and \cite{MeMiller}.

The goal of this paper is to describe, as completely, explicitly and uniformly as possible, the spectral expansion of the cuspidal part of $L^2(SL(3,\Z)\backslash PSL(3,\R))$, and provide all of the associated information necessary for a number theorist to apply analysis on this space as well as at the minimal $K$-types.
We classify the spectral parameters of the cusp forms at the minimal $K$-types, and then, taking as input the spectral parameters and Fourier-Whittaker coefficients of such minimal cusp forms, generate the rest of the spectral expansion.
The classification of the minimal $K$-types is given in theorem \ref{thm:GDCusp}, and the spectral expansion is theorem \ref{thm:CuspSpectralExpand}.

The uniformity of theorem \ref{thm:CuspSpectralExpand} is necessary if one needs to expand, say, a Poincar\'e series which is not $K$-finite.
Since certain cusp forms (and Eisenstein series) miss a number of $K$-types (i.e. generalized principal series forms), one would typically expect such an expansion to require evaluating the Wigner coefficients of the $K$-part of the Poincar\'e series, a task which is at best difficult for any large class of test functions.
However, in theorem \ref{thm:CuspSpectralExpand}, the sum is always taken over \textit{all} $K$-types, relying on the fact that the Jacquet-Whittaker function for any given cusp form is simply zero on the types missed by that form.
It is not too hard to see that the Archimedian weight function for a generalized principal series form in such an expansion, viewed as a sum over minimal-weight forms, is just the analytic continuation of the weight function for a full principal series form, and this can be evaluated without ever mentioning Wigner $\WigDName$-matrices at all.
We anticipate using such expansions to study smooth sums of exponential sums on $GL(3)$.

On the other hand, the majority of the study of analytic number theory on automorphic forms takes place at the minimal $K$-types.
With the longer history of automorphic forms on $GL(2)$, it is perhaps easy to lose track of the fundamental importance of the connection between the three realizations of the Whittaker functions:
\begin{enumerate}
\item The Whittaker functions of holomorphic and Maass cusp forms and Eisenstein series, which arise through their Fourier coefficients.
	These are the main number-theoretic objects which occur in $L$-functions (see \cite[section 6.5]{Gold01}), various period integral formulas \cite{JPSS01,Watson01}, etc.
	
\item The classical Whittaker functions, which arise as the solutions to differential equations or through generalized hypergeometric series.
	These are the main analytic objects, and show up in a multitude of integral formulas (see the index entry for ``Whittaker functions'' in \cite{GradRyzh}), treatises on asymptotics \cite{Olver01}, etc.
	
\item The Jacquet-Whittaker functions, which arise via an integral operator between representation spaces.
	These are the main algebraic objects, which have numerous functional identities (for $GL(3)$, these are equations (3.4)-(3.8) of \cite{PartI}, to name a few), and are the focus of much study in representation theory (\cite{Shal01} is fundamental there).
\end{enumerate}
The bulk of this paper is concerned with making this connection very precise at the minimal $K$-types, from which the spectral expansion follows fairly easily via a double induction argument in section \ref{sect:GoingUp}.

We give expressions for the Whittaker functions at the minimal $K$-types in theorem \ref{thm:MinWhitt}; these are due, in greatly different language, to Miyazaki \cite{MiyaWhitt} who obtained them by studying their differential equations and then applied a certain induction argument, but we take this a step farther by evaluating the base case directly from the Jacquet integral, thus fixing the relationship between elements of the principal series representation and the Whittaker model, as well as the cusp forms, via the Fourier expansion.
This is significant for the theory of special functions on $GL(3)$ because the Mellin-Barnes integral is simpler analytically, but the Jacquet integral has a certain algebraic structure (i.e. the many identities of \cite[section 3.1]{PartI}) that we fundamentally rely on in constructions such as the Kuznetsov formulas (see \cite{SpectralKuz, WeylI, WeylII}, and note the extensive use of the properties of the Jacquet-Whittaker functions there).
Conversely, we must then rely on Miyazaki's solution of the differential equations (or representation theory) for the uniqueness property.

Such precise knowledge on the minimal-weight forms is, to the author, of primary use in developing spectral Kuznetsov/Petersson trace formulas and thereby Weyl laws and the many derived applications for studying such forms.
These will appear in future papers on the topic, e.g. \cite{WeylI, WeylII}.

A primary tenet of this paper, in combination with the previous part \cite{PartI}, is to be self-contained.
We generally succeed here except for three external pieces which are collected into theorem \ref{thm:ext}:
First, we do not prove the initial spectral decomposition; this is a standard proof using the techniques of Selberg studying integral operators first carried out by \cite{GGPS} and can be found in the first few pages of Harish-Chandra's book \cite{HC01}, or Langland's paper \cite{Langlands02}, or in a number of other books, e.g. \cite{OsWarn}, see also \cite[appendix A]{IwaIntro}.
Second, we quote the trivial bound for the principal series representations; this is a bound on Langlands parameters for the conjecturally non-existentent complementary series and can be found in \cite{Vahu}.
Lastly, we do not prove multiplicity one for Whittaker functions; from the analytic perspective, we are avoiding the solution to a lengthy problem in partial differential equations, but it follows in the representation-theoretic language from Shalika's paper \cite{Shal01}.
The details of these connections may be found in the discussion following theorem \ref{thm:ext}.

The proofs below will generally assume that the quotient is by $\Gamma=SL(3,\Z)$, but of course only the Fourier expansion will see the particular discrete group in use.

\section{Some notation and background from Part I}
We recall the notation of Part I of this series of papers \cite{PartI}.
Throughout the current paper, section, equation and theorem numbers beginning with an I reference \cite{PartI}, so for example theorem I.1.1 references \cite[theorem 1.1]{PartI} and (I.2.3) references \cite[eq. (2.3)]{PartI}.

Let $G=PSL(3,\R) = GL(3,\R)/\R^\times$ and $\Gamma=SL(3,\Z)$.
The Iwasawa decomposition of $G$ is $G=U(\R) Y^+ K$ using the groups $K=SO(3,\R)$,
\[ U(R) = \set{\Matrix{1&x_2&x_3\\&1&x_1\\&&1} \setdiv x_i\in R}, \qquad R \in \set{\R,\Q,\Z}, \]
\[ Y^+ = \set{\diag\set{y_1 y_2,y_1,1} \setdiv y_1,y_2 > 0}. \]
The measure on the space $U(\R)$ is simply $dx := dx_1 \, dx_2 \, dx_3$, and the measure on $Y^+$ is
\[ dy := \frac{dy_1 \, dy_2}{(y_1 y_2)^3}, \]
so that the measure on $G$ is $dg :=dx \, dy \, dk$, where $dk$ is the Haar probability measure on $K$ (see section I.2.2.1).
We generally identify elements of quotient spaces with their coset representatives, and in particular, we view $U(\R)$, $Y^+$, $K$ and $\Gamma$ as subsets of $G$.  (Note that these subspaces of $GL(3,\R)$ do inject into the quotient $G$, which is not generally the case for equivalent subspaces of $GL(2n,\R)$, since $-I \in SL(2n,\R)$.)

Characters of $U(\R)$ are given by
\[ \psi_m(x) = \psi_{m_1,m_2}(x) = \e{m_1 x_1+m_2 x_2}, \qquad \e{t} = e^{2\pi i t}, \]
where $m\in\R^2$.
Characters of $Y^+$ are given by the power function on $3\times 3$ diagonal matrices, defined by
\[ p_\mu\paren{\diag\set{a_1,a_2,a_3}} = \abs{a_1}^{\mu_1} \abs{a_2}^{\mu_2} \abs{a_3}^{\mu_3}, \]
where $\mu\in\C^3$.
We assume $\mu_1+\mu_2+\mu_3=0$ so this is defined modulo $\R^\times$, renormalize by $\rho=(1,0,-1)$, and extend by the Iwasawa decomposition
\[ p_{\rho+\mu}\paren{r x y k} = y_1^{1-\mu_3} y_2^{1+\mu_1}, \qquad r\in\R^\times,x\in U(\R),y\in Y^+,k\in K. \]

The Weyl group $W$ of $G$ contains the six matrices
\begin{equation*}
	\begin{array}{rclcrclcrcl}
		I &=& \Matrix{1\\&1\\&&1}, && w_2 &=& -\Matrix{&1\\1\\&&1}, && w_3 &=& -\Matrix{1\\&&1\\&1}, \\
		w_4 &=& \Matrix{&1\\&&1\\1}, && w_5 &=& \Matrix{&&1\\1\\&1}, && w_l &=& -\Matrix{&&1\\&1\\1}.
	\end{array}
\end{equation*}
The group of diagonal, orthogonal matrices $V \subset G$ contains the four matrices $v_{\varepsilon_1,\varepsilon_2}=\diag\set{\varepsilon_1,\varepsilon_1 \varepsilon_2,\varepsilon_2}, \varepsilon\in\set{\pm 1}^2$, which we abbreviate $V = \set{\vpmpm{++}, \vpmpm{+-}, \vpmpm{-+}, \vpmpm{--}}$.
The Weyl group induces an action on the coordinates of $\mu$ by $p_{\mu^w}(a):=p_\mu(waw^{-1})$, and we denote the coordinates of the permuted parameters by $\mu^w_i := (\mu^w)_i$, $i=1,2,3$.

Part I made explicit the continuous and residual parts of the Langlands spectral expansion, and it remains to do this for $L^2_\text{cusp}(\Gamma\backslash G)$, the space of square-integrable functions $f:\Gamma\backslash G\to\C$ satisfying the cuspidality condition described below theorem I.1.1, or equivalently, whose degenerate Fourier coefficients are all zero:
\begin{align}
\label{eq:CuspCond}
	\int_{U(\Z)\backslash U(\R)} f(ug) \psi_n(u) du = 0 \qquad \text{whenever $n_1 n_2=0$, $n \in \Z^2$.}
\end{align}
We will describe such functions in terms of their Fourier expansion (see section I.3.6) and their decomposition over the Wigner $\WigDName$-matrices.

If we describe elements $k=k(\alpha,\beta,\gamma)\in K$ in terms of the $Z$-$Y$-$Z$ Euler angles
\begin{align}
\label{eq:kabcDef}
	k(\alpha,\beta,\gamma) := k(\alpha,0,0) \, w_3 \, k(-\beta,0,0) \, w_3 \, k(\gamma,0,0), \quad k(\theta,0,0) := \Matrix{\cos\theta&-\sin\theta&0\\ \sin\theta&\cos\theta&0\\ 0&0&1},
\end{align}
then the Wigner $\WigDName$-matrix $\WigDMat{d}$ is the $(2d+1)$-dimensional representation of $K$ primarily characterized by
\begin{align}
\label{eq:WigDPrimary}
	\WigDMat{d}(k(\theta,0,0)) = \Dtildek{d}{e^{i\theta}}, \qquad \Dtildek{d}{s} := \diag\set{s^d,\ldots,s^{-d}}, s\in\C.
\end{align}
The entries of the matrix-valued function $\WigDMat{d}$ are indexed from the center:
\[ \WigDMat{d} = \Matrix{\WigD{d}{-d}{-d}&\ldots&\WigD{d}{-d}{d}\\ \vdots&\ddots&\vdots\\ \WigD{d}{d}{-d}&\ldots&\WigD{d}{d}{d}}, \]
so in particular $\WigD{d}{m'}{m}(k(\theta,0,0)) = e^{-im'\theta}$ when $m'=m$ and zero otherwise, as the indices $m'$ and $m$ run through the integers $-d,\ldots,d$.
Similarly, we index the rows $\WigDRow{d}{m'}=\paren{\WigD{d}{m'}{-d},\ldots,\WigD{d}{m'}{d}}$ and columns $\WigD{d}{\cdot}{m} = \trans{\paren{\WigD{d}{-d}{m},\ldots,\WigD{d}{d}{m}}}$ from the central entry, as well.
The entries, rows, and columns of of the derived matrix- and vector-valued functions (e.g. the Whittaker function \eqref{eq:JacWhittDef}) will be indexed similarly.
The Wigner $\WigDName$-matrices exhaust the equivalence classes of unitary, irreducible representations of the compact group $K$, hence they give a basis of $L^2(K)$, as in section I.2.2.1, by the Peter-Weyl theorem.

The entries of the matrix $\WigDMat{d}(k(0,\beta,0)) = \WigDMat{d}(w_3) \WigDMat{d}(k(-\beta,0,0)) \WigDMat{d}(w_3)$ are known as the Wigner $\WigdName$-polynomials.
For the most part, we will avoid the Wigner $\WigdName$-polynomials by treating $\WigDMat{d}(w_3)$ as a black box; that is, as some generic orthogonal matrix.
The notable exceptions are in section \ref{sect:WignerdJacobiP} and proposition \ref{prop:GDWhittFEs}, and we frequently use the facts (see section I.2.2.2)
\begin{align}
\label{eq:WigDv}
	\WigDMat{d}(v_{\varepsilon,+1}) =& \diag\set{\varepsilon^d,\ldots,\varepsilon^{-d}}, & \WigD{d}{m'}{m}(v_{\varepsilon,-1}) =& (-1)^d\varepsilon^{m'} \delta_{m'=-m}.
\end{align}
The notation $\delta_P$ here is one if the predicate $P$ is true, and zero if it is false.

A complete list of the characters of $V$, is given by $\chi_{\varepsilon_1,\varepsilon_2}, \varepsilon\in\set{\pm 1}^2$ which act on the generators by
\begin{align}
\label{eq:chipmpmdef}
	\chi_{\varepsilon_1,\varepsilon_2}(\vpmpm{-+})=&\varepsilon_1 & \chi_{\varepsilon_1,\varepsilon_2}(\vpmpm{+-}) =&\varepsilon_2.
\end{align}
These give rise to the projection operators
\begin{align}
\label{eq:SigmadchiDef}
	\Sigma^d_\chi = \frac{1}{4} \sum_{v\in V} \chi(v) \WigDMat{d}(v),
\end{align}
which are written out explicitly in section I.2.2.2 using the description \eqref{eq:WigDv}.
We sometimes abbreviate $\Sigma^d_{\varepsilon_1,\varepsilon_2} = \Sigma^d_{\chi_{\varepsilon_1,\varepsilon_2}}$.

Throughout the paper, we take the term ``smooth'', in reference to some function, to mean infinitely differentiable on the domain.
The letters $x,y,k,g, v$ and $w$ will generally refer to elements of $U(\R)$, $Y^+$, $K$, $G$, $V$, and $W$, respectively.
The letters $\chi$ and $\psi$ will generally refer to characters of $V$ and $U(\R)$, respectively, and $\mu$ will always refer to an element of $\C^3$ satisfying $\mu_1+\mu_2+\mu_3=0$.
Vectors (resp. matrices) not directly associated with the Wigner $\WigDName$-matrices, e.g. elements $n \in \Z^2$ are indexed in the traditional manner from the left-most entry (resp. the top-left entry), e.g.  $n=(n_1,n_2)$; in case it becomes necessary to make it explicit, we will refer to this as ``standard indexing''.

There is a technical point in relation to differentiability on quotient spaces:
We denote by $C^\infty(\Gamma\backslash G)$ the space of infinitely differentiable functions on $\Gamma\backslash G$, and one may define the differentiability on either of the smooth manifolds $\Gamma\backslash G$ or $G$ itself (i.e. $C^\infty(\Gamma\backslash G)$ as the space of infinitely differentiable functions of $G$ which are left-invariant by $\Gamma$) or conceivably by restricting to just the left-translation invariant differential operators, i.e. those coming from the Lie algebra of $G$ (see section \ref{sect:LieAlgebras}).
In the present case, it turns out the space of functions does not depend on these choices, and this is a discussion best left to a text on smooth manifolds (see \cite[theorem 9.16]{LeeManifolds}).

The majority of the paper will be concerned with the matrix-valued Jacquet-Whittaker function at each $K$-type $\WigDMat{d}$:
\begin{align}
\label{eq:JacWhittDef}
	W^d(g,\mu,\psi) := \int_{U(\R)} I^d(w_l u g, \mu) \wbar{\psi(u)} du, \qquad I^d(xyk,\mu) := p_{\rho+\mu}(y) \WigDMat{d}(k),
\end{align}
whose functional equations in $\mu$ (proposition I.3.3),
\begin{align}
\label{eq:WhittFEs}
	W^d(g,\mu,\psi_{1,1}) = T^d(w,\mu) W^d(g,\mu^w,\psi_{1,1}), \qquad w \in W,
\end{align}
are generated by the matrices
\begin{align}
\label{eq:Tdw2}
	T^d(w_2,\mu) :=&\pi^{\mu_1-\mu_2} \Gamma^d_\mathcal{W}(\mu_2-\mu_1,+1), \\
\label{eq:Tdw3}
	T^d(w_3,\mu) :=& \pi^{\mu_2-\mu_3} \WigDMat{d}(\vpmpm{--}w_l) \Gamma^d_\mathcal{W}(\mu_3-\mu_2,+1) \WigDMat{d}(w_l\vpmpm{--}),
\end{align}
and $\Gamma^d_\mathcal{W}(u,\varepsilon)$ is a diagonal matrix coming from the functional equation of the classical Whittaker function (I.2.20):
If $\mathcal{W}^d(y,u)$ is the diagonal matrix-valued function with entries (see section I.2.3.1)
\begin{equation}
\label{eq:classWhittDef}
\begin{aligned}
	\mathcal{W}^d_{m,m}(y,u) =& \int_{-\infty}^\infty \paren{1+x^2}^{-\frac{1+u}{2}} \paren{\frac{1+ix}{\sqrt{1+x^2}}}^{-m} \e{-yx} dx \\
	=& \piecewise{\displaystyle \frac{(\pi\abs{y})^{\frac{1+u}{2}}}{\abs{y} \Gamma\paren{\frac{1-\varepsilon m+u}{2}}} W_{-\frac{\varepsilon m}{2}, \frac{u}{2}}(4\pi\abs{y}) & \If y \ne 0, \\[5pt]
\displaystyle \frac{2^{1-u}\pi \,\Gamma(u)}{\Gamma\paren{\frac{1+u+m}{2}} \Gamma\paren{\frac{1+u-m}{2}}} & \If y=0,}
\end{aligned}
\end{equation}
(where $W_{\alpha,\beta}(y)$ is the classical Whittaker function), then for $y\ne 0$, we have the functional equations
\begin{align}
\label{eq:WhittGammas}
	\mathcal{W}^d(y,-u) =& (\pi\abs{y})^{-u} \Gamma^d_\mathcal{W}(u,\sgn(y)) \mathcal{W}^d(y,u) & \Gamma_{\mathcal{W},m,m}^d(u,\varepsilon) =& \frac{\Gamma\paren{\frac{1-\varepsilon m+u}{2}}}{\Gamma\paren{\frac{1-\varepsilon m-u}{2}}}.
\end{align}

The function $W^d(g,\mu,\psi)$, as an integral of a Wigner $\WigDName$-matrix, is again matrix-valued, and we index its rows $W^d_{m'}$, columns $W^d_{\cdot,m}$, and entries $W^d_{m',m}$ from the central entry, i.e. by the same convention as the Wigner $\WigDName$-matrices.
The matrices $T^d(w,\mu)$ are indexed by the central entry as well, and satisfy the composition
\begin{align}
\label{eq:TdCompose}
	T^d(w w',\mu) =& T^d(w,\mu)T^d(w',\mu^w).
\end{align}

\section{The main results}
\label{sect:MainThms}
Let $\mathcal{A}$ be the space of smooth, scalar-valued cusp forms, equipped with the usual $L^2$ inner product; that is, the functions $f\in L^2(\Gamma\backslash G)$ which are both infinitely differentiable and satisfy the cuspidality condition \eqref{eq:CuspCond}.
We further impose on elements $f\in\mathcal{A}$ the technical requirement that
\begin{align}
\label{eq:AbddCond}
	\exists r>0 \text{ such that }\forall X\in\frak{g}_\C \text{ we have }\sup_{g\in G} \abs{(Xf)(g)} \norm{g}^{-r} < \infty.
\end{align}
Here $\frak{g}_\C$ is the complexified Lie algebra of $G$ (see section \ref{sect:LieAlgebras}), and $\norm{g}^2 = \sum_{j,k} \abs{g_{j,k}}^2$ is the Euclidean norm resulting from the natural inclusion $G \subset \R^9$.

Now let $\mathcal{A}^d$, $d \ge 0$ be the space of smooth, vector-valued cusp forms with $K$-type $\WigDMat{d}$; that is, the functions $f:\Gamma\backslash G\to\C^{2d+1}$ taking values in the complex $(2d+1)$-dimensional row vectors, satisfying $f(gk) = f(g)\WigDMat{d}(k)$ and having finite norm under the natural inner product
\begin{align}
\label{eq:VecCuspInnerProd}
	\innerprod{f_1,f_2} = \int_{\Gamma\backslash G} f_1(g) \wbar{\trans{f_2(g)}} dg = \int_{\Gamma\backslash G/K} f_1(z) \wbar{\trans{f_2(z)}} dz.
\end{align}
We again impose on elements of $\mathcal{A}^d$ the moderate growth condition \eqref{eq:AbddCond}.

In section \ref{sect:KProjOps}, we describe the decomposition of $\mathcal{A}$ into the spaces $\mathcal{A}^d$ by the projection operators on $K$.
As discussed in the introduction, we will use three external results about these cusp forms.
\begin{thm}
\label{thm:ext} \ 

	\begin{enumerate}[ref={\thelem.\arabic*}]
		\item \label{thm:SpectralDecomp}
			The space of scalar-valued cusp forms decomposes into a direct sum of simultaneous eigenspaces $\mathcal{A}_\mu$ of the Casimir differential operators $\Delta_1$ and $\Delta_2$ (see section \ref{sect:Casimir}).
			We parameterize the eigenspaces by the eigenvalues of the corresponding power function
			\[ \Delta_i p_{\rho+\mu} = \lambda_i(\mu) p_{\rho+\mu}, \qquad \lambda_1(\mu) = 1-\tfrac{\mu_1^2+\mu_2^2+\mu_3^2}{2}, \qquad \lambda_2(\mu) = \mu_1 \mu_2 \mu_3, \]
			i.e. functions $f \in \mathcal{A}_\mu$ satisfy $\Delta_i f = \lambda_i(\mu) f$.
			Again, the eigenspaces $\mathcal{A}_\mu$ further decompose into $\mathcal{A}^d_\mu$.
			The space $\mathcal{A}^d_\mu$ is finite-dimensional.

		\item \label{thm:Unrammu}
			If $\mathcal{A}^d_\mu \ne \set{0}$ with $d\in\set{0,1}$ and $\mu$ of the form $(x+it,-x+it,-2it)$, $x,t\in\R$, then $\abs{x} < \frac{1}{2}$.
			
		\item \label{thm:WhittMultOne}
			The $n$-th Fourier coefficient of $\phi \in \mathcal{A}_\mu$ lies in the image of the Jacquet integrals.
			That is, for any $\phi \in \mathcal{A}_\mu$ which lies in an irreducible component of the right-regular representation of $G$, there exists some $f \in C^\infty(K)$ and some coefficients $c_n \in \C$, $n\in\Z^2$, so that for any $n \in \Z^2$,
			\[ \int_{U(\Z)\backslash U(\R)} \phi(ug) \wbar{\psi_n(u)} du = c_n \int_{U(\R)} I_{\mu, f}(w_l u g) \wbar{\psi_n(u)} du, \]
			with $I_{\mu, f}(xyk)=p_{\rho+\mu}(y) f(k)$.
			The integral on the right converges absolutely on $\Re(\mu_1) > \Re(\mu_2) > \Re(\mu_3)$, and must be interpretted by analytic continuation when one or more $\Re(\mu_i)=\Re(\mu_j)$, $i \ne j$.
	\end{enumerate}
\end{thm}
Part 1 is theorem 3 and lemma 18 in \cite{HC01}.
Part 2 is the unitary dual estimate noted in \cite[appendix A.1]{MillerSchmid}; of course, stronger estimates are known for $d=0$, e.g \cite[appendix 2]{KS01}, and the $d=1$ case follows by Rankin-Selberg theory (as \cite{JacShal02} does for the unramified case), and this computation is given in \cite{ArithKuzI}.
Part 3 is a well-known reformulation of Shalika's multiplicity one theorem \cite[theorem 3.1]{Shal01}; it is the combination of several deep results of representation theory, and requires a representation-theoretic proof; we give the details in appendix \ref{app:MultOne}.
In particular, the meaning of the Jacquet integral on the right-hand side in case some $\Re(\mu_i)=\Re(\mu_j)$, $i \ne j$ is made explicit there.
The coefficients $c_n$ are called the Fourier-Whittaker coefficients, though in practice they are usually scaled by a factor $\abs{n_1 n_2}$, as in (I.3.31), and the Whittaker function used is instead the completed Whittaker function at the minimal $K$-type, as in theorem \ref{thm:MinWhitt}, below.

The assumption that $\phi$ lies in an irreducible component of the right-regular representation of $G$ is not restrictive, and this, as well as the meaning of the statement, is explained in appendix \ref{app:MultOne}.

\subsection{The minimal-weight forms}
In section \ref{sect:Xjaction}, we will show the Lie algebra of $G$ gives rise to five differential-vector operators $Y^a:\mathcal{A}^d\to\mathcal{A}^{d+a}$, $a=-2,-1,0,1,2$ (see \eqref{eq:vecYDef} and \eqref{eq:vecYLeftTrans}) which are left-translation invariant.
We describe the cusp forms by their behavior under the action of these operators.
In section \ref{sect:MinKProof}, we will show
\begin{prop}
\label{prop:EquivMinimalKTypes}
	Define the differential operator
	\begin{align}
	\label{eq:LambdaXDef}
		\Lambda_x =& 27\Delta_2^2 + 4\paren{\Delta_1+x^2-1}\paren{\Delta_1+4x^2-1}^2, \qquad x \in \R.
	\end{align}
	Then for $\phi\in\mathcal{A}^d$, the following are equivalent:
	\begin{enumerate}
		\item $\phi$ is orthogonal to $Y^1 \mathcal{A}^{d-1}$, $Y^2 \mathcal{A}^{d-2}$ and $Y^0 Y^2 \mathcal{A}^{d-2}$.
		\item $\phi$ is a zero of $Y^{-1}$, $Y^{-2}$ and an eigenfunction of $Y^0$.
		\item $d=0$, $d=1$ or $\phi$ is a zero of the operator $\Lambda_{\frac{d-1}{2}}$.
	\end{enumerate}
\end{prop}
We say that such a $\phi$ is at its minimal $K$-type, or that $\phi$ is a minimal-weight form, and we define $\mathcal{A}^{d*} \subset \mathcal{A}^d$ to be the subspace of such forms.
Further define $\mathcal{A}^{d*}_\mu$ to be the subspace of simultaneous eigenfunctions of $\Delta_1$ and $\Delta_2$ with eigenvalues matching $p_{\rho+\mu}$.

The value of a description as in part 3 is that it involves only the particular $K$-type we are interested in, without reference to $\mathcal{A}^{d-1}$ or $\mathcal{A}^{d-2}$.
We initially take part 2 of the proposition as our working definition of minimal-weight forms, and the equivalence of the two remaining conditions comes at the very end of the paper in section \ref{sect:MinKProof}; in particular, after we have theorem \ref{thm:GDCusp}, below.

The appearance of $Y^0$ in the above proposition is strictly to accomodate an exceptional case that occurs at $d=2$ and begins to trouble us in section \ref{sect:GoingDown} (see the discussion following proposition \ref{prop:GDPower}).

For vectors in $\C^{2d+1}$ associated in some manner with the Wigner $\WigDName$-matrix (or the Jacquet-Whittaker function), we again index from the central entry.
For $\abs{j} \le d$, let $\bv^d_j$ be the $(2d+1)$-dimensional row vector with entries
\begin{align}
\label{eq:vdef}
	\bv^d_{j,m'} =& \delta_{m'=j} & \abs{m'} \le& d,
\end{align}
and set
\begin{align}
\label{eq:udef}
	\bu^{d,\pm}_j = \frac{1}{2}(\bv^d_j \pm (-1)^d \bv^d_{-j}).
\end{align}
In section \ref{sect:GDCuspForms}, we prove
\begin{thm}
\label{thm:GDCusp}
	Suppose $\phi \in \mathcal{A}^{d*}_\mu$.
	Then its $n$-th Fourier coefficient, $n\in\Z^2$, is a multiple of $f W^d(\cdot, \mu, \psi_n)$ where $f\in\C^{2d+1}$ and $\mu$ can be taken as one of the following:
	\begin{enumerate}
		\item For $d=0$, we use $f=1$ and $\Re(\mu)=0$ or $\mu=(x+it,-x+it,-2it)$, $\abs{x}<\frac{1}{2}$.
		\item For $d=1$, we use $f=\bu^{1,-}_0$ and $\Re(\mu)=0$ or $\mu={(x+it,-x+it,-2it)}$, $\abs{x}<\frac{1}{2}$.
		\item For $d\ge 2$, we use $f=\bu^{d,+}_d$ and $\mu=\paren{\frac{d-1}{2}+i t, -\frac{d-1}{2}+i t, -2i t}$,
	\end{enumerate}
	with $x,t$ real.
\end{thm}
The (scalar) multiple of the theorem is again the Fourier-Whittaker coefficient.
Notice that the components of $\phi=(\phi_{-d},\ldots,\phi_d)\in\mathcal{A}^{d*}_\mu$ are scalar-valued cusp forms $\phi_j\in\mathcal{A}_\mu$; then for $\phi_j$, the element of $C^\infty(K)$ whose existence is assured by theorem \ref{thm:WhittMultOne}, call it $\tilde{f}$, is precisely $\tilde{f}(k) = f \WigD{d}{\cdot}{j}(k)$ where $f$ is the vector given in the above theorem.

We note that the symmetric square of a holomorphic modular form of even weight $k$ will occur at minimal weight $d=2k-1$ with $\mu$ as in theorem \ref{thm:GDCusp} part 3 at $t=0$, c.f. \cite[section 6.4]{MeMiller}.

The Strong Selberg Eigenvalue conjecture states that cuspidal representations occuring in the spectral expansion should be tempered; in the context of theorem \ref{thm:GDCusp}, this is precisely the statement that only the case $\Re(\mu)=0$ occurs in parts 1 and 2.
The argument of section \ref{sect:MinKProof} also gives an equivalent statement of this conjecture:
\begin{conj}
	The null space of $\Lambda_x$ in the cusp forms $L^2_\text{cusp}(\Gamma\backslash G)$ is trivial for $x \in (0,\frac{1}{2})$.
\end{conj}
In fact, we expect $\Lambda_x$, $x\in(0,\frac{1}{2})$, to be a positive operator on the cusp forms.

A by-product of the analysis of section \ref{sect:GDCuspForms} is the explicit determination of all vectors $f\in\C^{2d+1}$ such that $f W^d(\cdot, \mu, \psi_n)$ is identically zero, when $\mu$ is in the standard form \eqref{eq:GDAssume}.
In studying the action of the Lie algebra, i.e. the action of the $Y^a$ operators, on the Whittaker functions of cusp forms, theorem \ref{thm:WhittMultOne} allows us to pass from Whittaker functions defined as Fourier coefficients of cusp forms to Jacquet integrals of power functions (that is, elements of the principal series representation).
In this way, and having studied the operation of the $Y^a$ operators on the power functions, the determination of precisely which power functions are killed by the Jacquet integral becomes the main obstruction to studying the action of the $Y^a$ operators on the Whittaker functions.
Once we have a solid grasp on the action of the Lie algebra on the Whittaker functions, the return to studying cusp forms is more-or-less immediate from the Fourier expansion.

In theorem \ref{thm:GDCusp}, the vectors $f=\bu^{d,\varepsilon}_{2j+\delta}$, with $\delta\in\set{0,1}$ and $\varepsilon=\pm 1$, are the rows of the matrices $\Sigma^d_\chi$ as in \eqref{eq:SigmadchiDef} where $\chi = \chi_{(-1)^\delta,\varepsilon}$ as in \eqref{eq:chipmpmdef}.
Specifically, for $\Phi\in\mathcal{A}^{d*}$, the character $\chi$ is
\begin{align}
\label{eq:chifromphi}
	\piecewise{\chipmpm{++} & \If d=0, \\ \chipmpm{+-} & \If d=1, \\ \chi_{(-1)^d,+} & \If d \ge 2.}
\end{align}
We have essentially forced a choice of the character $\chi$ in theorem \ref{thm:GDCusp} to make the vector $f$ as nice as possible, and this is done by applying the functional equations of the Jacquet-Whittaker function as in proposition \ref{prop:GDWhittFEs} (which permutes both the character and the coordinates of $\mu$, see (I.3.27)); other allowable choices are $\chipmpm{--}$ or $\chipmpm{-+}$ when $d=1$ and $\chi_{(-1)^d,-}$ or $\chi_{\pm,(-1)^d}$ for $d \ge 2$.

In section \ref{sect:MinWhitt}, we compute the Mellin-Barnes integrals for the Whittaker functions of $GL(3)$ automorphic forms at their minimal $K$-types.
These have been computed by Miyazaki \cite{MiyaWhitt} for $d \ge 2$, Manabe, Ishii and Oda \cite{ManIshOda} for $d=1$ and Bump \cite{Bump01} for $d=0$.
In case $d\ge 1$, the bases used are somewhat different than ours, and for $d \ge 2$, the results here are somewhat stronger than Miyazaki's because we have solidified the connection between the Jacquet-Whittaker function and the Mellin-Barnes integral.
(Theorem 5.9 of \cite{MiyaWhitt} contains the phrase ``there is an element [of the Whittaker model]''; we have very precisely answered the question of which vector in the principal series representation gives rise to that element.)

For $\alpha\in\set{0,1}^3$, $\beta,\eta\in\Z^3$, $\ell\in\Z^2$ and $s\in\C^2$ define
\begin{align}
	\Lambda_\alpha(\mu) =& \pi^{-\frac{3}{2}+\mu_3-\mu_1} \Gamma\paren{\tfrac{1+\alpha_1+\mu_1-\mu_2}{2}} \Gamma\paren{\tfrac{1+\alpha_2+\mu_1-\mu_3}{2}} \Gamma\paren{\tfrac{1+\alpha_3+\mu_2-\mu_3}{2}},
\end{align}
\begin{align}
	\wtilde{G}(d,\beta,\eta, s,\mu)= \frac{\prod_{i=1}^3 \Gamma\paren{\frac{\beta_i+s_1-\mu_i}{2}} \Gamma\paren{\frac{\eta_i+s_2+\mu_i}{2}}}{\Gamma\paren{\frac{s_1+s_2+\sum_i (\beta_i+\eta_i)-2d}{2}}},
\end{align}
\begin{align}
	\wtilde{G}^0(\ell, s,\mu)= \wtilde{G}(0,0,0,s,\mu), \qquad \wtilde{G}^1(\ell,s,\mu)= \wtilde{G}(1,(\ell_1,\ell_1,1-\ell_1),(\ell_2,\ell_2,1-\ell_2),s,\mu),
\end{align}
and for $d \ge 2$,
\begin{align}
\label{eq:MinWhittGPSRLam}
	\Lambda^*(\mu) =& (-1)^d \pi^{-\frac{3}{2}+\mu_3-\mu_1} \Gamma(d) \Gamma\paren{\tfrac{1+\mu_1-\mu_3}{2}} \Gamma\paren{\tfrac{2+\mu_1-\mu_3}{2}}, \\
\label{eq:MinWhittGPSR}
	\wtilde{G}^d(\ell,s,\mu) =& \wtilde{G}(d,(d,0,\ell_1),(0,d,\ell_2),s,\mu).
\end{align}
Now for $\abs{m'}\le d$, write $m'=\varepsilon m$ with $\varepsilon = \pm 1$ and $0 \le m \le d$, and set
\begin{align}
	G^d_{m'}(s,\mu) =& \sqrt{\binom{2d}{d+m}} \sum_{\ell=0}^{m} \varepsilon^\ell \binom{m}{\ell} \wtilde{G}^d((d-m,\ell), s,\mu),
\end{align}
and take $G^d(s,\mu)$ to be the vector with coordinates $G^d_{m'}(s,\mu)$, $m'=-d,\ldots,d$.

We define the completed minimal-weight Whittaker function at each weight $d$ as
\begin{align}
	W^{d*}(y, \mu) =& \frac{1}{4\pi^2} \int_{\Re(s)=\mathfrak{s}} (\pi y_1)^{1-s_1} (\pi y_2)^{1-s_2} G^d\paren{s,\mu} \frac{ds}{(2\pi i)^2},
\end{align}
for any $\mathfrak{s} \in (\R^+)^2$.

\begin{thm}
\label{thm:MinWhitt}
	The Whittaker functions at the minimal $K$-types are
	\begin{align*}
		\Lambda_{(0,0,0)}(\mu) W^0(y, \mu, \psi_{1,1}) =& W^{0*}(y, \mu),
	\end{align*}
	\begin{align*}
		\sqrt{2}\Lambda_{(0,1,1)}(\mu) \bu^{1,-}_0 W^1(y, \mu, \psi_{1,1}) =& W^{1*}(y, \mu), \\
		-2 \Lambda_{(1,0,1)}(\mu) \bu^{1,-}_1 W^1(y, \mu, \psi_{1,1}) =& W^{1*}\paren{y, \mu^{w_4}}, \\
		2 \Lambda_{(1,1,0)}(\mu) \bu^{1,+}_1 W^1(y, \mu, \psi_{1,1}) =& W^{1*}\paren{y, \mu^{w_5}},
	\end{align*}
	and for $\mu=(\frac{d-1}{2}+it,-\frac{d-1}{2}+it,-2it)$ with $d \ge 2$,
	\begin{align*}
		\Lambda^*(\mu) W^d_{-d}(y, \mu, \psi_{1,1}) =& W^{d*}(y, \mu), \qquad W^d_{d,m} = 0.
	\end{align*}
\end{thm}

As mentioned in the introduction, this is the main vehicle for the analytic study of $GL(3)$ automorphic forms.
It has applications in the functional equations of $GL(3)$ $L$-functions (see e.g. section 6.5, especially lemma 6.5.21, of \cite{Gold01} or \cite{HIM01}) and their Rankin-Selberg convolutions (see e.g. \cite{Stade02, HIM02, WeylI, WeylII}), and is fundamental to the creation of Kuznetsov-type trace formulas (see \cite{SpectralKuz,WeylI,WeylII}), among other applications.

\subsection{The full spectral expansion}
If $\Phi \in \mathcal{A}^{d_0*}_\mu$, $d_0 \ge 0$ has Fourier coefficients
\[ \rho_\Phi(n) f W^{d_0}(\cdot, \mu, \psi_n) \]
with the parameters $\mu$ and $f$ given by theorem \ref{thm:GDCusp} and $\chi$ is given by \eqref{eq:chifromphi}, then for all $d$ (not just $d \ge d_0$), we may construct two matrix-valued forms
\begin{align}
\label{eq:PhiFourierWhitt}
	\Phi^d(g) =& \sum_{\gamma\in (U(\Z)V)\backslash SL(2,\Z)} \sum_{v\in V} \sum_{n \in \N^2} \rho_\Phi(n) \Sigma^d_\chi W^d(\gamma v g, \mu, \psi_n),
\end{align}
and $\wtilde{\Phi}^d(g)$ defined similarly but using the spectral parameters $-\wbar{\mu}$ instead.
Note that this is not the dual form.
We will insist on a slightly unusual normalization \eqref{eq:PhiNormalization} of $\Phi$.

In section \ref{sect:GoingUp}, we look at the spaces of vector-valued forms generated by applying the $Y^a$ operators to the minimal-weight cusp forms.
More precisely, we operate at the level of the coefficient vectors.
That is, if $v$ lies in an appropriate subspace of $\C^{2d+1}$ (the rowspace of $\Sigma^d_\chi$), then the entries of $p_{\rho+\mu}(y) v \WigDMat{d}(k)$ lie in the principal series representation and the entries of $v W^d(g,\mu,\psi_n)$ lie in the Whittaker model, and we will show the entries of $v \Phi^d(g)$ (or more precisely $v \wtilde{\Phi}^d(g)$) are cusp forms.
Since the $Y^a$ operators function identically on all three objects by left-translation invariance, we are free to study their behavior on objects of the first type.

Now let $\mathcal{S}_{3,\mu}^d$ be a basis of $\mathcal{A}_\mu^{d*}$ for each $d \ge 0$, and take $\mathcal{S}_3^d$ to be the union of $\mathcal{S}_{3,\mu}^d$ for all $\mu$ and $\mathcal{S}_3 = \bigcup_d \mathcal{S}_3^d$.
As described above, for each $\Phi\in\mathcal{S}_3$ and every $d$, we may construct the two matrix-valued forms $\Phi^d$ and $\wtilde{\Phi}^d$.
It is the case that the function $\wtilde{\Phi}^d$ will be identically zero when $d$ is less than the minimal weight of $\Phi$, and in section \ref{sect:CuspStructure}, we pull everything together into the following very uniform expansion of cusp forms:
\begin{thm}
\label{thm:CuspSpectralExpand}
	For $f\in\mathcal{A}$, we have
	\begin{align*}
		f(g) = \sum_{\Phi\in\mathcal{S}_3} \sum_{d=0}^\infty (2d+1) \Tr\paren{\Phi^d(g)\int_{\Gamma\backslash G} f(g')\wbar{\trans{\wtilde{\Phi}^d(g')}}dg'}.
	\end{align*}
\end{thm}
The equivalent statement for a vector-valued form $f\in\mathcal{A}^d$ is
\begin{align}
\label{eq:VectCuspSpectralExpand}
	f(g) = \sum_{\Phi\in\mathcal{S}_3} \int_{\Gamma\backslash G} f(g')\wbar{\trans{\wtilde{\Phi}^d(g')}}dg' \, \Phi^d(g).
\end{align}
This completes the spectral expansion.

Each vector-valued cusp form of $K$-type $\WigDMat{d}$ corresponds to $2d+1$ scalar-valued forms, and the multiplicities of a given $\Phi\in\mathcal{S}_3^{d_0*}$ in the vector-valued forms are essentially the rank of the matrices $\Sigma^d_\chi$, with the middle $2d_0-1$ rows removed if $d_0 \ge 2$.
By the types listed in theorem \ref{thm:GDCusp}, these multiplicities are
\begin{enumerate}
\item \(\displaystyle \piecewise{\frac{d}{2}+1 & d \text{ even},\\ \frac{d-1}{2} & d \text{ odd,}} \)
\item \(\displaystyle \floor{\tfrac{d+1}{2}} \)
\item \(\displaystyle \floor{\tfrac{d-d_0}{2}}+1 \qquad d\ge d_0. \)
\end{enumerate}

\section{Acknowledgements}
The author would like to thank Stephen D. Miller for his insight into the non-spherical Maass forms; in particular, Prof. Miller and the author independently and near-simultaneously arrived at forms of equation \eqref{eq:XjOp}.
The author would also like to thank James Cogdell and Joseph Hundley for their comments and helpful discussions along the way, and the referee for a careful reading of the paper.

\section{Background}
\subsection{The Lie algebras}
\label{sect:LieAlgebras}
The complexified Lie algebra of $G=PSL(3,\R)$ is $\frak{g}_\C=\frak{sl}(3)\otimes_\R \C$, the space of complex matrices of trace zero, and the complexified Lie algebra of $K=SO(3,\R)$ is $\frak{k}_\C$, the space of anti-symmetric matrices.
These matrices act as differential operators on smooth functions of $G$ by
\begin{align}
\label{eq:LieActDef}
	(Xf)(g) := \left. \frac{d}{dt} f(g\exp tX) \right|_{t=0}, \qquad (X+iY)f = Xf+iYf,
\end{align}
when the entries of $X$ and $Y$ are real.
We will not differentiate notationally between a matrix in the Lie algebra and the associated differential operator.
Note that the commutator of the differential operators associated to two such matrices is then given by the differential operator associated to the commutator of the matrices; that is,
\[ [X,Y] := X\circ Y-Y\circ X = XY-YX. \]

We take as our basis of $\frak{k}_\C$ the three matrices
\begin{align}
\label{eq:Kmats}
	K_{\pm 1} =&\Matrix{0&0&\mp 1\\0&0&-i\\ \pm 1&i&0}, & K_0 =& \sqrt{2}\Matrix{0&1&0\\-1&0&0\\0&0&0}.
\end{align}
In terms of the Z-Y-Z coordinates of section I.2.2, the associated differential operators acting on a function of $k(\alpha,\beta,\gamma)$, are
\begin{align}
\label{eq:Kop}
	K_{\pm 1} =& e^{\mp i\gamma}\paren{i \cot\beta \, \partial_\gamma \mp \partial_\beta-i\csc\beta \, \partial_\alpha}, & K_0 =& -\sqrt{2}\partial_\gamma.
\end{align}
On an entry of the Wigner-$\WigDName$ matrix, the $K_{\pm 1}$ operators increase or decrease the column \cite[section 3.8]{BiedLouck},
\begin{align}
\label{eq:K1actWD}
	K_{\pm 1} \WigD{d}{m'}{m} = \sqrt{d(d+1)-m(m\pm1)} \WigD{d}{m'}{m\pm 1}, \qquad \abs{m'},\abs{m} \le d,
\end{align}
and $K_0$ does not,
\begin{align}
\label{eq:K0actWD}
	K_0 \WigD{d}{m'}{m} = \sqrt{2}i m \WigD{d}{m'}{m}, \qquad \abs{m'},\abs{m} \le d.
\end{align}
The Laplacian on $K$ is given by
\begin{align}
	2\Delta_K =K_1 \circ K_{-1} +K_{-1} \circ K_1 - K_0 \circ K_0,
\end{align}
and the Wigner-$\WigDName$ matrix satisfies
\begin{align}
	\Delta_K \WigDMat{d} = d(d+1) \WigDMat{d}.
\end{align}

We extend to a basis for $\frak{g}_\C$ by adding the five matrices
\begin{align}
\label{eq:Xmats}
	X_{\pm 2} =& \Matrix{1&\pm i&0\\ \pm i&-1&0\\0&0&0}, & X_{\pm 1} =& -\Matrix{0&0&\pm 1\\0&0&i\\ \pm 1&i&0}, & X_0 =& -\frac{\sqrt{6}}{3}\Matrix{1&0&0\\0&1&0\\0&0&-2}.
\end{align}
Though the details of their construction are not used in this paper, these matrices are deliberately constructed so that
\begin{align}
\label{eq:XkalphaComm}
	k(\alpha,0,0) X_j k(-\alpha,0,0) = e^{-ij\alpha} X_j, \qquad \abs{j}\le 2,
\end{align}
and they are orthogonal to $\frak{k}_\C$ and have identical norm under the Killing form (up to a constant, the natural inner-product resulting from the inclusion $\frak{g}_\C \subset \C^9$):
\begin{equation}
\label{eq:XKKilling}
\begin{aligned}
	\sum_{i,j} [X_{k_1}]_{i,j} \wbar{[X_{k_2}]_{i,j}} =& 4 \delta_{k_1=k_2}, & \sum_{i,j} [X_{k_1}]_{i,j} \wbar{[K_{k_2}]_{i,j}} =& 0, \\
	\sum_{i,j} [K_{k_1}]_{i,j} \wbar{[K_{k_2}]_{i,j}} =& 4 \delta_{k_1=k_2},
\end{aligned}
\end{equation}
where $[X_k]_{i,j}$ means the entry at index $i,j$ of the matrix $X_k$ as in \eqref{eq:Xmats}, and similarly for $[K_k]_{i,j}$.
Such a choice makes the description \eqref{eq:XjOp} relatively nice.

For use in sections \ref{sect:XjOps} and \ref{sect:Xjaction}, we introduce the right $K$-invariant operators (which act on the left):
For smooth functions of $K$, define
\[ (\KLeft{j} f)(k) = \left.\frac{d}{dt} f(\exp(tK_j) k)\right|_{t=0}, \]
then we have \cite[section 3.8]{BiedLouck}
\begin{align}
\label{eq:KLeft}
	\KLeft{0} = -\sqrt{2}\partial_\alpha, \qquad \KLeft{\pm 1} = e^{\pm i\alpha}\paren{i\csc\beta\partial_\gamma-i\cot\beta\partial_\alpha\mp\partial_\beta}.
\end{align}
These perform the symmetric operation to the $K_j$ on an entries of the Wigner-$\WigDName$ matrix:
\begin{align}
\label{eq:KLeftactWD}
	\KLeft{\pm 1} \WigD{d}{m'}{m} = \sqrt{d(d+1)-m'(m'\mp1)} \WigD{d}{m'\mp 1}{m}, \; \KLeft{0} \WigD{d}{m'}{m} = \sqrt{2}i m' \WigD{d}{m'}{m}, \; \abs{m'},\abs{m} \le d.
\end{align}
And the $K$ Laplacian may also be written as
\begin{align}
\label{eq:DeltaKLeft}
	2\Delta_K = \KLeft{1} \circ \KLeft{-1}+\KLeft{-1} \circ \KLeft{1} - \KLeft{0} \circ \KLeft{0}.
\end{align}
We extend these operators to smooth functions on $G$ by the expressions \eqref{eq:KLeft}.

\subsection{The Casimir Operators}
\label{sect:Casimir}
The (normalized) Casimir operators are defined by
\begin{align}
\label{eq:CasimirDef}
	\Delta_1 = -\frac{1}{2}\sum_{i,j} E_{i,j} \circ E_{j,i}, \qquad \Delta_2 = \frac{1}{3}\sum_{i,j,k} E_{i,j} \circ E_{j,k}\circ E_{k,i}+\Delta_1,
\end{align}
where $E_{i,j}$, $i,j\in\set{1,2,3}$ is the element of the Lie algebra of $GL(3,\R)$ whose matrix has a 1 at position $i,j$ and zeros elsewhere, using the standard indexing.
In the Lie algebra, these are sometimes called the Capelli elements, and we require their expressions as differential operators.
For operators of this complexity, this is typically done on a computer, but in appendix \ref{app:CasCompute}, we give sufficient details that the computation could (but likely shouldn't) be carried out by hand.

On functions of $G/K$ these become \cite[eqs (6.1.1)]{Gold01}
\begin{align}
	\Delta_1^\circ &= -y_1^2\partial_{y_1}^2-y_2^2\partial_{y_2}^2+y_1y_2\partial_{y_1}\partial_{y_2}-y_1^2(x_2^2+y_2^2)\partial_{x_3}^2 -y_1^2\partial_{x_1}^2-y_2^2\partial_{x_2}^2-2y_1^2x_2\partial_{x_1} \partial_{x_3}, \\
	\Delta_2^\circ =& -y_1^2y_2\partial_{y_1}^2 \partial_{y_2}+y_1y_2^2\partial_{y_1} \partial_{y_2}^2-y_1^3y_2\partial_{x_3}^2\partial_{y_1}+y_1y_2^2\partial_{x_2}^2 \partial_{y_1}-2y_1^2y_2x_2\partial_{x_1}\partial_{x_3}\partial_{y_2} \\
	& +(-x_2^2+y_2^2)y_1^2y_2\partial_{x_3}^2\partial_{y_2}-y_1^2y_2\partial_{x_1}^2 \partial_{y_2}+2y_1^2y_2^2\partial_{x_1}\partial_{x_2}\partial_{x_3} +2y_1^2y_2^2x_2\partial_{x_2} \partial_{x_3}^2 \nonumber \\
	& +y_1^2\partial_{y_1}^2-y_2^2\partial_{y_2}^2+2y_1^2x_2\partial_{x_1}\partial_{x_3} +(x_2^2+y_2^2)y_1^2\partial_{x_3}^2+y_1^2\partial_{x_1}^2-y_2^2\partial_{x_2}^2. \nonumber
\end{align}
We will require expressions for the full operators, i.e. those acting on functions of $G$.

We define the operators
\begin{equation}
\label{eq:ZDef}
\begin{aligned}
	Z_{\pm 2} =& \paren{2y_2 \partial_{y_2}-y_1 \partial_{y_1}} \pm 2 i y_2 \partial_{x_2}, \\
	Z_{\pm 1} =& -2i y_1(\partial_{x_1}+x_2 \partial_{x_3}) \mp 2 y_1 y_2 \partial_{x_3}, \\
	Z_0 =& -\sqrt{6} y_1 \partial_{y_1},
\end{aligned}
\end{equation}
then the full Casimir operators on $G$ are given by the following lemma.
\begin{lem}
\label{lem:CasCompute}
Acting on functions of the Iwasawa coordinates $x_1, x_2, x_3, y_1, y_2, \alpha, \beta, \gamma$, the Casimir operators have the form
\begin{align}
\label{eq:FullDelta1}
	8\Delta_1 =& 8\Delta_1^\circ-2\KLeft{1} \circ Z_{-1} - \sqrt{2}i \KLeft{0} \circ \paren{Z_2-Z_{-2}}-2\KLeft{-1} \circ Z_1, \\
\label{eq:FullDelta2}
	96\Delta_2 =& 96\Delta_2^\circ +2\KLeft{1} \circ \mathcal{T}_1+\sqrt{2}i\KLeft{0}\circ \mathcal{T}_0+2\KLeft{-1} \circ \mathcal{T}_{-1} \\
	& \qquad +6\sqrt{2} i \paren{\KLeft{1}+\KLeft{-1}} \circ K_0 \circ \paren{Z_{-1}-Z_1}, \nonumber
\end{align}
where
\begin{align*}
	\mathcal{T}_{\pm 1}=& \sqrt{6}Z_{\mp1}\circ Z_0+6\paren{1-Z_{\mp2}} \circ Z_{\pm1}, \\
	\mathcal{T}_0=& 3\paren{Z_1\circ Z_1-Z_{-1}\circ Z_{-1}}+2\paren{\sqrt{6} Z_0+6} \circ \paren{Z_{-2}-Z_2}.
\end{align*}
\end{lem}
Given a computer algebra package (and a sufficiently fast computer), one can compute these expressions by computing the differential operator for each $E_{i,j}$ acting on
\[ f=f(x_1,x_2,x_3,y_1,y_2,\alpha,\beta,\gamma) \]
from the Iwasawa decomposition (as in section I.2.4) for
\[ xyk\exp(t E_{i,j}) \approx xyk(I+t E_{i,j}), \qquad \text{($t$ small)} \]
(one can simplify this process using \eqref{eq:LieAlgTrick} or \eqref{eq:EijToXK}), then forming $\Delta_i f$ directly from \eqref{eq:CasimirDef} and having the computer algebra package collect terms according to the derivatives of $f$.
We give a somewhat more explicit proof in appendix \ref{app:CasCompute}.

\subsection{The projection operators}
\label{sect:KProjOps}
Let $\wtilde{C}(G,\WigDMat{d})$ be the space of smooth functions on $G$ that lie in the span of the entries of $\WigDMat{d}$.
That is $f \in \wtilde{C}(G,\WigDMat{d})$ if and only if
\[ f(xyk) = \sum_{\abs{m'},\abs{m} \le d} f_{m',m}(xy) \WigD{d}{m'}{m}(k), \]
for some collection of functions $f_{m',m}:G\to\C$.
For $f \in \wtilde{C}(G,\WigDMat{d})$, taking the expansion (I.2.10) of $f_g(k) := f(gk)$ into Wigner $\WigDName$-functions and evaluating at the identity $k=I$ gives
\[ f(g) = (2d+1) \sum_{\abs{m} \le d} \int_K f(g k') \wbar{\WigD{d}{m}{m}(k')} dk' = \sum_{\abs{m} \le d} (\tildeWigP{d}{m} f)(g), \]
where
\begin{align}
\label{eq:tildePDef}
	(\tildeWigP{d}{m} f)(g) = (2d+1) \int_K f(gk) \wbar{\WigD{d}{m}{m}(k)} dk
\end{align}
is the projection onto the span of the $m$-th column of $\WigDMat{d}$.

The projection operators may also be written as
\[ \tildeWigP{d}{m} f(xyk) = (2d+1) \sum_{\abs{m'} \le d} \int_K f(xyk') \wbar{\WigD{d}{m'}{m}(k')} dk' \, \WigD{d}{m'}{m}(k). \]
Also let
\[ \tildeWigPd{d} = \sum_{\abs{m} \le d} \tildeWigP{d}{m}. \]
Note that the projection operators \eqref{eq:tildePDef} commute with the differential operators $\Delta_1$ and $\Delta_2$ by translation invariance, but the projection onto the span of a single Wigner function $\WigD{d}{m'}{m}$ might not.

Now if $f \in \wtilde{C}(G,\WigDMat{d})$, we may write $f$ in terms of row vector-valued functions which transform by $\WigDMat{d}$,
\[ f(xyk) =  \sum_{\abs{\ell} \le d} f_{\ell,\ell}(xyk), \qquad f_\ell(xyk) := (2d+1) \int_K f(xyk k') \WigDRow{d}{\ell}\paren{{k'}^{-1}} dk' = f_\ell(xy) \WigDMat{d}(k) \]
(here $f_{\ell,\ell}$ is the entry at index $\ell$ of the vector-valued function $f_\ell$), so we may move to considering vector-valued functions, say $C(G,\WigDMat{d})$ is the space of smooth vector-valued functions $f(xyk) = f(xy) \WigDMat{d}(k)$.
Define the vector projection operators $\WigP{d}{m}:\wtilde{C}(G,\WigDMat{d}) \to C(G,\WigDMat{d})$ by
\[ (\WigP{d}{m} f)(g) = (2d+1) \int_K f(g k') \WigDRow{d}{m}\paren{{k'}^{-1}} dk'. \]

Note that we have an isomorphism $\wtilde{C}(G,\WigDMat{d}) \cong C(G,\WigDMat{d})^{2d+1}$ as each scalar-valued function gives rise to one vector-valued function for each row of $\WigDMat{d}$.
These functions are inherently linear combinations of the rows of $\WigDMat{d}(k)$,
\[ (f_{-d}(xyk), \ldots, f_d(xyk)) = \sum_{\abs{m'} \le d} f_{m'}(xy) \WigDRow{d}{m'}(k), \]
and the entries may be written as
\[ f_{m}(xyk) = \sum_{\abs{m'} \le d} f_{m'}(xy) \WigD{d}{m'}{m}(k). \]
The raising and lowering operators on $K$ then give
\begin{align}
\label{eq:Kpm1f}
	K_{\pm 1} f_{m}(xyk) = \sqrt{d(d+1)-m(m\pm1)} f_{m\pm 1}(xyk),
\end{align}
using $f_m=0$ when $\abs{m} > d$.
The operators $\tildeWigP{d}{\ell}$ and $K_{\pm 1}$ have simple commutation relations,
\begin{align}
	\tildeWigP{d}{\ell} K_{\pm 1} = K_{\pm 1} \tildeWigP{d}{\ell\mp 1},
\end{align}
and this gives
\begin{align}
\label{eq:KPComm}
	\WigP{d}{\ell} K_{\pm 1} = \sqrt{d(d+1)-\ell(\ell\pm1)} \WigP{d}{\ell\mp 1}.
\end{align}

\subsection{Clebsch-Gordan coefficents}
\label{sect:CG}
It is a fact from the theory of the Wigner-$\WigDName$ matrices that any product $\WigD{a}{j}{\ell} \WigD{d}{m'}{m}$ lies in the span of $\set{\WigD{d+b}{m'+j}{m+\ell}\setdiv\abs{b}\le a}$.
More precisely, we have \cite[eq. (3.189)]{BiedLouck},
\begin{align}
\label{eq:CGExpand}
	\WigD{k}{j}{i} \WigD{d}{m'}{m} = \sum_{\abs{a} \le k} \CG{d,k,a}{m',j} \CG{d,k,a}{m,i} \WigD{d+a}{m'+j}{m+i},
\end{align}
using the Clebsch-Gordan coefficients,
\[ \CG{d,k,a}{m,i} = \left<k \, i \, d \, m \middle| (d+a) \, (i+m) \right>, \]
(we prefer the shorthand $\CG{d,k,a}{m,i}$ over the standard notation $\left<\ldots\middle|\ldots\right>$ in the interests of compactness) or the Wigner three-$j$ symbols via
\begin{align}
\label{eq:CGtoWig3j}
	\Matrix{j_1&j_2&j_3\\m_1&m_2&-m_3} = \frac{(-1)^{j_1-j_2+m_3}}{\sqrt{2j_3+1}} \left<j_1 \, m_1 \, j_2 \, m_2 \middle| j_3 \, m_3 \right>.
\end{align}
The coefficients are defined to be zero unless
\begin{align}
	\abs{m} \le d, \abs{i} \le k, \abs{m+i}\le d+a, \abs{d-k} \le d+a \le d+k.
\end{align}

One particular symmetry relation is \cite[eq. 34.3.10]{DLMF},
\begin{align}
\label{eq:CGsymm}
	\CG{d,k,a}{m,i} = (-1)^{k-a} \CG{d,k,a}{-m,-i}.
\end{align}

Because we only use $k=2$ and $k=1$, these may be found in Abramowitz and Stegun \cite[tables 27.9.2 \& 4]{AbramStegun}, and we list the relevant values here:
Set
\[ \frak{C}^{d,k,a}_{m,i} = \frac{\sqrt{(2d+2a+1)\,(2d+a-k)!}}{\sqrt{(k-i)!\,(k+i)!\,(2d+a+k+1)!}} \times \piecewise{\frac{(-1)^i\sqrt{(d-m)!\,(d+m)!}}{\sqrt{(d+a-i-m)!\,(d+a+i+m)!}} & \If a \le 0, \\[10pt] \frac{\sqrt{(d+a-i-m)!\,(d+a+i+m)!}}{\sqrt{(d-m)!\,(d+m)!}} & \Otherwise,} \]
then
\begin{align}
\label{eq:CG2pm2}
	\CG{d,2,-2}{m,i} =& 2\sqrt{6} \frak{C}^{d,2,-2}_{m,i}, & \CG{d,2, 2}{m,i} =& 2\sqrt{6} \frak{C}^{d,2,2}_{m,i}, \\
\label{eq:CG2pm1}
	\CG{d,2,-1}{m,i} =& 2\sqrt{6} (i(d+1)+2m) \frak{C}^{d,2,-1}_{m,i}, & \CG{d,2, 1}{m,i} =& 2\sqrt{6} (di-2m) \frak{C}^{d,2,1}_{m,i},
\end{align}
\begin{align}
\label{eq:CG20}
	\CG{d,2, 0}{m,i} =& 2(2 d^2 (i^2-1)+d (5 i^2+6 i m-2)+3(i + m) (i + 2 m)) \frak{C}^{d,2,0}_{m,i},
\end{align}
\begin{align}
\label{eq:CG1}
	\CG{d,1,-1}{m,i} =& -\sqrt{2} \frak{C}^{d,1,-1}_{m,i}, & \CG{d,1,0}{m,i} =& -2(i(d+1)+m) \frak{C}^{d,1,1}_{m,i}, & \CG{d,1,1}{m,i} =& \sqrt{2} \frak{C}^{d,1,1}_{m,i}.
\end{align}

\subsection{The commutation relations}
\label{sect:CommRels}
We record the commutation relations amongst the differential operators.
This is most easily expressed in terms of the Clebsch-Gordan coefficients:
\begin{align}
	[X_j,X_k] =& 2\sqrt{5}i^{1-j-k}\CG{2,2,-1}{j,k} K_{j+k}, & [X_j,K_k] =&\sqrt{12}i^{1+k}\CG{2,1,0}{j,k} X_{j+k}, \\
	[K_j,K_k] =& 2i\CG{1,1,0}{j,k} K_{j+k},
\end{align}
and the Clebsch-Gordon coefficients, as matrices $\CG{d,k,a}{}$, are
\begin{align*}
	\CG{2,2,-1}{} =& \frac{1}{10}\Matrix{0&0&0&2\sqrt{5}&2\sqrt{10}\\0&0&-\sqrt{30}&-\sqrt{10}&2\sqrt{5}\\0&\sqrt{30}&0&-\sqrt{30}&0\\-2\sqrt{5}&\sqrt{10}&\sqrt{30}&0&0\\-2\sqrt{10}&-2\sqrt{5}&0&0&0} &
	\CG{2,1,0}{} =& \frac{1}{6} \Matrix{0&2\sqrt{6}&2\sqrt{3}\\-2\sqrt{3}&\sqrt{6}&3\sqrt{2}\\-3\sqrt{2}&0&3\sqrt{2}\\-3\sqrt{2}&-\sqrt{6}&2\sqrt{3}\\-2\sqrt{3}&-2\sqrt{6}&0} \\
	\CG{1,1,0}{} =& \frac{1}{2}\Matrix{0&\sqrt{2}&\sqrt{2}\\-\sqrt{2}&0&\sqrt{2}\\-\sqrt{2}&-\sqrt{2}&0}
\end{align*}

\subsection{Barnes integrals}
We will make use of Barnes' Second Lemma \cite[section 6]{Barnes02}:
For $a,b,c,d,e\in\C$,
\begin{align}
\label{eq:BarnesSecond}
	& \int_{-i\infty}^{i\infty} \frac{\Gamma(a+s)\Gamma(b+s)\Gamma(c+s)\Gamma(d-s)\Gamma(e-s)}{\Gamma(f+s)} \frac{ds}{2\pi i} \\
	&\qquad = \frac{\Gamma(a+e)\Gamma(b+e)\Gamma(c+e)\Gamma(a+d)\Gamma(b+d)\Gamma(c+d)}{\Gamma(f-a)\Gamma(f-b)\Gamma(f-c)}, \nonumber
\end{align}
where $f=a+b+c+d+e$.

\subsection{The Weyl group and the $V$ group}
\label{sect:WeylV}
In section \ref{sect:GoingDown}, we will make extensive use of the Weyl and $V$ group elements and their images under the Wigner $\WigDName$-matrices directly, and in preparation, we give a few identities that will smooth out the analysis there.
First, we note that the Weyl group is generated by the transpositions $w_2$ and $w_3$:
\begin{align}
	w_4=w_3w_2, \qquad w_5=w_2w_3, \qquad w_l=w_2w_3w_2=w_3w_2w_3.
\end{align}
Next, we investigate the commutation relations between $W$ and $V$:
\begin{equation}
\label{eq:wvwinv}
\begin{aligned}
	w_2 v_{\varepsilon_1,\varepsilon_2} w_2 &= v_{\varepsilon_1 \varepsilon_2,\varepsilon_2}, & w_3 v_{\varepsilon_1,\varepsilon_2} w_3 =& v_{\varepsilon_1, \varepsilon_1 \varepsilon_2}, & w_4 v_{\varepsilon_1,\varepsilon_2} w_5 &= v_{\varepsilon_1 \varepsilon_2,\varepsilon_1}, \\
	w_5 v_{\varepsilon_1,\varepsilon_2} w_4 &= v_{\varepsilon_2,\varepsilon_1 \varepsilon_2}, & w_l v_{\varepsilon_1,\varepsilon_2} w_l =& v_{\varepsilon_2, \varepsilon_1}.
\end{aligned}
\end{equation}

Now we need to see how the matrices $\WigDMat{d}(w_2)$ and $\WigDMat{d}(v)$, $v\in V$ interact with the rows of the $\Sigma^d_\chi$ matrices, which are the vectors $\bu^{d,\pm}_m$ as in \eqref{eq:udef}.
From \eqref{eq:WigDv}, we can see that
\begin{align}
\label{eq:buDv}
	\bu^{d,+}_m \WigDMat{d}(v_{\varepsilon_1,\varepsilon_2}) =& \varepsilon_1^j \bu^{d,+}_m, & \bu^{d,-}_m \WigDMat{d}(v_{\varepsilon_1,\varepsilon_2}) =& \varepsilon_1^j \varepsilon_2 \bu^{d,-}_m,
\end{align}
or equivalently,
\begin{align}
\label{eq:buDv2}
	\bu^{d,\varepsilon}_m \WigDMat{d}(\vpmpm{-+}) =& (-1)^j \bu^{d,\varepsilon}_m, & \bu^{d,\varepsilon}_m \WigDMat{d}(\vpmpm{+-}) =& \varepsilon \bu^{d,\varepsilon}_m.
\end{align}
From sections I.2.2.2 and I.2.2.3, we can see that
\begin{align}
\label{eq:w2toRi}
	\WigDMat{d}(w_2) =& \WigDMat{d}(\vpmpm{+-})\Dtildek{d}{i} = \Dtildek{d}{i} \WigDMat{d}(\vpmpm{--}),
\end{align}
and it follows that if $\varepsilon=(-1)^m$, then
\begin{align}
\label{eq:buSignReduct}
	\bu^{d,\pm}_m \Dtildek{d}{i} = i^{-m} \bu^{d,\pm\varepsilon}_m, \qquad \bu^{d,\pm}_m \WigDMat{d}(w_2) = \bu^{d,\pm}_m \Dtildek{d}{i} \WigDMat{d}(\vpmpm{--}) = \pm i^{-m} \bu^{d,\pm\varepsilon}_m.
\end{align}

\subsection{Wigner $\WigdName$-polynomials and Jacobi $P$-polynomials}
\label{sect:WignerdJacobiP}
The Wigner $\WigdName$-polynomials
\begin{align}
\label{eq:WigdDef}
	\Wigd{d}{m'}{m}(\cos\beta)=\WigD{d}{m'}{m}(k(0,\beta,0))
\end{align}
may be given in terms of Jacobi polynomials \cite[eq. (3.72)]{BiedLouck}
\begin{align}
\label{eq:WigdtoJacobi}
	\Wigd{d}{m'}{m}(x) = 2^{-m} \sqrt{\frac{(d+m)! \, (d-m)!}{(d+m')! \, (d-m')!}} \paren{1-x}^{\frac{m-m'}{2}} \paren{1+x}^{\frac{m+m'}{2}} \JacobiP{d-m}{m-m'}{m+m'}(x),
\end{align}
and they satisfy the symmetries \cite[eqs (3.80)-(3.82)]{BiedLouck}
\begin{align}
\label{eq:WigdSymms}
	\Wigd{d}{m'}{m}(x) = (-1)^{m'+m} \Wigd{d}{-m'}{-m}(x) = (-1)^{m'+m} \Wigd{d}{m}{m'}(x).
\end{align}
We will need the first two Jacobi polynomials; they are \cite[eq. 18.5.7]{DLMF}
\begin{align}
\label{eq:JacPBase}
	\JacobiP{0}{\alpha}{\beta}(x)=1, \qquad \JacobiP{1}{\alpha}{\beta}(x)=\frac{\alpha-\beta+x(2+\alpha+\beta)}{2}.
\end{align}

\section{The action of the $X_j$ operators}
\subsection{The $X_j$ as differential operators}
\label{sect:XjOps}
Starting from \eqref{eq:ZDef}, we define the differential operators
\begin{align}
\label{eq:tildeZDef}
	\wtilde{Z}_{\pm2} = Z_{\pm2} \mp i \tfrac{\sqrt{2}}{2}\KLeft{0}, \qquad \wtilde{Z}_{\pm 1} = Z_{\pm 1}-\KLeft{\pm1}, \qquad \wtilde{Z}_0=Z_0.
\end{align}

\begin{lem}
As differential operators in the Iwasawa $UYK$, i.e. $g=xyk$, coordinates, the $X_j$ operators have the form
\begin{align}
\label{eq:XjOp}
	X_j =& \sum_{\ell=-2}^2 \WigD{2}{\ell}{j}(k) \wtilde{Z}_\ell.
\end{align}
\end{lem}
We have expressed, as we may, all of the trigonometric polynomials in $\alpha, \beta, \gamma$ in terms of the Wigner $\WigDName$-matrices in preparation for application of the Clebsch-Gordan multiplication rules, see section \ref{sect:CG}.
\begin{proof}
We have the relation \eqref{eq:XkalphaComm}, and it can be computed directly using 
\begin{align*}
	\WigDMat{2}(w_3) =& \frac{1}{4}\Matrix{
	1&2i&-\sqrt{6}&-2i&1\\
	-2i&2&0&2&2i\\
	-\sqrt{6}&0&-2&0&-\sqrt{6}\\
	2i&2&0&2&-2i\\
	1&-2i&-\sqrt{6}&2i&1}
\end{align*}
that
\[ w_3 X_j w_3 = \sum_{\abs{\ell} \le 2} \WigD{2}{\ell}{j}(w_3) X_\ell, \]
so from \eqref{eq:kabcDef} and \eqref{eq:WigDPrimary} it follows that for $k \in K$,
\begin{align}
\label{eq:SimpkXjkinv}
	k X_j k^{-1} = \sum_{\abs{\ell} \le 2} \WigD{2}{\ell}{j}(k) X_\ell.
\end{align}
This applies to the simple trick
\begin{align}
\label{eq:LieAlgTrick}
	(X_j f)\paren{xyk} &= \left.\frac{d}{dt} f\paren{xy \exp(t kX_jk^{-1})k}\right|_{t=0},
\end{align}
for a smooth function $f$ on $G$.
We switch to the Iwasawa-friendly basis
\[ N_1=E_{2,3}, \; N_2=E_{1,2}, \; N_3=E_{1,3}, \; 3A_1=E_{1,1}+E_{2,2}-2E_{3,3}, \; 3A_2=2E_{1,1}-E_{2,2}-E_{3,3}. \]
The conversion is
\begin{align}
\label{eq:XtoIwaConv}
	X_{\pm2}=\pm 2i N_2-A_1+\sqrt{2}A_2\mp iK_0, \quad X_{\pm1}=-2i N_1 \mp 2N_3-K_{\pm1}, \quad X_0=-\sqrt{6} A_1,
\end{align}
and we can compute directly from the definition \eqref{eq:LieActDef} that
\begin{align}
\label{eq:UYops}
	N_1=y_1\partial_{x_1}+y_1 x_2 \partial_{x_3}, \quad N_2= y_2\partial_{x_2}, \quad N_3 = y_1 y_2 \partial_{x_3}, \quad A_1=y_1 \partial_{y_1}, \quad A_2=y_2\partial_{y_2}
\end{align}
as differential operators on functions of $U(\R) Y^+$ alone.
For example, if we write $f(xy)=f(x_1,x_2,x_3,y_1,y_2)$ for a smooth function on $U(\R) Y^+$, the action of $N_1$ is computed by
\begin{align*}
	(N_1 f)(xy):=\left.\frac{d}{dt} f(x y \exp(t N_1)) \right|_{t=0}=\left.\frac{d}{dt} f(x_1+t y_1,x_2,x_3+t y_1 x_2,y_1, y_2) \right|_{t=0}.
\end{align*}

The lemma follows from applying \eqref{eq:SimpkXjkinv}, \eqref{eq:XtoIwaConv} and \eqref{eq:UYops} in \eqref{eq:LieAlgTrick}.
Note that the $N_j$ and $A_j$ operators are still acting on the right of the $U(\R)Y^+$ part of the Iwasawa decomposition, and this produces the $Z_j$ operators, but the $K_j$ operators are acting now \textit{on the left} of the $K$ part.
\end{proof}

\subsection{The action of $X_j$ on vector-valued functions}
\label{sect:Xjaction}
We wish to describe the effect of each operator $X_j$ on functions $f \in C(G,\WigDMat{d})$.
We may restrict our attention to $\tildeWigP{d+a}{\ell} X_j f$ with $\abs{a} \le 2, \abs{\ell} \le d+a$, by the explicit form of the $X_j$ in \eqref{eq:XjOp} and the fact that a product $\WigD{a}{i}{j} \WigD{d}{m'}{m}$ lies in the span of $\set{\WigD{d+b}{m'+i}{m+j}\setdiv\abs{b}\le a}$, as in \eqref{eq:CGExpand}.
Further, the components of $f$ may be obtained by applying $K_{\pm 1}$ repeatedly to the central entry $f_0$ using \eqref{eq:Kpm1f}, and we have already described the commutation relations of these operators with $\tildeWigPd{d+a}$ and $X_j$ in \eqref{eq:KPComm} and section \ref{sect:CommRels}, so we need only consider $\tildeWigP{d+a}{\ell} X_j f_0$.
Lastly, by \eqref{eq:XjOp} and \eqref{eq:CGExpand} only the $j$-th column of $\WigDMat{d+a}$ is involved in the operator $X_j f_0$, so it is sufficient to consider $\tildeWigP{d+a}{j} X_j f_0$ for $\abs{a} \le 2$.
This gives, in principle, 25 operators to consider, but up to applying $K_{\pm 1}$ to the result, we have only the following five operators:
Define the vector
\[ \CGB{d}{} = \paren{\CGB{d}{-2}, \ldots, \CGB{d}{2}} := \tfrac{\sqrt{6}}{3}\paren{-2(d+1),-(d+3),-3,(d-2),2d}, \]
and for $\abs{a},\abs{j} \le 2$, $\abs{m'} \le d$, set
\begin{align}
\label{eq:YtildeInitDef}
	\wtilde{Y}^{d,a}_{m',j} = \CG{d,2,a}{m'-j,j} \times \piecewise{Z_{-2}-m'-2 & \If j=-2, \\ Z_{-1} & \If j=-1, \\ Z_0-\CGB{d}{a} & \If j=0, \\ Z_1 & \If j=1, \\ Z_2+m'-2 & \If j=2.}
\end{align}
Then we define the operator on $\wtilde{Y}^{d,a}$ on smooth, scalar-valued functions $f\in C^\infty(G)$ by
\begin{align}
\label{eq:YtildeInitDef2}
	(\wtilde{Y}^{d,a} f)(xyk) = \sum_{m'=-d}^d \WigD{d+a}{m'}{0}(k) \sum_{j=-2}^2 \wtilde{Y}^{d,a}_{m',j} f_{m'-j}(xy),
\end{align}
where
\begin{align}
\label{eq:YtildeInitDef3}
	\sum_{m'=-d}^d \WigD{d}{m'}{0}(k) f_{m'}(xy) = (\tildeWigP{d}{0} f)(xyk).
\end{align}
Here, and throughout, we define $f_{m'}=0$ for $\abs{m'} > d$.

\begin{prop}
\label{prop:XjAction}
	Suppose $f \in \wtilde{C}(G,\WigDMat{d})$, then we have
	\begin{align*}
		\tildeWigP{d+a}{0} X_0 f =& \CG{d,2,a}{0,0} \wtilde{Y}^{d,a} f, \\
		\tildeWigP{d+a}{\pm 1} X_{\pm 1} f =& \tfrac{1}{\sqrt{(d+a)(d+a+1)}}\CG{d,2,a}{0,\pm 1} K_{\pm 1} \wtilde{Y}^{d,a} f, \\
		\tildeWigP{d+a}{\pm 2} X_{\pm 2} f =& \tfrac{1}{\sqrt{(d+a)(d+a+1)}\sqrt{(d+a)(d+a+1)-2}} \CG{d,2,a}{0,\pm 2} K_{\pm 1} K_{\pm 1} \wtilde{Y}^{d,a} f.
	\end{align*}
\end{prop}

Then we can realize $\wtilde{Y}^{d,a}$ as a composition of left-invariant projection and differential operators:
\begin{cor}
\label{cor:wtildeYasLeftInv}
The operator $\wtilde{Y}^{d,a}:C^\infty(G)\to C^\infty(G)$, defined by \eqref{eq:YtildeInitDef} and \eqref{eq:YtildeInitDef2} is given by
\[ \wtilde{Y}^{d,a} = \piecewise{\displaystyle \frac{1}{\CG{d,2,a}{0,0}} \tildeWigP{d+a}{0} X_0 \tildeWigP{d}{0} & \If a=-2,0,2,\\ \displaystyle \frac{1}{\CG{d,2,a}{0,1}\sqrt{(d+a)(d+a+1)}} K_{-1} \tildeWigP{d+a}{1} X_1 \tildeWigP{d}{0} & \If a=\pm 1.} \]
In particular, $\wtilde{Y}^{d,a}$ is left translation invariant.
\end{cor}
Note that $\CG{d,2,\pm1}{0,0}=0$ by \eqref{eq:CG2pm1}, so we cannot use the top expression for $\wtilde{Y}^{d,a}$ when $a=\pm1$. 

For each $d$, we may extend this to an operator $Y^a:C^\infty(G,\C^{2d+1})\to C^\infty(G,\C^{2d+2a+1})$ on smooth, vector-valued functions
\[ C^\infty(G,\C^{2d+1}) :=\set{f:G\to \C^{2d+1}\setdiv f \text{ smooth}}, \]
by applying $\wtilde{Y}^{d,a}$ to the central entry and projecting back to a vector-valued function:
\[ Y^a f = \WigP{d+a}{0} \wtilde{Y}^{d,a} f_0, \qquad f=(f_{-d},\ldots,f_d) \in C^\infty(G,\C^{2d+1}). \]
The components of $Y^a f \in C^\infty(G,\C^{2d+2a+1})$ are then given by
\begin{align}
\label{eq:vecYDef}
	(Y^a f)_{m'}(xy) = \sum_{j=-2}^2 \wtilde{Y}^{d,a}_{m',j} f_{m'-j}(xy),
\end{align}
and again we may realize this as a left-invariant operator by
\begin{align}
\label{eq:vecYLeftTrans}
	Y^a f = \piecewise{\displaystyle \frac{1}{\CG{d,2,a}{0,0}} \WigP{d+a}{0} X_0 f_0 & \If a=-2,0,2, \\ \displaystyle \frac{1}{\CG{d,2,a}{0,1}} \WigP{d+a}{1} X_1 f_0 & \If a=\pm 1.}
\end{align}
The dimension $d$ in the domain of the operator $Y^a$ is to be understood from context; alternately, via the natural extension, one may think of $Y^a$ as an operator on the union of vector-valued functions of all (odd) dimensions
\begin{align}
\label{eq:YDom}
	Y^a:\bigcup_d C^\infty(G,\C^{2d+1}) \to \bigcup_d C^\infty(G,\C^{2d+1}),
\end{align}
which satisfies $Y^a C^\infty(G,\C^{2d+1}) \subset C^\infty(G,\C^{2d+2a+1})$ for any $d \ge 0$, $\abs{a}\le 2$, $d+a \ge 0$.
In the course of the current paper, we have no need to place any algebraic structure on the space $\bigcup_d C^\infty(G,\C^{2d+1})$.

\begin{proof}[Proof of proposition \ref{prop:XjAction}]
Consider $\tildeWigP{d+a}{j} X_j f$, and assume for convenience $\CG{d,2,a}{0,j} \ne 0$, since otherwise the result is trivial.
The function of the operators
\[ \tfrac{1}{\sqrt{(d+a)(d+a+1)}} K_{\pm 1} \qquad \text{and} \qquad \tfrac{1}{\sqrt{(d+a)(d+a+1)}\sqrt{(d+a)(d+a+1)-2}} K_{\pm 1} K_{\pm 1} \]
is precisely to shift the column of the Wigner $\WigDName$-matrix on which $\wtilde{Y}^{d,a} f$ is supported from zero (by \eqref{eq:YtildeInitDef2}) to $\pm1$ and $\pm2$, respectively.

We apply \eqref{eq:KLeftactWD} and \eqref{eq:CGExpand} to \eqref{eq:XjOp}, giving
\begin{align*}
	\frac{1}{\CG{d,2,a}{0,j}} \tildeWigP{d+a}{j} X_j f =	& \sum_{m'=-d}^d \sum_{\ell=-2}^2 \CG{d,2,a}{m',\ell} \WigD{d+a}{m'+\ell}{j}(k) Z_\ell f_{m'}(xy) \\
	& \qquad + \sum_{m'=-d}^d f_{m'}(xy) \sum_\pm \pm m' \CG{d,2,a}{m',\pm2} \WigD{d+a}{m'\pm2}{j}(k) \\
	& \qquad - \sum_{m'=-d}^d f_{m'}(xy) \sum_\pm \CG{d,2,a}{m'\mp1,\pm1} \sqrt{d(d+1)-m'(m'\mp 1)} \WigD{d+a}{m'}{j}(k).
\end{align*}

Now we send $m' \mapsto m'-\ell$ in the first two sums (using $\ell=\pm2$ in the second sum), and the result follows from
\[ \sum_\pm \CG{d,2,a}{m'\mp1,\pm1} \sqrt{d(d+1)-m'(m'\mp 1)} = \CG{d,2,a}{m',0} \CGB{d}{a}, \]
which can be verified directly from \eqref{eq:CG2pm2}-\eqref{eq:CG1} or by the properties of the Wigner $3j$-Symbol \cite[section 34.3]{DLMF}.
\end{proof}

\subsection{The action of $Y^a$ on the power function}
For each $d$, define $Y^a_\mu:\C^{2d+1}\to\C^{2d+2a+1}$ by
\begin{align}
\label{eq:YmuDef}
Y^a_\mu f = p_{-\rho-\mu}(y)Y^a p_{\rho+\mu}(y) f.
\end{align}
As with the definition \eqref{eq:YDom}, either the dimension of the domain of $Y^a_\mu$ is to be understood from context, or we may use the natural extension
\begin{align}
\label{eq:YmuDom}
	Y^a_\mu:\bigcup_d \C^{2d+1} \to \bigcup_d \C^{2d+1},
\end{align}
which satisfies $Y^a_\mu \C^{2d+1}\subset\C^{2d+2a+1}$, and again, we have no need to place an algebraic structure on $\bigcup_d \C^{2d+1}$.

The $Z_j$ operators collectively act on the power function by eigenvalues as
\[ (Z_{-2} p_{\rho+\mu}, \ldots, Z_2 p_{\rho+\mu}) = (\mu_1-\mu_2+1,0,\sqrt{6}(\mu_3-1),0,\mu_1-\mu_2+1) p_{\rho+\mu}, \]
and using $\bv^d_j$ and $\bu^{d,\pm}_j$ as in \eqref{eq:vdef} and \eqref{eq:udef}, we have
\begin{align}
	Y^a \bv^d_j p_{\rho+\mu} = \Bigl(
	&\bv^{d+a}_{j-2} \CG{d,2,a}{j,-2}(\mu_1-\mu_2+1-j) \\
	+&\bv^{d+a}_{j} \CG{d,2,a}{j,0}(\sqrt{6}(\mu_3-1)-\CGB{d}{a}) \nonumber \\
	+&\bv^{d+a}_{j+2} \CG{d,2,a}{j,2}(\mu_1-\mu_2+1+j) \Bigr) p_{\rho+\mu}, \nonumber
\end{align}
and by the symmetry \eqref{eq:CGsymm},
\begin{align}
\label{eq:GeneralFourTermYda}
	Y^a_\mu \bu^{d,\pm}_j =
	&\CG{d,2,a}{j,-2} (\mu_1-\mu_2+1-j) \bu^{d+a,\pm}_{j-2} \\
	+&\CG{d,2,a}{j,0} (\sqrt{6}(\mu_3-1)-\CGB{d}{a}) \bu^{d+a,\pm}_{j} \nonumber \\
	+&\CG{d,2,a}{j,2} (\mu_1-\mu_2+1+j) \bu^{d+a,\pm}_{j+2}. \nonumber
\end{align}
The reason for the $(-1)^d$ in the definition of $\bu^{d,\pm}_j$ is to maintain consistency across parities of $a$ in the above equation.

Because this is the key identity, we write it out explicitly for each $a$.
We assume $\abs{j} \le d$.
At $a=0$, $Y^0_\mu \bu^{0,\pm}_j=0$, and when $d > 0$,
\begin{align}
\label{eq:ThreeTermRecursion}
	0=& \sqrt{6(d+2-j)(d+1-j)(d+j)(d-1+j)} (\mu_1-\mu_2+1-j) \bu^{d,\pm}_{j-2} \\
	&-2\paren{\sqrt{6}(d(d+1)-3j^2) \mu_3+\sqrt{d(d+1)(2d-1)(2d+3)}Y^0_\mu} \bu^{d,\pm}_j \nonumber \\
	&+\sqrt{6(d+2+j)(d+1+j)(d-j)(d-1-j)} (\mu_1-\mu_2+1+j) \bu^{d,\pm}_{j+2}. \nonumber
\end{align}
At $a=1$,
\begin{align}
\label{eq:FourTermInduction1}
	&\sqrt{2d(d+1)(d+2)(2d+1)} Y^1_\mu \bu^{d,\pm}_j = \\
	&\qquad-\sqrt{(d+1-j)(d+2-j)(d+3-j)(d+j)} (\mu_1-\mu_2+1-j) \bu^{d+1,\pm}_{j-2} \nonumber \\
	&\qquad-2j\sqrt{(d+1-j)(d+1+j)} (3\mu_3-d-1) \bu^{d+1,\pm}_{j} \nonumber \\
	&\qquad+\sqrt{(d-j)(d+1+j)(d+2+j)(d+3+j)} (\mu_1-\mu_2+1+j) \bu^{d+1,\pm}_{j+2}. \nonumber
\end{align}
At $a=-1$,
\begin{align}
\label{eq:FourTermReduction1}
	& \sqrt{2d(d-1)(d+1)(2d+1)} Y^{-1}_\mu \bu^{d,\pm}_j = \\
	&\qquad-\sqrt{(d+1-j)(d-2+j)(d-1+j)(d+j)} (\mu_1-\mu_2+1-j) \bu^{d-1,\pm}_{j-2} \nonumber \\
	&\qquad+2j\sqrt{(d-j)(d+j)} (3\mu_3+d) \bu^{d-1,\pm}_{j} \nonumber \\
	&\qquad+\sqrt{(d-2-j)(d-1-j)(d-j)(d+1+j)} (\mu_1-\mu_2+1+j) \bu^{d-1,\pm}_{j+2}. \nonumber
\end{align}
At $a=2$,
\begin{align}
\label{eq:FourTermInduction2}
	&2\sqrt{(d+1)(d+2)(2d+1)(2d+3)} Y^2_\mu \bu^{d,\pm}_j = \\
	&\qquad\sqrt{(d+1-j)(d+2-j)(d+3-j)(d+4-j)} (\mu_1-\mu_2+1-j) \bu^{d+2,\pm}_{j-2} \nonumber \\
	&\qquad+2\sqrt{(d+1-j)(d+2-j)(d+1+j)(d+2+j)} (3\mu_3-2d-3) \bu^{d+2,\pm}_{j} \nonumber \\
	&\qquad+\sqrt{(d+1+j)(d+2+j)(d+3+j)(d+4+j)} (\mu_1-\mu_2+1+j) \bu^{d+2,\pm}_{j+2}. \nonumber
\end{align}
At $a=-2$,
\begin{align}
\label{eq:FourTermReduction2}
	& 2\sqrt{d(d-1)(2d-1)(2d+1)} Y^{-2}_\mu \bu^{d,\pm}_j = \\
	&\qquad\sqrt{(d-3+j)(d-2+j)(d-1+j)(d+j)} (\mu_1-\mu_2+1-j) \bu^{d-2,\pm}_{j-2} \nonumber \\
	&\qquad+2\sqrt{(d-1-j)(d-j)(d-1+j)(d+j)} (3\mu_3+2d-1) \bu^{d-2,\pm}_{j} \nonumber \\
	&\qquad+\sqrt{(d-3-j)(d-2-j)(d-1-j)(d-j)} (\mu_1-\mu_2+1+j) \bu^{d-2,\pm}_{j+2}. \nonumber
\end{align}

\section{Interactions with the Lie algebra}
From the discussion of the previous section and sections \ref{sect:LieAlgebras} and \ref{sect:KProjOps}, in studying the action of the Lie algebra $\frak{g}_\C$ on $\mathcal{A}$, it is sufficient to study the action of the $Y^a$ operators on $\mathcal{A}^d$.
Further, by the Fourier expansion (I.3.30), it is sufficient to study the action of $Y^a_\mu$ on $\C^{2d+1}$, provided we know which vectors $f\in\C^{2d+1}$ are sent to zero by the Jacquet-Whittaker function, which we will carefully study in section \ref{sect:GDWhitt}.

For the moment, we need to study the interaction of the Lie algebra, via the $Y^a_\mu$ operators, with several common operations.

\subsection{The adjoint of $Y^a$}
First, we compute the adjoint of the $Y^a$ operators with respect to the inner product on $\mathcal{A}^d$.
We use the composition of operators given in \eqref{eq:vecYLeftTrans}.

Suppose $f\in\wtilde{C}(\Gamma\backslash G,\WigDMat{d})$ and $h\in C(\Gamma\backslash G,\WigDMat{d})$ are square-integrable, then
\begin{align*}
	\int_{\Gamma\backslash G} (\WigP{d}{m}f)(g) \, \wbar{\trans{h(g)}} dg &= (2d+1) \int_{\Gamma\backslash G} f(g) \int_K \WigDRow{d}{m}\paren{k^{-1}} \WigDMat{d}(k) dk \, \wbar{\trans{h(g)}} dg \\
	&= (2d+1) \int_{\Gamma\backslash G} f(g) \, \wbar{h_{m}(g)} dg,
\end{align*}
since the integral over $K$ in the middle is just the $m$-th row of the identity matrix (indexing from the center).

If instead $f,h\in\mathcal{A}$, with $\mathcal{A}$ as in section \ref{sect:MainThms}, then
\begin{align}
\label{eq:LieAdjointOp}
	\int_{\Gamma\backslash G} (X f)(g) \, \wbar{h(g)} dg &= \int_{\Gamma\backslash G}  f(g) \, \wbar{(-\wbar{X} h)(g)} dg,
\end{align}
directly from the definitions \eqref{eq:LieActDef} and the Haar measure $dg$.

Now suppose $f\in \mathcal{A}^d$ and $h\in \mathcal{A}^{d+a}$, then if $a$ is even,
\begin{align*}
	\int_{\Gamma\backslash G} (Y^a f)(g) \, \wbar{\trans{h(g)}} dg &= \frac{(2d+2a+1)}{\CG{d,2,a}{0,0}} \int_{\Gamma\backslash G} (X_0 f_0)(g) \, \wbar{h_0(g)} dg \\
	&= \frac{(2d+2a+1)}{\CG{d,2,a}{0,0}} \int_{\Gamma\backslash G} f_0(g) \, \wbar{(-\wbar{X_0} h_0)(g)} dg \\
	&= \frac{(2d+2a+1)}{(2d+1)\CG{d,2,a}{0,0}} \int_{\Gamma\backslash G} f(g) \wbar{\trans{\paren{\WigP{d}{0} (-\wbar{X_0}) h_0}(g)}} \, dg,
\end{align*}
so the adjoint of $Y^a$ is
\[ \what{Y}^a h = -\frac{(2d+2a+1)}{(2d+1)\CG{d,2,a}{0,0}} \WigP{d}{0} X_0 h_0 = -\frac{(2d+2a+1)\CG{d+a,2,-a}{0,0}}{(2d+1)\CG{d,2,a}{0,0}} Y^{-a}, \]
since the Clebsch-Gordan coefficients and $X_0$ are real.
Similarly, if $a$ is odd, the adjoint is
\[ \what{Y}^a h = \tfrac{(2d+2a+1)}{(2d+1)\CG{d,2,a}{0,1}} \WigP{d}{0} (-\wbar{X_1}) h_1, \]
but we would prefer to use $h_0$, since this is where we have done all of our computations so far, so we use
\[ \what{Y}^a h = \tfrac{(2d+2a+1)}{(2d+1)\sqrt{(d+a)(d+a+1)}\CG{d,2,a}{0,1}} \WigP{d}{0} X_{-1} K_1 h_0, \]
and
\begin{align*}
	\WigP{d}{0} X_{-1} K_1 h_0 =& \WigP{d}{0} K_1 X_{-1} h_0 -\sqrt{6} \WigP{d}{0} X_0 h_0 = \sqrt{d(d+1)} \WigP{d}{-1} X_{-1} h_0-0 \\
	=& \sqrt{d(d+1)} \CG{d+a,2,-a}{0,-1} Y^{-a},
\end{align*}
(recall the remark below corollary \ref{cor:wtildeYasLeftInv}) giving
\[ \what{Y}^a = \frac{(2d+2a+1)\sqrt{d(d+1)} \CG{d+a,2,-a}{0,-1}}{(2d+1)\sqrt{(d+a)(d+a+1)}\CG{d,2,a}{0,1}} Y^{-a}. \]

In general, after writing out the Clebsch-Gordan coefficients, we have:
\begin{prop}
With respect to the inner products \eqref{eq:VecCuspInnerProd} on $\mathcal{A}^d$ and $\mathcal{A}^{d+a}$, the operator $Y^a$ has adjoint
\begin{align}
\label{eq:AdjY}
	\what{Y}^a =& -(-1)^a \sqrt{\tfrac{2d+2a+1}{2d+1}} Y^{-a}.
\end{align}
\end{prop}

By analogy with \eqref{eq:YmuDef}, we define the adjoint action on power functions $\what{Y}^a_\mu:\C^{2d+2a+1}\to\C^{2d+1}$ by
\begin{align}
\label{eq:AdjYmuDef}
\what{Y}^a_\mu f := p_{-\rho-\mu}(y)\what{Y}^a p_{\rho+\mu}(y) f = -(-1)^a \sqrt{\tfrac{2d+2a+1}{2d+1}} Y^{-a}_\mu f.
\end{align}

\subsection{The intertwining operators}
We require an understanding of the interaction between the Lie algebra, i.e. the $Y^a$ operators, and the intertwining operators, a.k.a. the matrices $T^d(w,\mu)$ occuring in the functional equations of the Whittaker function, as in \eqref{eq:WhittFEs}.

First, a lemma:
\begin{lem}
\label{eq:WhittLinIdepn}
	Suppose none of the differences $\mu_i-\mu_j,i\ne j$ are in $\Z$, then for a vector $f\in\C^{2d+1}$, $d \ge 0$ and any $U(\R)$-character $\psi$, the vector-valued function $f W^d(g,\mu,\psi)$ is identically zero as a function of $g$ exactly when $f=0$.
\end{lem}
\begin{proof}
For a non-degenerate character $\psi_y$, $y_1y_2\ne0$, and either $\Re(\mu_1)>\Re(\mu_2)>\Re(\mu_3)$ or $\Re(\mu)=0$, this follows from (I.3.28) and (I.3.29) and the fact that $W^d(g,\mu,\psi_{0,0})$ and $T^d(w,\mu)$ are given by products of $\WigDMat{d}(\vpmpm{--})$, $\WigDMat{d}(w_l)$, $\Gamma^d_\mathcal{W}(u,+1)$, and $\mathcal{W}^d(0,u)$ which are invertible matrices when $u \notin \Z$ (see \eqref{eq:WhittGammas} and (I.2.18)).
This extends to the cases where only one $\Re(\mu_i)=\Re(\mu_j)$, $i \ne j$ and degenerate characters $\psi_{0,y_2}, \psi_{y_1,0},\psi_{0,0}$ by the same argument as in section I.3.5.
\end{proof}

Suppose none of $\mu_i-\mu_j\in\Z,i\ne j$, then for $f\in\C^{2d+1}$ and any $w\in W$,
\begin{align*}
	(Y^a_\mu (f T^d(w^{-1},\mu^w)))W^{d+a}(g,\mu,\psi_{1,1}) =& Y^a (f T^d(w^{-1},\mu^w))W^d(g,\mu,\psi_{1,1})) \\
	=& Y^a (f W^d(g,\mu^w,\psi_{1,1})) \\
	=& (Y^a_{\mu^w} f) W^{d+a}(g,\mu^w,\psi_{1,1}) \\
	=& ((Y^a_{\mu^w} f) T^{d+a}(w^{-1}, \mu^w)) W^{d+a}(g,\mu,\psi_{1,1}),
\end{align*}
and the lemma implies
\begin{align}
\label{eq:IntertwiningOpsY}
	Y^a_\mu (f T^d(w^{-1},\mu^w)) = (Y^a_{\mu^w} f) T^{d+a}(w^{-1}, \mu^w).
\end{align}
This continues to an equality of meromorphic functions in $\mu$.

\subsection{Duality}
As in (I.5.4), we define the involution $\iota:G\to G$ by $g^\iota = w_l\trans{\paren{g^{-1}}} w_l$.
To every cusp form $\phi \in \mathcal{A}^d$, we may associate a dual form $\wcheck{\phi}(g) = \phi(g^\iota w_l) \in \mathcal{A}^d$, which is frequently used to show isomorphisms between subspaces of cusp forms.
We will require an understanding of this duality, and in particular, its interaction with the Lie algebra, at the level of Whittaker functions.

Comparing (I.3.22) to (I.3.24), we have the relation
\[ \WigDMat{d}(\vpmpm{--} w_l) W^d(I,\mu,\psi_y) \WigDMat{d}(w_l \vpmpm{--}) = W^d(I,-\mu^{w_l},\psi_{y^\iota}), \]
and by (I.3.7), we have
\[ W^d(y,\mu,\psi_{1,1}) = \WigDMat{d}(\vpmpm{--} w_l) W^d(y^\iota,-\mu^{w_l},\psi_{1,1}) \WigDMat{d}(\vpmpm{--} w_l). \]
Thus in general, we have
\begin{align}
\label{eq:WhittDuality}
	W^d(g,\mu,\psi_{1,1}) = \WigDMat{d}(\vpmpm{--} w_l) W^d(\vpmpm{--} g^\iota w_l,-\mu^{w_l},\psi_{1,1}).
\end{align}

For $f\in C^\infty(G)$, set $\wcheck{f}(g)=f(g^\iota w_l)$, then the action of the $K_j$ operators on the dual form $\wcheck{f}$ is given by
\begin{align*}
	K_j \wcheck{f}(g) =& \left. \frac{d}{dt} f\paren{g^\iota w_l \exp (-t \trans{K_j})}\right|_{t=0} = \left. \frac{d}{dt} f\paren{g^\iota w_l \exp t K_j}\right|_{t=0} = \wcheck{K_j f}(g^\iota).
\end{align*}
Similarly $X_j \wcheck{f}(g) = -\wcheck{X_j f}(g^\iota)$.
Therefore, for $f\in\C^d$,
\begin{align*}
	(Y^a_\mu (f \WigDMat{d}(\vpmpm{--} w_l))) W^{d+a}(g,\mu,\psi_{1,1}) =& Y^a f W^d(\vpmpm{--} g^\iota w_l,-\mu^{w_l},\psi_{1,1}) \\
	=& -(Y^a_{-\mu^{w_l}} f) W^{d+a}(\vpmpm{--} g^\iota w_l,-\mu^{w_l},\psi_{1,1}) \\
	=& -(Y^a_{-\mu^{w_l}} f) \WigDMat{d}(\vpmpm{--} w_l) W^{d+a}(g,\mu,\psi_{1,1}),
\end{align*}
and we conclude as in the previous subsection that
\begin{align}
\label{eq:YamuDuality}
	Y^a_\mu (f \WigDMat{d}(\vpmpm{--} w_l)) = -(Y^a_{-\mu^{w_l}} f) \WigDMat{d}(\vpmpm{--} w_l).
\end{align}

\section{The minimal-weight Whittaker functions}
\label{sect:MinWhitt}
In this section, we will prove theorem \ref{thm:MinWhitt}, and analyze some additional, degenerate cases of the Jacquet-Whittaker integral for the purpose of proving they cannot occur among the cusp forms.
First, we give some general results on Whittaker functions that were not needed for the previous paper.

\subsection{The differential equations satisfied by the Whittaker functions.}
We wish to develop raising and lowering operators on the entries of the vector-valued Whittaker functions.
Suppose $\Phi:G\to\C^{2d+1}$ is defined over the Iwasawa decomposition by
\[ \Phi(xyk) = \phi(xy) \WigDMat{d}(k), \qquad \phi(xy) = \paren{\phi_{-d}(xy), \ldots, \phi_d(xy)}, \]
then the components of $\Phi=(\Phi_{-d},\ldots,\Phi_d)$ are given by
\[ \Phi_m = \sum_{m'=-d}^d \phi_{m'} \WigD{d}{m'}{m}, \qquad \Phi = \sum_{m'=-d}^d \phi_{m'} \WigDRow{d}{m'} \]
and requiring $\Delta_i \Phi = \lambda_i \Phi$, $i=1,2$, is equivalent to
\begin{align*}
	0 =& (4\Delta_1^\circ+m' (Z_2-Z_{-2})-4\lambda_1) \phi_{m'} \\
	& \qquad -\sqrt{d(d+1)-m'(m'+1)} Z_{-1} \phi_{m'+1} \\
	& \qquad -\sqrt{d(d+1)-m'(m'-1)} Z_1 \phi_{m'-1}, \\
	0 =& (48\Delta_2^\circ-m' \mathcal{T}_0-48\lambda_2) \phi_{m'} \\
	& \qquad + \sqrt{d(d+1)-m'(m'+1)} (\mathcal{T}_1-6(m'+1)(Z_{-1}-Z_1)) \phi_{m'+1} \\
	& \qquad +\sqrt{d(d+1)-m'(m'-1)} (\mathcal{T}_{-1}-6(m'-1)(Z_{-1}-Z_1)) \phi_{m'-1},
\end{align*}
by \eqref{eq:FullDelta1} and \eqref{eq:FullDelta2}.

Suppose also that $\phi(xy) = \psi_{1,1}(x) \phi(y)$, then
\begin{align}
\label{eq:Delta1Whitt}
	0 =& (\Delta^\text{Whitt}_1-2\pi m' y_2-\lambda_1) \phi_{m'}(y) \\
	& \qquad - \pi \sqrt{d(d+1)-m'(m'+1)} y_1 \phi_{m'+1}(y) \nonumber \\
	& \qquad - \pi \sqrt{d(d+1)-m'(m'-1)} y_1 \phi_{m'-1}(y), \nonumber \\
\label{eq:Delta2Whitt}
	0 =& (\Delta^\text{Whitt}_2 +2\pi m' y_2\paren{y_1 \partial_{y_1} - 1}-\lambda_2) \phi_{m'}(y) \\
	& \qquad +\pi \sqrt{d(d+1)-m'(m'+1)} y_1 (1-y_2\partial_{y_2}-2\pi y_2) \phi_{m'+1}(y) \nonumber \\
	& \qquad +\pi \sqrt{d(d+1)-m'(m'-1)} y_1 (1-y_2\partial_{y_2}+2\pi y_2) \phi_{m'-1}(y), \nonumber
\end{align}
where
\begin{align*}
	\Delta^\text{Whitt}_1 =& -y_1^2\partial_{y_1}^2-y_2^2\partial_{y_2}^2+y_1y_2\partial_{y_1}\partial_{y_2} +4\pi^2 y_1^2 +4\pi^2 y_2^2, \\
	\Delta^\text{Whitt}_2 =& -y_1^2y_2\partial_{y_1}^2 \partial_{y_2}+y_1y_2^2\partial_{y_1} \partial_{y_2}^2-4\pi^2 y_1y_2^2 \partial_{y_1} +4\pi^2 y_1^2y_2 \partial_{y_2} \\
	& +y_1^2\partial_{y_1}^2-y_2^2\partial_{y_2}^2-4\pi^2 y_1^2+4\pi^2 y_2^2.
\end{align*}

Some rearranging gives
\begin{align}
\label{eq:SpmAct}
	S^{\pm}_{m'} \phi_{m'}(y) =& \pm \sqrt{d(d+1)-m'(m'\pm 1)} \phi_{m'\pm 1}(y), \\
\label{eq:SpmDef}
	S^\pm_{m'} :=& \frac{1}{4 \pi^2 y_1 y_2} \bigl((1-y_2\partial_{y_2}\pm 2\pi y_2) \paren{\Delta^\text{Whitt}_1-\lambda_1}+\paren{\Delta^\text{Whitt}_2-\lambda_2} \\
	& \qquad \qquad +2\pi m' y_2\paren{y_1 \partial_{y_1}+y_2\partial_{y_2}-1\mp 2\pi y_2}\bigr). \nonumber
\end{align}
One particular consequence is that the full vector-valued Whittaker function may be generated from any given entry, so any vector-valued Whittaker function is identically zero exactly when any entry is identically zero.
Precisely, we have:
\begin{lem}
\label{lem:WhittEntryRaiseLower}
For any vector-valued function $\Phi(xyk)=\psi_{1,1}(x) \phi(y) \WigDMat{d}(k)$, with \\ $\phi(y)=(\phi_{-d}(y), \ldots, \phi_d(y))$ which satisfies $\Delta_i \Phi = \lambda_i \Phi$, $i=1,2$, and any $-d \le m_1 < m_2 \le d$, we have
\[ \phi_{m_2} = C_1 S^+_{m_2-1}\cdots S^+_{m_1+1}  S^+_{m_1} \phi_{m_1}, \qquad \phi_{m_1} = C_2 S^-_{m_1+1}\cdots S^-_{m_2-1}  S^-_{m_2} \phi_{m_2}, \]
where
\[ \frac{1}{C_1} = \prod_{j=m_1}^{m_2-1} \sqrt{d(d+1)-j(j+1)}, \qquad \frac{1}{C_2} = (-1)^{m_2-m_1} \prod_{j=m_1+1}^{m_2} \sqrt{d(d+1)-j(j-1)}. \]
\end{lem}

\subsection{The central entry}
\label{sect:WhittCentralEntry}
We turn now to the proof of theorem \ref{thm:MinWhitt}, which is the evaluation of the Jacquet integral at the minimal $K$-types.
As mentioned in the introduction, the proof occurs in two steps; in this subsection, we directly evaluate the Jacquet integral at the central entry, and in the next, we show the Mellin-Barnes integrals of theorem \ref{thm:MinWhitt} satisfy equations \eqref{eq:SpmAct} and \eqref{eq:Delta1Whitt} to complete the theorem.
For the central entry, we also consider the evaluation of $W^d_{-d,0}(g,\mu,\psi_{1,1})$ where $\mu$ is of the form $(d-1,0,1-d)$; we will see this cannot occur as the Whittaker function of a cusp form, but we still require knowledge of precisely when the function is identically zero.

Suppose $\mu_1-\mu_2=d-1$, $d \ge 2$, and temporarily suppose $\Re(\mu_2-\mu_3)$ is large; we will reach the cases of theorem \ref{thm:MinWhitt} by analytic continuation, below.
The cases $d=0,1$ of theorem \ref{thm:MinWhitt} may be handled similarly.
Starting from (I.3.22), we send $u_3 \mapsto u_3 \sqrt{1+u_2^2}$  to obtain
\begin{align}
\label{eq:NewI322}
	W^d_{-d,0}(I, \mu, \psi_y) =& (-1)^d \int_{\R^2} \paren{1+u_3^2}^{\frac{-1+\mu_3-\mu_1}{2}} \paren{1+u_2^2}^{\frac{-1+\mu_3-\mu_2}{2}} \mathcal{W}_{-d}\paren{y_1 \frac{\sqrt{1+u_3^2}}{\sqrt{1+u_2^2}},\mu_1-\mu_2} \\
	& \qquad \Wigd{d}{-d}{0}\paren{\frac{-u_3}{\sqrt{1+u_3^2}}} \e{-y_1 \frac{u_2 u_3}{\sqrt{1+u_2^2}}}\e{-y_2 u_2} du_2 \, du_3.
\end{align}
Then \eqref{eq:classWhittDef} and \cite[eq. 13.18.2]{DLMF} imply
\begin{align*}
	\mathcal{W}_{-d}(y, d-1) =& \frac{(2\pi)^d y^{d-1}}{(d-1)!} \exp\paren{-2\pi y}.
\end{align*}
We also have, by \eqref{eq:WigdtoJacobi} and \eqref{eq:JacPBase},
\[ \Wigd{d}{-d}{0}(x) = \Wigd{d}{0}{d}(x) = \frac{\sqrt{(2d)!}}{d!\, 2^d} \paren{1-x^2}^{\frac{d}{2}}. \]

Plugging in the two previous displays, applying (I.3.17) (or the definition of the gamma function) to each of the three exponentials, and evaluating the $u$ integrals with (I.2.27) gives
\begin{align*}
	& W^d_{-d,0}(I, \mu, \psi_y) = \\
	& (-1)^d \frac{\sqrt{(2d)!} \pi^d y_1^{d-1}}{d!\,(d-1)!} \int_{\Re(s)=\epsilon} (2\pi y_1)^{-s_1-s_3} (2\pi y_2)^{-s_2} \Gamma\paren{s_1} \Gamma\paren{s_3} \Gamma\paren{s_2} \cos \frac{\pi s_2}{2} \cos \frac{\pi s_3}{2} \\
	& \qquad B\paren{\frac{1-s_3-s_2}{2},\frac{\mu_1-\mu_2-s_1+s_2}{2}} B\paren{\frac{1-s_3}{2},\frac{d+\mu_3-\mu_2+s_1+s_3}{2}} \frac{ds}{(2\pi i)^3},
\end{align*}
where $B(a,b) = \Gamma(a) \Gamma(b)/\Gamma(a+b)$ is the Euler beta function.
We may now suppose $\mu=\paren{\frac{d-1}{2}+it,-\frac{d-1}{2}+it,-2it}$ or $\mu=(d-1,0,1-d)$, as the convergence is clear; i.e. the above integral converges to an entire function of $\mu$ by the exponential decay coming from Stirling's formula, so we have the anticipated analytic continuation.

Send $s_1 \mapsto s_1-s_3$, and apply the duplication and reflection properties of the gamma function so that
\begin{align*}
	W^d_{-d,0}(I, \mu, \psi_y) =& \frac{(-1)^d \sqrt{(2d)!} \pi^{\frac{3}{2}}}{8 \, d!\,(d-1)!} \int_{\Re(s)=\epsilon} (\pi y_1)^{d-1-s_1} (\pi y_2)^{-s_2} \frac{\Gamma\paren{\frac{s_2}{2}}\Gamma\paren{\frac{1+\mu_1-\mu_2+s_1}{2}}}{\Gamma\paren{\frac{1-s_2}{2}}\Gamma\paren{\frac{1+\mu_1-\mu_2-s_1}{2}}} \\
	& \qquad \frac{\Gamma\paren{\frac{1-s_3-s_2}{2}} \Gamma\paren{\frac{\mu_1-\mu_2-s_1+s_3+s_2}{2}} \Gamma\paren{\frac{s_3}{2}} \Gamma\paren{\frac{s_1-s_3}{2}} \Gamma\paren{\frac{1+s_1-s_3}{2}}}{\Gamma\paren{\frac{d+1+\mu_3-\mu_2+s_1-s_3}{2}}} \frac{ds}{(2\pi i)^3},
\end{align*}
using
\[ 1+\mu_1-\mu_3=d+\mu_2-\mu_3=\piecewise{\tfrac{d+1}{2}+3it & \If \mu=\paren{\frac{d-1}{2}+it,-\frac{d-1}{2}+it,-2it}, \\ 2d-1 & \If \mu=(d-1,0,1-d).} \]

We apply \eqref{eq:BarnesSecond} on $-\frac{s_3}{2}$, substitute $(s_1,s_2)\mapsto (s_1+d-1-\mu_1,s_2+\mu_3)$, and use (I.3.6) to obtain
\begin{align}
\label{eq:WhittCentralEntry}
	W^d_{-d,0}(y, \mu, \psi_{1,1}) =& \frac{(-1)^d}{4\pi^2} \int_{\Re(s)=\mathfrak{s}} (\pi y_1)^{1-s_1} (\pi y_2)^{1-s_2} \frac{G^d_0(s,\mu)}{\Lambda^*(\mu)} \frac{ds}{(2\pi i)^2},
\end{align}
with $G^d$ as in \eqref{eq:MinWhittGPSR}, and using the contour
\[ \mathfrak{s}=\piecewise{(\epsilon,\epsilon) & \If \mu=\paren{\frac{d-1}{2}+it,-\frac{d-1}{2}+it,-2it}, \\ (\epsilon,d-1+\epsilon) & \If \mu=(d-1,0,1-d).} \]

The opposite case $W^d_{d,0}(y, \mu, \psi_{1,1})=0$ follows from the pole of the gamma function of \eqref{eq:classWhittDef} in \eqref{eq:NewI322}.

\subsection{The remaining components of the Whittaker function}
We now complete the proof of theorem \ref{thm:MinWhitt}.
We continue to assume $d \ge 2$ and specify $\mu=\paren{\frac{d-1}{2}+it,-\frac{d-1}{2}+it,-2it}$.
We've computed the base case $m'=0$ directly, and the cases $m'=\pm 1$ can be verified from the $S^\pm_{m'}$ operator, so we have to compute $\abs{m'} \ge 2$.
It is sufficient to verify the terms $1 \le \abs{m'} \le d-1$ satisfy \eqref{eq:Delta1Whitt}, and in the following proof we assume $\abs{m'} \ge 2$ for convenience.
Let $m'=\varepsilon m$, then we need to verify
\begin{align*}
	0 =& \int_{\Re(s)=\mathfrak{s}} \Bigl( (\Delta^\text{Whitt}_1-2\pi m' y_2-\lambda_1(\mu)) (\pi y_1)^{1-s_1} (\pi y_2)^{1-s_2} \sum_{\ell=0}^{m} \varepsilon^\ell \binom{m}{\ell} \wtilde{G}^d((d-m,\ell), s,\mu) \\
	&- (d-m)(\pi y_1)^{1-s_1} (\pi y_2)^{1-s_2}\sum_{\ell=0}^{m+1} \varepsilon^\ell \binom{m+1}{\ell} \wtilde{G}^d((d-m-1,\ell), s,\mu) \\
	&-(d+m) (\pi y_1)^{1-s_1} (\pi y_2)^{1-s_2} \sum_{\ell=0}^{m-1} \varepsilon^\ell \binom{m-1}{\ell} \wtilde{G}^d((d-m+1,\ell), s,\mu) \Bigr)\frac{ds}{(2\pi i)^2}.
\end{align*}

We apply the differential operator and shift $s$ on the various terms so we may deal with the Mellin transform directly.
We factor out
\[ \Gamma\paren{\tfrac{d-1-\mu_1+s_1}{2}}\Gamma\paren{\tfrac{d-\mu_1+s_1}{2}} \Gamma\paren{\tfrac{\mu_1+s_2}{2}}\Gamma\paren{\tfrac{1+\mu_1+s_2}{2}} \]
which replaces the terms $\wtilde{G}^d(\ldots)$ with a polynomial times a beta function of the form
\[ B\paren{\tfrac{d-m+a-\mu_3+s_1}{2}, \tfrac{\ell+b+\mu_3+s_2}{2}}, \qquad a=0,2, \qquad b=0,1,2. \]
Now apply
\[ B(x+1,y)=B(x,y)-B(x,y+1) \]
to normalize $a \mapsto 0$, and substitute $\ell \mapsto \ell-b$ to align the beta functions.
After some algebra and removing a factor $(s_1-s_2-2it)$, we need to show
\begin{align*}
	0=& (s_2+\mu_3)B\paren{\tfrac{d-m-\mu_3+s_1}{2},\tfrac{\mu_3+s_2}{2}}+\varepsilon m(s_2+\mu_3+1) B\paren{\tfrac{d-m-\mu_3+s_1}{2},\tfrac{\mu_3+s_2+1}{2}}\\
	&+ \sum_{\ell=2}^m \varepsilon^\ell \paren{\binom{m}{\ell}(s_2+\mu_3+\ell)-\binom{m}{\ell-2}(d-m+s_1+s_2+\ell-2)} \\
	& \qquad \times B\paren{\tfrac{d-m-\mu_3+s_1}{2},\tfrac{\mu_3+s_2+\ell}{2}}\\
	&-\varepsilon^{m+1} m(d-1+s_1+s_2) B\paren{\tfrac{d-m-\mu_3+s_1}{2},\tfrac{\mu_3+s_2+m+1}{2}} \\
	&-\varepsilon^m (d+s_1+s_2) B\paren{\tfrac{d-m-\mu_3+s_1}{2},\tfrac{\mu_3+s_2+m+2}{2}}
\end{align*}
The even and odd terms of this sum then separately telescope to zero.

\subsection{The bad Whittaker functions}
\label{sect:BadWhitt}
Suppose $\mu=(d-1,0,1-d)$.
It is a well-known fact (see the corollary to \cite[lemma 15]{HC01}) that cusp forms, and hence their Fourier coefficents are bounded (as functions of $G$).
Then \eqref{eq:WhittCentralEntry} shows the Whittaker functions $W^d_{-d}(\cdot, \mu, \psi_{1,1})$ are non-zero and have bad asymptotics, i.e. they are not suitable for cusp forms since the central entry tends to infinity as $y_2\to 0$ (shift the $s$ contours to the left past the pole at $(s_1,s_2)=(0,d-1)$).
We may instead rule out any cusp form having such a Whittaker function by noticing it would necessarily have a real eigenvalue of the skew-symmetric operator $Y^0$, and this is what we do in section \ref{sect:GDCuspForms}.

\section{Going Down}
\label{sect:GoingDown}

We determine the spectral parameters and Whittaker functions of all forms which are killed by both $Y^{-1}$, $Y^{-2}$.
Equivalently, we determine vectors in $\C^{2d+1}$ which are killed by both of $Y^{-1}_\mu$, $Y^{-2}_\mu$.
The unitaricity conditions on the spectral parameters of Maass cusp forms, meaning the symmetries of the differential operators $\Delta_1$ and $\Delta_2$, imply that, for such forms, $-\wbar{\mu}$ is a permutation of $\mu$, and further, because of the absence of poles in the required functional equations of the Whittaker function, we may assume $\Re(\mu_1)\ge\Re(\mu_2)\ge\Re(\mu_3)$, so
\begin{align}
\label{eq:GDAssume}
	\mu = (i t_1, i t_2,-i(t_1+t_2)) \qquad \text{ or } \qquad \mu=(x+i t, -2i t, -x+i t),
\end{align}
where $t_1-t_2,2t_1+t_2,t_1+2t_2\ne0$, $x\ge 0$, and possibly $t=0$.

\subsection{For the power function}
Define
\begin{align}
	g_1=-\sqrt{6} (1+\mu_3) \bu^{2,+}_2+(\mu_1-\mu_2-1) \bu^{2,+}_0,
\end{align}
\begin{align}
\label{eq:g2def}
	g_2^{d,\delta,\varepsilon} =& \sum_{0 \le 2j+\delta \le d} g_{2,\delta,2j+\delta}^d \bu^{d,\varepsilon}_{2j+\delta},
\end{align}
where the coefficients are given by $g_{2,0,0}^d = \tfrac{1}{2}$, and
\begin{align}
	g_{2,k,2j+k}^d =& (-1)^{j} \prod_{i=0}^{j-1} \frac{\sqrt{(d-1-2i-k)(d-2i-k)} (3\mu_3+2d-1+2i+k)}{\sqrt{(d+1+2i+k)(d+2+2i+k)} (\mu_1-\mu_2-1-2i-k)},
\end{align}
for $0 < 2j+k \le d$.
We also define $g_3^{d,\delta,\varepsilon}$ using the same coefficients as $g_2^{d,\delta,\varepsilon}$, i.e.
\begin{align}
	g_{3,\delta,2j+\delta}^d :=g_{2,\delta,2j+\delta}^d,  \qquad 0 \le 2j+\delta < d,
\end{align}
but take
\begin{align}
\label{eq:g3LastCoefDef}
	g_{3,\delta,d}^d :=\frac{g_{2,\delta,d-2}^d}{\sqrt{d(2d-1)}}
\end{align}
to be the coefficient of $\bu^{d,\varepsilon}_d$, instead.
Lastly, for $d$ odd, define
\begin{align}
\label{eq:g4def}
	g_4^{d,\varepsilon} =& \sum_{\kappa \le 2j+\kappa \le d} g_{2,\kappa,2j+\kappa}^d \bu^{d,\varepsilon}_{2j+\kappa},
\end{align}
where $\kappa=\frac{d+1}{2}$.

\begin{prop}
\label{prop:GDPower}
	If $f\in\C^{2d+1}$ is such that
	\begin{align}
		\label{eq:GDLowering}
		Y^{-1}_\mu f=0 \qquad \text{ and } \qquad Y^{-2}_\mu f=0,
	\end{align}
	subject to \eqref{eq:GDAssume}, then at least one of the following is true:
	\begin{enumerate}
		\item $f=0$,
		\item $d\in\set{0,1}$,
		\item $d=2$ and $f$ is a multiple of $g_1$,
		\item $\mu=\paren{\frac{d-1}{2}+i t, -2i t, -\frac{d-1}{2}+i t}$, with either $d$ even or $t\ne 0$, and $f$ is a linear combination of $g_2^{d,0,\varepsilon}$ and $g_2^{d,1,-\varepsilon}$, with $\varepsilon=(-1)^d$,
		\item $\mu=\paren{\frac{d-1}{2}, 0, -\frac{d-1}{2}}$, $d$ odd, and $f$ is a linear combination of $g_2^{d,\delta,\varepsilon}$, $g_4^{d,+}$ and $g_4^{d,-}$, with $\delta\equiv\kappa+1\pmod{2}$, $\varepsilon=(-1)^\kappa$, $\kappa=\frac{d+1}{2}$,
		\item $\mu=\paren{d-1, 0, 1-d}$, and $f$ is a linear combination of $g_3^{d,1,\varepsilon}$, $g_3^{d,0,\varepsilon}$, $\bu^{d,+}_d$ and $\bu^{d,-}_d$, with $\varepsilon=(-1)^d$.
	\end{enumerate}
\end{prop}
We call such $f$ \textit{power-function minimal} for $\mu$ at weight $d$.

We note that $g_1$ in case 3 may be a false positive in the sense that it is killed by the lowering operators, but not necessarily an eigenfunction of $Y^0_\mu$, while the vectors of the other cases are all also eigenfunctions of $Y^0_\mu$.
(We will not show this directly, but it can be deduced from proposition \ref{prop:GDWhittFEs} and \eqref{eq:IntertwiningOpsY}.)
That is, if none of differences $\mu_i-\mu_j$ are $1$, then
\begin{align}
\label{eq:Ym2Y0g1}
	\sqrt{35} Y^{-2}_\mu Y^0_\mu g_1 = -8\sqrt{2}(\mu_1-\mu_2-1)(\mu_1-\mu_3-1)(\mu_2-\mu_3-1)\bu^{0,+}_0 \ne 0,
\end{align}
so the minimal weight of the corresponding principle series representation is $d=0$ instead of $d=2$, and of course, the instances where some $\mu_i-\mu_j=1$ reduce to cases 4 or 6.

\begin{proof}[Proof of proposition \ref{prop:GDPower}]
The result is trivial for $d\in\set{0,1}$, and $d=2$ is simple to compute (note that $Y^{-1}_\mu \bu^{2,\pm}_0$, $Y^{-1}_\mu\bu^{2,\pm}_2$ and $Y^{-2}_\mu\bu^{2,\pm}_1$ all vanish), so we may assume $d \ge 3$.

For clarity, we write $f^d :=f$.
We may split $f^d$ according to characters of $V$,
\begin{align*}
	f^d=&\sum_{\abs{m'}\le d} f^d_{m'} \bv^d_{m'} = \sum_{\pm,\delta\in\set{0,1}} f^{d,\delta,\pm}, \\
	f^{d,\delta,\pm}=&\sum_{2j+\delta\le d} \paren{f^d_{2j+\delta}\pm (-1)^d f^d_{-2j-\delta}} \bu^{d,\pm}_{2j+\delta} \times \piecewise{\frac{1}{2} & \If 2j+\delta=0, \\[5pt] 1 & \Otherwise,}
\end{align*}
and $f^d$ is a zero of both $Y^{-1}_\mu$ and $Y^{-2}_\mu$ exactly when all of the $f^{d,\delta,\pm}$ are, so we now fix $d$, $\mu$ and a choice of $\delta\in\set{0,1}$ and parity $\varepsilon=\pm1$ and assume $f^d$ is of the form
\begin{align*}
	f^d=&\sum_{0\le 2j+\delta\le d} f^d_{2j+\delta} \bu^{d,\varepsilon}_{2j+\delta},
\end{align*}
after relabelling the subscripts.
Let
\[ f^{d-1} = \sqrt{2d(d-1)(d+1)(2d+1)} Y^{-1}_\mu f^d \text{ and } f^{d-2} = 2\sqrt{d(d-1)(2d-1)(2d+1)} Y^{-2}_\mu f^d. \]
From \eqref{eq:FourTermReduction1} and \eqref{eq:FourTermReduction2}, their coefficients are given by
\begin{align}
\label{eq:Downfdm1}
	f^{d-1}_j =&-\sqrt{(d-1-j)(d+j)(d+1+j)(d+2+j)} (\mu_1-\mu_2-1-j) f^d_{j+2} \\
	&\qquad +2j\sqrt{(d-j)(d+j)} (3\mu_3+d) f^d_j \nonumber \\
	&\qquad+c_j \sqrt{(d-j)(d+1-j)(d+2-j)(d-1+j)} (\mu_1-\mu_2-1+j) f^d_{j-2} \nonumber \\
\label{eq:Downfdm2}
	f^{d-2}_j =&\sqrt{(d-1+j)(d+j)(d+1+j)(d+2+j)} (\mu_1-\mu_2-1-j) f^d_{j+2} \\
	&\qquad+2\sqrt{(d-1-j)(d-j)(d-1+j)(d+j)} (3\mu_3+2d-1) f^d_j \nonumber \\
	&\qquad+c_j \sqrt{(d-1-j)(d-j)(d+1-j)(d+2-j)} (\mu_1-\mu_2-1+j) f^d_{j-2}, \nonumber
\end{align}
for $j \ge 2$, where $c_j=2$ if $j=2$ and $1$ otherwise.
When $\varepsilon=-(-1)^d$, we set $f^d_0=0$.
Some care must be taken when dealing with the coefficients $f^{d-1}_j$, $f^{d-2}_j$, for $j=0,1$, as the previous expressions do not necessarily apply:
\begin{align}
\label{eq:Downfdm1j1}
	f^{d-1}_1 =&-\sqrt{(d-2)(d+1)(d+2)(d+3)} (\mu_1-\mu_2-2) f^d_3 \\
	&\qquad +\sqrt{(d-1)(d+1)} (2(3\mu_3+d)+\varepsilon(-1)^d d(\mu_1-\mu_2)) f^d_1, \nonumber \\
\label{eq:Downfdm2j1}
	f^{d-2}_1 =&\sqrt{d(d+1)(d+2)(d+3)} (\mu_1-\mu_2-2) f^d_3 \\
	&\qquad+\sqrt{d(d-2)(d-1)(d+1)} (2(3\mu_3+2d-1)+\varepsilon(-1)^d (\mu_1-\mu_2)) f^d_1, \nonumber \\
\label{eq:Downfdm1j0}
	f^{d-1}_0 =&-\sqrt{d(d-1)(d+1)(d+2)} (\mu_1-\mu_2-1) f^d_2 \, \delta_{\varepsilon=-(-1)^d}, \\
\label{eq:Downfdm2j0}
	f^{d-2}_0 =& \sqrt{d(d-1)}\Bigl(\sqrt{(d+1)(d+2)} (\mu_1-\mu_2-1) f^d_2\\
	& \qquad +2 \sqrt{d(d-1)} (3\mu_3+2d-1) f^d_0\Bigr)\delta_{\varepsilon=+(-1)^d}. \nonumber
\end{align}

Two useful linear combinations are
\begin{align}
\label{eq:Down3TRa}
	s_{1,j} :=& \tfrac{1}{2(d-1)\sqrt{d+j}}\paren{\sqrt{d-1-j} f^{d-1}_j-\sqrt{d-1+j} f^{d-2}_j}\\
	=&-\sqrt{(d+1+j)(d+2+j)} (\mu_1-\mu_2-1-j) f^d_{j+2} \nonumber \\
	&-\sqrt{(d-1-j)(d-j)} (3\mu_3+2d-1+j) f^d_j, \nonumber \\
\label{eq:Down3TRb}
	s_{2,j} :=& \tfrac{1}{2(d-1)\sqrt{d-j}}\paren{\sqrt{d-1+j} f^{d-1}_j+\sqrt{d-1-j} f^{d-2}_j}\\
	=&\sqrt{(d-1+j)(d+j)} (3\mu_3+2d-1-j) f^d_j \nonumber \\
	&+\sqrt{(d+1-j)(d+2-j)} (\mu_1-\mu_2-1+j) f^d_{j-2}, \nonumber
\end{align}
for $3 \le j \le d-2$.
The given expression \eqref{eq:Down3TRa} for $s_{1,j}$ continues to hold for $j=1,2$, and we have the following expressions for $s_{2,j}$, $j=2,1$,
\begin{align}
\label{eq:GDs22}
	s_{2,2} =&\sqrt{(d+1)(d+2)} (3\mu_3+2d-3) f^d_2 +2\sqrt{d(d-1)} (\mu_1-\mu_2+1) f^d_0, \\
\label{eq:GDs21}
	s_{2,1} =& -2f^d_1 \sqrt{d(d+1)} \times \piecewise{(\mu_2-\mu_3+1-d) & \If \varepsilon=+(-1)^d,\\ (\mu_1-\mu_3+1-d) & \If \varepsilon=-(-1)^d,}
\end{align}
but for $j=0$, it is simplest to use $f^{d-1}_0$ and $f^{d-2}_0$ directly.
Note that $s_{1,j}=s_{2,j}=0$ is fully equivalent to $f^{d-1}_j=f^{d-2}_j=0$ for each $1 \le j \le d-2$.

For $1 \le j \le d-4$, we further cancel to obtain
\begin{align}
\label{eq:Down3TRc}
	s_{3,j} :=&(3\mu_3+2d-3-j)s_{1,j}+(\mu_1-\mu_2-1-j) s_{2,j+2} \\
	=&-4 \sqrt{(d-j)(d-1-j)} (\mu_1-\mu_3+1-d)(\mu_2-\mu_3+1-d) f^d_j, \nonumber
\end{align}
and for $3 \le j \le d-2$, we have
\begin{align}
\label{eq:Down3TRd}
	s_{4,j} :=&(3\mu_3+2d-3-j)s_{1,j-2}+(\mu_1-\mu_2-1-j) s_{2,j} \\
	=&4 \sqrt{(d+j)(d-1+j)} (\mu_1-\mu_3+1-d)(\mu_2-\mu_3+1-d) f^d_j. \nonumber
\end{align}
Then
\begin{align}
\label{eq:GDPowReformulation}
	\left\{\begin{array}{c} f^{d-1}=0 \\ f^{d-2}=0 \end{array}\right\} \Leftrightarrow
	\begin{cases}
		 f^{d-1}_{d-1}=f^{d-1}_0=f^{d-2}_0=0,\\
		 s_{1,j}=0, j=1,\ldots, d-2,\\
		 s_{2,1}=s_{2,2}=0, \\
		 s_{3,j-2}=s_{4,j}=0, j=3,\ldots,d-2, \\
	\end{cases}
\end{align}
because for each $3 \le j \le d-2$ the coefficient of $s_{2,j}$ in one of $s_{3,j-2}$ or $s_{4,j}$ is non-zero.

The proof now proceeds by cases on $\mu$.

\subsubsection*{Case I:\nopunct} $\mu_1-\mu_3+1\ne d$, $\mu_2-\mu_3+1\ne d$ \\
These assumptions (recall \eqref{eq:GDAssume}) imply $\mu_1-\mu_2+1\ne d$, as well.
Then $s_{3,j}=0$ is equivalent to $f^d_j = 0$ for all $1 \le j \le d-4$, and similarly $s_{4,j}=0$ is equivalent to $f^d_j = 0$ on $3 \le j \le d-2$.
(Note that for $d \ge 6$, $\set{1 \le j \le d-4}\cup\set{3\le j\le d-2}=\set{1 \le j \le d-2}$, but the same does not hold for $3 \le d \le 5$.)

Now suppose $f^{d-1}=0$ and $f^{d-2}=0$ with $d \ge 6$.
Then \eqref{eq:Down3TRc}, \eqref{eq:Down3TRd} and \eqref{eq:GDs21} imply $f^d_j = 0$, $1 \le j \le d-2$.
In case $\varepsilon=+(-1)^d$ and $\delta=0$, then \eqref{eq:GDs22} implies $f^d_0=0$ since $\mu_1-\mu_2\ne-1$.

If $d\equiv\delta\pmod{2}$, then $s_{1,d-2}=0$ implies $f^d_d=0$ and we have $f^d=0$.
If $d\not\equiv\delta\pmod{2}$, then $f^{d-2}_{d-3}=f^{d-1}_{d-1}=0$ implies $f^d_{d-1}=0$ and we have $f^d=0$.

For $3 \le d \le 5$, first suppose $\delta=1$, and note \eqref{eq:GDs21} implies $f^d_1=0$.
Then $s_{1,1}=0$ implies $f^d_3=0$ unless $\mu_1-\mu_2=2$, in which case $\mu=(2,0,-2)$ (recall \eqref{eq:GDAssume}), so $d=4$ and $s_{2,3}=0$ again implies $f^d_3=0$.
Lastly, if $d=5$, then $s_{1,3}=0$ implies $f^d_5=0$.

Now suppose $3 \le d \le 5$ with $\delta=0$, then, as before, \eqref{eq:GDs22} implies $f^d_0=0$, and $f^{d-1}_0=f^{d-2}_0=0$ implies $f^d_2=0$ unless $\mu_1-\mu_2=1$ in which case \eqref{eq:GDs22} works since $\mu=(1,0,-1)$ implies $d\ne 3$.
When $d\in\set{4,5}$, $s_{1,2}=0$ implies $f^d_4=0$ unless $\mu_1-\mu_2=3$ in which case $f^{d-1}_4=0$ works since $d=5$ and $\mu_3=-3$.

\subsubsection*{Case II:\nopunct} $\mu=\paren{\frac{d-1}{2}+i t, -2i t, -\frac{d-1}{2}+i t}$ with $t\ne 0$ or $d$ even \\
Under the current assumptions, we have all $s_{3,j}=0$,$s_{4,j}=0$, and we use \eqref{eq:GDPowReformulation}.
When $\delta=0$, $s_{2,1}=0$ becomes trivial, and similarly, when $\delta=1$, $s_{2,2}=f^{d-1}_0=f^{d-2}_0=0$ becomes trivial.

Since $\mu_1-\mu_2-1\notin\Z$, $s_{1,j}=0$ is equivalent to
\begin{align}
\label{eq:IIRecurse}
	f^d_{j+2} =&-\frac{\sqrt{(d-1-j)(d-j)} (3\mu_3+2d-1+j)}{\sqrt{(d+1+j)(d+2+j)} (\mu_1-\mu_2-1-j)} f^d_j,
\end{align}
If $\delta=0$ and $\varepsilon=+(-1)^d$, then $s_{2,2}=0$ is redundant over $f^{d-2}_0=0$.
If $\delta=0$ and $\varepsilon=-(-1)^d$, then $f^{d-1}_0=0$ implies $f^d_2=0$ and hence $f^d=0$.
If $\delta=1$ and $\varepsilon=+(-1)^d$, then $s_{2,1}=0$ implies $f^d_1=0$ and hence $f^d=0$.
If $\delta=1$ and $\varepsilon=-(-1)^d$, then $s_{2,1}=0$ is trivial.

This completes the analysis of conclusion 4 of the proposition.
The need for $g_{2,0,0}^{d,\delta,\varepsilon}=\frac{1}{2}$ comes from the spare 2 in $f^{d-2}_0$, as compared to $s_{1,j}$, $j \ge 1$.

\subsubsection*{Case III:\nopunct} $\mu=\paren{d-1, 0, 1-d}$ \\
Since $\mu_1-\mu_2-1=-(3\mu_3+2d-1)=d-2$, this is identical to case II, except that $s_{1,d-2}=0$ is trivial, and now when $\delta=1$, $s_{2,1}=0$ implies $f^d_1=0$ unless $\varepsilon = +(-1)^d$.
The coefficient $g_{2,1,d}^d$ is undefined, but we may freely choose the coefficient of $\bu^{d,\varepsilon}_d$ since $Y^{-1}_\mu \bu^{d,\varepsilon}_d=0$ and $Y^{-2}_\mu \bu^{d,\varepsilon}_d = 0$.
Our particular choice \eqref{eq:g3LastCoefDef} makes $g_3^{d,\delta,\varepsilon}$ an eigenfunction of $Y^0_\mu$, as we will see later.
This completes the analysis of conclusion 6 of the proposition.

\subsubsection*{Case IV:\nopunct} $\mu=\paren{\frac{d-1}{2}, 0, -\frac{d-1}{2}}$, $d$ odd \\
Since $\mu_1-\mu_2-1=\kappa-2$, the case $\delta \not\equiv \kappa \pmod{2}$ is identical to case II.
In case $\delta \equiv \kappa\pmod{2}$, $s_{1,\kappa-2}=0$ implies $f^d_{\kappa-2}=0$, and recursively $s_{1,j}=0$ implies $f^d_j=0$ for $1 \le j < \kappa$ and $s_{2,2}=0$ implies $f^d_0=0$.
Then \eqref{eq:IIRecurse} applies for $j \ge \kappa$, and this completes the analysis of conclusion 5 of the proposition.
\end{proof}

\subsection{For Whittaker functions}
\label{sect:GDWhitt}
We wish to prove a version of proposition \ref{prop:GDPower} for Whittaker functions, but this is somewhat complicated by the possibility that the Whittaker function itself may be zero, and the unpleasant shape of the vectors in proposition \ref{prop:GDPower}, so we start by a careful examination of the intertwining operators.
\begin{prop}
\label{prop:GDWhittFEs}
	Suppose $d \ge 2$, and let $C$ be a non-zero constant, depending on $d$ and $\mu$, whose value may differ between occurences.
	\begin{enumerate}
		\item If $\mu=\paren{\frac{d-1}{2}+i t, -2i t, -\frac{d-1}{2}+i t}$, with either $d$ even or $t\ne 0$, then
			\begin{align*}
				g_2^{d,0,\varepsilon} =& C \bu^{d,\varepsilon}_d T^d(w_3,\mu^{w_3}),&
				g_2^{d,1,-\varepsilon} =& C \bu^{d,-\varepsilon}_d T^d(w_3,\mu^{w_3}),
			\end{align*}
			with $\varepsilon=(-1)^d$.
		\item If $\mu=\paren{\frac{d-1}{2}, 0, -\frac{d-1}{2}}$ with $d\equiv 1 \pmod{4}$, then
			\begin{align*}
				g_2^{d,0,-} =& C \bu^{d,-}_d T^d(w_3,\mu^{w_3}),&
				g_4^{d,+} =& C \bu^{d,+}_d T^d(w_3,\mu^{w_3}),&
				g_4^{d,-} =& C \bu^{d,+}_d T^d(w_4,\mu^{w_5}),
			\end{align*}
		\item If $\mu=\paren{\frac{d-1}{2}, 0, -\frac{d-1}{2}}$ with $d\equiv 3 \pmod{4}$, then
			\begin{align*}
				g_2^{d,1,+} =& C \bu^{d,+}_d T^d(w_3,\mu^{w_3}),&
				g_4^{d,-} =& C \bu^{d,-}_d T^d(w_3,\mu^{w_3}), &
				g_4^{d,+} =& C \bu^{d,+}_{d-1} T^d(w_4,\mu^{w_5}),&
			\end{align*}
		\item If $\mu=\paren{d-1, 0, 1-d}$, then
			\begin{align*}
				g_3^{d,1,\varepsilon} =& C \bu^{d,\varepsilon}_{d-1} T^d(w_3,\mu^{w_3}),&
				g_3^{d,0,\varepsilon} =& C \bu^{d,\varepsilon}_d T^d(w_3,\mu^{w_3}),
			\end{align*}
		with $\varepsilon=(-1)^d$.
	\end{enumerate}
\end{prop}
Note: We have supressed the constants for purely aesthetic reasons, their values may be extracted from the computations below, but they are irrelevant in the context of the current paper.
To be precise, the proof proceeds by comparing ratios of successive coefficients, and its use in theorem \ref{thm:GDCusp} is such that the non-zero constant is simply included in the (scalar) Fourier-Whittaker coefficients.

We note that some care must be taken in the use of the intertwining operators when two of the complex parameters simultaneously encounter a singularity.
This occurs, e.g. in case 3 of the proposition when considering $T^d(w_4,\mu^{w_5})$ since both $\mu_1-\mu_3$ and $\mu_1-\mu_2$ are integral, and we consider this case in particular as an example at the end of the proof.

\begin{proof}[Proof of proposition \ref{prop:GDWhittFEs}]
We wish to compute
\[ \bu^{d,\pm}_d T^d(w_3,\mu^{w_3}), \quad \bu^{d,\pm}_{d-1} T^d(w_3,\mu^{w_3}), \quad \bu^{d,+}_d T^d(w_4,\mu^{w_5}), \quad \bu^{d,-}_{d-1} T^d(w_4,\mu^{w_5}) \]
for all $d$ and $\mu$.
We give the general proceedure first, then illustrate with an example below.

Note that
\[ T^d(w_4,\mu^{w_5}) = T^d(w_3,\mu^{w_5}) T^d(w_2,\mu^{w_2}), \]
and $T^d(w_2,\mu^{w_2})$ is diagonal and satisfies a symmetry relation under $(m,m)\mapsto(-m,-m)$, so it suffices to compute
\[ \bu^{d,\pm}_d \WigDMat{d}(w_l) \Gamma^d_\mathcal{W}(u,+1) \WigDMat{d}(w_l), \qquad  \bu^{d,\pm}_{d-1} \WigDMat{d}(w_l) \Gamma^d_\mathcal{W}(u,+1) \WigDMat{d}(w_l). \]

There is a coincidence of form among the intertwining operators $T^d(w,\mu)$, the constant terms of the Eisenstein series and the Whittaker functions at a degenerate character:
Superficially, this is because they are all generated by compositions (compare \eqref{eq:TdCompose} and (I.4.12)) of analogous diagonal matrices (compare (I.2.18), \eqref{eq:Tdw2}, (I.2.20), the definition of $M^d$ in section I.4.4.1 and (I.3.11)).
More fundamentally, the fact that composition of the constant terms of the minimal parabolic Eisenstein series, which are the Whittaker functions $\Sigma^d_{\chipmpm{++}} W^d(I,w,\mu,\psi_{00})$ (see (I.3.3)), gives the functional equations of the Eisenstein series, and hence also the functional equations of its non-degenerate Fourier coefficients, which are the Whittaker functions $\Sigma^d_{\chipmpm{++}} W^d(g,\mu,\psi_{1,1})$, implies the relationship precisely, at least in the case $\chi=\chipmpm{++}$.
We use this commonality here in the form
\begin{align}
\label{eq:GammaWtoClassWhitt}
	&\Gamma_\mathcal{W}^d(u,+1) = \\
	& \frac{i\,\Gamma(1+u)}{2^{1+u}\pi} \paren{\exp\paren{\tfrac{i\pi u}{2}} \WigDMat{d}(\vpmpm{++})-\exp\paren{-\tfrac{i\pi u}{2}} \WigDMat{d}(\vpmpm{-+})} \Dtildek{d}{-i}\mathcal{W}^d(0,-u). \nonumber
\end{align}
In fact, from \eqref{eq:WigDv}, \eqref{eq:WigDPrimary} and \eqref{eq:classWhittDef}, the (diagonal) entry at row $m$ of the right-hand side is
\[ 2i^m\frac{\Gamma(1+u)\Gamma(-u)}{\Gamma\paren{\frac{1-u+m}{2}} \Gamma\paren{\frac{1-u-m}{2}}} \times \piecewise{-\sin\paren{\pi\frac{u}{2}} & \text{for even }m,\\ i \cos\paren{\pi\frac{u}{2}} & \text{for odd }m,} \]
and the reflection formula for the gamma function shows the quotient formed by dividing this by $\Gamma^d_{\mathcal{W},m,m}(u,+1)$ is 1.

From \eqref{eq:w2toRi} and \eqref{eq:WigDv}, we see that conjugating a diagonal matrix by $\WigDMat{d}(w_2)$ simply reverses the order of the diagonal entries.
So using \eqref{eq:GammaWtoClassWhitt}, \eqref{eq:classWhittDef} and the results of section \ref{sect:WeylV} (or by the reflection formula, see (I.2.20)), we can see
\begin{align*}
	\WigDMat{d}(w_2)\Gamma_\mathcal{W}^d(u,+1)\WigDMat{d}(w_2) =& \WigDMat{d}(\vpmpm{-+}) \Gamma_\mathcal{W}^d(u,+1).
\end{align*}
Applying this back in \eqref{eq:GammaWtoClassWhitt}, we have
\begin{align}
\label{eq:TdtoFd}
	&\WigDMat{d}(\vpmpm{--}w_l)\Gamma_\mathcal{W}^d(u,+1) \WigDMat{d}(w_l\vpmpm{--}) =\\
	& \frac{i\,\Gamma(1+u)}{2^{1+u}\pi} \Dtildek{d}{i} \paren{\exp\paren{\tfrac{i\pi u}{2}} \WigDMat{d}(\vpmpm{+-})-\exp\paren{-\tfrac{i\pi u}{2}} \WigDMat{d}(\vpmpm{-+})} F^d(u) \Dtildek{d}{i}, \nonumber
\end{align}
\begin{align}
	F^d(u) :=& \WigDMat{d}(w_3) \Dtildek{d}{-i}\mathcal{W}^d(0,-u) \WigDMat{d}(w_3).
\end{align}

So we need to compute the first and last two rows of the matrix-valued function
\begin{align*}
	F^d(u) =& \int_{-\infty}^\infty (1+x^2)^{\frac{-1+u}{2}} \WigDMat{d}\paren{\tildek{1,\frac{x+i}{\sqrt{1+x^2}},1}} dx,
\end{align*}
(using $\tildek{e^{i\alpha},e^{i\beta},e^{i\gamma}}=k(\alpha,\beta,\gamma)$ as in (I.2.5)) which has components
\begin{align*}
	F^d_{m',m}(u) = \int_{-1}^1 (1-x^2)^{-1-\frac{u}{2}} \Wigd{d}{m'}{m}(x) dx,
\end{align*}
recalling \eqref{eq:kabcDef}, \eqref{eq:WigDPrimary}, \eqref{eq:classWhittDef} and \eqref{eq:WigdDef}.
For the moment, we assume $-\Re(u) > 1$ so this and the subsequent integrals converge.

We use \cite[eq. 3.196.3]{GradRyzh} in the form
\begin{align}
	\int_{-1}^1 (1-x)^{a-1} (1+x)^{b-1} dx =& 2^{a+b-1} B(a,b), & \Re(a),\Re(b)>0,
\end{align}
where $B(a,b)$ is again the Euler beta function.
Then from \eqref{eq:WigdtoJacobi} and \eqref{eq:JacPBase}, we have
\begin{align}
\label{eq:FddEval}
	F^d_{m',d}(u) =& 2^{-d} \pi^{\frac{1}{2}} \sqrt{\frac{(2d)!}{(d+m)!\,(d-m)!}} \frac{\Gamma\paren{\frac{d-m'-u}{2}}\Gamma\paren{\frac{d+m'-u}{2}}}{\Gamma\paren{\frac{d-u}{2}}\Gamma\paren{\frac{d+1-u}{2}}}.
\end{align}
When $m=d-1$, \eqref{eq:JacPBase} becomes $\JacobiP{1}{d-1-m'}{d-1+m'}(x) = d x-m'$, so
\begin{align}
\label{eq:Fddm1Eval}
	\frac{F^d_{m',d-1}(u)}{2^{1-d} \sqrt{\frac{(2d-1)!}{(d+m')! \, (d-m')!}}} =& d \int_{-1}^1 \paren{1-x}^{-1+\frac{d-1-m'-u}{2}} \paren{1+x}^{-1+\frac{d+1+m'-u}{2}} dx \\
	& - d\int_{-1}^1 \paren{1-x}^{-1+\frac{d-1-m'-u}{2}} \paren{1+x}^{-1+\frac{d-1+m'-u}{2}} dx \nonumber\\
	& -m' \int_{-1}^1 \paren{1-x}^{-1+\frac{d-1-m'-u}{2}} \paren{1+x}^{-1+\frac{d-1+m'-u}{2}} dx \nonumber\\
	=& -m'\pi^{\frac{1}{2}} \frac{\Gamma\paren{\frac{d-1-m'-u}{2}}\Gamma\paren{\frac{d-1+m'-u}{2}}\Gamma\paren{\frac{1-u}{2}}}{\Gamma\paren{\frac{d-u}{2}}\Gamma\paren{\frac{d+1-u}{2}}\Gamma\paren{-\frac{u+1}{2}}}, \nonumber
\end{align}
using the usual recurrence relation of the gamma function.

We may deduce from the symmetries \eqref{eq:WigdSymms} with \eqref{eq:FddEval} and \eqref{eq:Fddm1Eval} that
\begin{align*}
	F^d_{\pm d,m}(u) =& (\mp 1)^{d+m} F^d_{\pm m,d}(u)= (\mp 1)^{d+m} F^d_{\abs{m},d}(u), \\
	F^d_{\pm (d-1),m}(u) =& (\mp 1)^{d-1+m} F^d_{\pm m,d-1}(u) = - \sgn(m) (\mp 1)^{d+m} F^d_{\abs{m},d-1}(u),
\end{align*}
and it follows that for $\delta\in\set{0,1}$
\begin{align}
\label{eq:ITOpsSignReduct1}
	\bu^{d,(-1)^\delta}_d F^d(u) =& \tfrac{1}{2} F^d_d(u) \pm(-1)^d \tfrac{1}{2} F^d_{-d}(u) = 2 \sum_{0\le m \equiv \delta} c_m F^d_{m,d}(u) (-1)^{d+m} \bu^{d,(-1)^d}_m \\
\label{eq:ITOpsSignReduct2}
	\bu^{d,(-1)^\delta}_{d-1} F^d(u) =& 2\sum_{0\le m \equiv 1-\delta} c_m F^d_{m,d-1}(u) (-1)^{d-1+m} \bu^{d,(-1)^{d-1}}_m
\end{align}
where $c_0=\frac{1}{2}$, $c_m=1,m\ne0$.

Now if we define $\delta,\eta\in\set{0,1}$ for a choice of the parity $\pm$ by $\pm1=(-1)^{d+\delta}=(-1)^\eta$, then \eqref{eq:TdtoFd}, \eqref{eq:buSignReduct} and \eqref{eq:ITOpsSignReduct1} imply
\begin{align}
\label{eq:uddw3FE}
	& \bu^{d,\pm}_d \WigDMat{d}(\vpmpm{--}w_l)\Gamma_\mathcal{W}^d(u,+1) \WigDMat{d}(w_l\vpmpm{--}) \\
	&= \pm(-1)^d \frac{i^{1-d}\Gamma(1+u)}{2^{1+u}\pi} \paren{\exp\paren{\tfrac{i\pi u}{2}} \mp \exp\paren{-\tfrac{i\pi u}{2}}} \bu^{d,\pm(-1)^d}_d F^d(u) \WigDMat{d}(w_2) \nonumber \\
	&= \frac{i^{d+\eta}}{2^{d-1}} \frac{\Gamma\paren{\frac{1+\eta+u}{2}}}{\Gamma\paren{\frac{\eta-u}{2}}} \sum_{0\le m \equiv \delta} c_m \bu^{d,\pm}_m i^{-m} \sqrt{\frac{(2d)!}{(d+m)!\,(d-m)!}} \frac{\Gamma\paren{\frac{d-m-u}{2}}\Gamma\paren{\frac{d+m-u}{2}}}{\Gamma\paren{\frac{d-u}{2}}\Gamma\paren{\frac{d+1-u}{2}}}. \nonumber
\end{align}
Here, we have used the reflection and duplication properties of the gamma function to simplify the leading coefficient, but again, this value is not actually relevant to the rest of the paper.
Similarly
\begin{align}
\label{eq:uddm1w3FE}
	& \bu^{d,\pm}_{d-1} \WigDMat{d}(\vpmpm{--}w_l) \Gamma_\mathcal{W}^d(u,+1) \WigDMat{d}(w_l\vpmpm{--}) \\
	&= \frac{i^{d+1+\eta}}{2^{d-2}} \frac{\Gamma\paren{\frac{1+\eta+u}{2}}\Gamma\paren{\frac{1-u}{2}}}{\Gamma\paren{\frac{\eta-u}{2}}\Gamma\paren{-\frac{u+1}{2}}} \sum_{0\le m \equiv 1-\delta} c_m \bu^{d,\pm}_m i^{-m} m \nonumber \\
	& \qquad \times \sqrt{\frac{(2d-1)!}{(d+m)! \, (d-m)!}} \frac{\Gamma\paren{\frac{d-1-m-u}{2}}\Gamma\paren{\frac{d-1+m-u}{2}}}{\Gamma\paren{\frac{d-u}{2}}\Gamma\paren{\frac{d+1-u}{2}}}, \nonumber
\end{align}
and using
\[ \bu^{d,\pm}_m \Gamma_\mathcal{W}^d(u,+1) = \frac{\Gamma\paren{\frac{1-m+u}{2}}}{\Gamma\paren{\frac{1-m-u}{2}}} \bu^{d,\pm(-1)^m}_m, \qquad 0 \le m \le d \]
(which follows from the reflection formula, see (I.2.20)), we have
\begin{align}
\label{eq:uddw4FE}
	& \bu^{d,\pm}_d \WigDMat{d}(\vpmpm{--}w_l) \Gamma_\mathcal{W}^d(u_1,+1) \WigDMat{d}(w_l\vpmpm{--}) \Gamma_\mathcal{W}^d(u_2,+1) \\
	&= \frac{i^{d+\eta}}{2^{d-1}} \frac{\Gamma\paren{\frac{1+\eta+u_1}{2}}}{\Gamma\paren{\frac{\eta-u_1}{2}}} \sum_{0\le m \equiv \delta} c_m \bu^{d,\varepsilon}_m i^{-m} \nonumber \\
	& \qquad \times \sqrt{\frac{(2d)!}{(d+m)!\,(d-m)!}} \frac{\Gamma\paren{\frac{d-m-u_1}{2}}\Gamma\paren{\frac{d+m-u_1}{2}}}{\Gamma\paren{\frac{d-u_1}{2}}\Gamma\paren{\frac{d+1-u_1}{2}}} \frac{\Gamma\paren{\frac{1-m+u_2}{2}}}{\Gamma\paren{\frac{1-m-u_2}{2}}}, \nonumber
\end{align}
\begin{align}
\label{eq:uddm1w4FE}
	& \bu^{d,\pm}_{d-1} \WigDMat{d}(\vpmpm{--}w_l)\Gamma_\mathcal{W}^d(u_1,+1) \WigDMat{d}(w_l\vpmpm{--}) \Gamma_\mathcal{W}^d(u_2,+1) \\
	&= \frac{i^{d+1+\eta}}{2^{d-2}} \frac{\Gamma\paren{\frac{1+\eta+u_1}{2}}\Gamma\paren{\frac{1-u_1}{2}}}{\Gamma\paren{\frac{\eta-u_1}{2}}\Gamma\paren{-\frac{u_1+1}{2}}} \sum_{0\le m \equiv 1-\delta} c_m \bu^{d,-\varepsilon}_m i^{-m} m \nonumber \\
	& \qquad \times \sqrt{\frac{(2d-1)!}{(d+m)! \, (d-m)!}} \frac{\Gamma\paren{\frac{d-1-m-u_1}{2}}\Gamma\paren{\frac{d-1+m-u_1}{2}}}{\Gamma\paren{\frac{d-u_1}{2}}\Gamma\paren{\frac{d+1-u_1}{2}}} \frac{\Gamma\paren{\frac{1-m+u_2}{2}}}{\Gamma\paren{\frac{1-m-u_2}{2}}}, \nonumber
\end{align}
where $\varepsilon=(-1)^d$.

One can check that the parities $\pm$ and $\delta$ match between the preceeding formulas and the claims of the theorem.
Note that when $u=\mu_2-\mu_3=\frac{d-1}{2}$ is an integer with $u\equiv \eta\equiv d-\delta \pmod{2}$, then the coefficients of $\bu^{d,\pm}_m$ are zero unless $d-m-u\le 0$ in the first formula or $d-1-m-u\le0$ in the second by the poles of the gamma functions; here we are not directly evaluating the gamma functions at a value of $u$, but taking the value of the whole meromorphic function at $u$.
One may compute the ratio of the coefficient of $\bu^{d,\cdot}_{m+2}$ to the coefficient of $\bu^{d,\cdot}_{m}$ in \eqref{eq:uddw3FE}-\eqref{eq:uddm1w4FE}, giving
\begin{align}
	& -\sqrt{\frac{(d-m) (d-1-m)}{(d+2+m)(d+1+m)}} \frac{(d-u+m)}{(d-2-u-m)},\\
	& -\sqrt{\frac{(d-m) (d-1-m)}{(d+2+m)(d+1+m)}} \frac{(m+2)(d-1-u+m)}{m(d-3-u-m)},\\
	& -\sqrt{\frac{(d-m) (d-1-m)}{(d+2+m)(d+1+m)}} \frac{(d-u_1+m)(-1-m-u_2)}{(d-2-u_1-m)(-1-m+u_2)},\\
\label{eq:uddm1w4FErat}
	& -\sqrt{\frac{(d-m) (d-1-m)}{(d+2+m)(d+1+m)}} \frac{(m+2)(d-1-u_1+m)(-1-m-u_2)}{m(d-3-u_1-m)(-1-m+u_2)},
\end{align}
respectively, with the understanding that the ratio is to be multiplied by 2 when $m=0$ (to accomodate $c_2/c_0=2$).
The proposition then follows by applying the explicit form of $\mu$ in each case with $u=\mu_2-\mu_3$, $u_1=\mu_1-\mu_3$, $u_2=\mu_1-\mu_2$.

As an example, consider case 3 of the proposition where $\mu=\paren{\frac{d-1}{2},0,-\frac{d-1}{2}}$ with $d\equiv 3\pmod{4}$.
Using the table of $\mu^w$ given in section I.2.1 and the definition \eqref{eq:Tdw2}, \eqref{eq:Tdw3} and \eqref{eq:TdCompose} of $T^d(w,\mu)$, we have
\begin{align*}
	\bu^{d,+}_{d-1} T^d(w_4,\mu^{w_5}) =& \bu^{d,+}_{d-1} T^d(w_3,\mu^{w_5}) T^d(w_2,\mu^{w_2}) \\
	=& \pi^{-\frac{3}{2}(d-1)} \bu^{d,+}_{d-1} \WigDMat{d}(\vpmpm{--}w_l)\Gamma_\mathcal{W}^d(d-1,+1) \WigDMat{d}(w_l\vpmpm{--}) \Gamma_\mathcal{W}^d(\tfrac{d-1}{2},+1)
\end{align*}
The coefficients of $\bu^{d,+}_m$ in $g_4^{d,+}$ (recall \eqref{eq:g4def}) are supported on $\frac{d+1}{2} \le m \equiv \frac{d+1}{2} \equiv 0 \pmod{2}$, and we compare this to \eqref{eq:uddm1w4FE} with $u_1=d-1$, $u_2=\frac{d-1}{2}$ (or rather, the analytic continuation to this point), $\delta=1$, $\eta=0$ and $\varepsilon=-1$.
As mentioned in the previous paragraph, \eqref{eq:uddm1w4FE} has a removable singularity at this $(u_1,u_2)$-point, and the summand is zero unless $m \ge \frac{d+1}{2}$, so the support of the coefficients matches that of $g_4^{d,+}$.
The ratio of successive coefficients \eqref{eq:uddm1w4FErat} reduces to
\begin{align*}
	& \sqrt{\frac{(d-m) (d-1-m)}{(d+2+m)(d+1+m)}} \frac{(m+\frac{d+1}{2})}{(m-\frac{d-3}{2})},
\end{align*}
and this matches the ratio $g_{2,\kappa,2j+\kappa+2}^d/g_{2,\kappa,2j+\kappa}^d$ with $m=2j+\kappa$.
So we conclude $g_4^{d,+}$ and $\bu^{d,+}_{d-1} T^d(w_4,\mu^{w_5})$ are the same, up to a non-zero constant, and since $g^d_{2,\kappa,0}=1$, we conclude the value of the constant is the coefficient of $\bu^{d,+}_m$ in $\bu^{d,+}_{d-1} T^d(w_4,\mu^{w_5})$ at $m=\kappa$, which is
\[ C = \pi^{-\frac{3}{2}(d-1)} 2^{3-2d} d! \sqrt{\frac{(2d-1)! \, (\kappa-1)!}{(3\kappa-1)!}}. \]
Here we have written the coefficient $C=C(0,0)$ of $\bu^{d,+}_\kappa$ in \eqref{eq:uddm1w4FE} by taking $u_1=d-1+a_1$ and $u_2=\frac{d-1}{2}+a_2$ and applying the reflection formula for the gamma functions (keeping in mind $d\equiv 3\pmod{4}$) so that
\[ C(a_1,a_2) =\frac{2^{1-2d-2a_1} (d+1)}{\pi^{\frac{3}{2}(d-1)+a_1+a_2}} \sqrt{\frac{(2d-1)!}{(\kappa-1)! \, (3\kappa-1)!}} \frac{\Gamma\paren{\frac{\kappa-a_1}{2}}\Gamma\paren{\frac{d+1+a_2}{2}}\Gamma\paren{d+1+a_1}}{\Gamma\paren{1-a_1}\Gamma\paren{\frac{2-b}{2}}\Gamma\paren{\frac{\kappa+2+2a}{2}}}, \]
and this expression holds in a neighborhood of $(a_1,a_2)=(0,0)$, as desired.
Of course, $\bu^{d,+}_{d-1} T^d(w_4,\mu^{w_5})/C(a_1,a_2)$ may be defined in terms of the ratios \eqref{eq:uddm1w4FErat} (including the terms with $m < \kappa$), and in this way our expression for $\bu^{d,+}_{d-1} T^d(w_4,\mu^{w_5})$ is holomorphic in a neighborhood of $(a_1,a_2)=(0,0)$.
\end{proof}

We must clarify precisely when the Whittaker function is identically zero, and we begin with the minimal $K$-types, as described in proposition \ref{prop:GDPower}.
The following proposition follows from theorem \ref{thm:MinWhitt} and the results of section \ref{sect:WhittCentralEntry}:
\begin{prop}
\label{prop:MinimalWhittsZero}
	Suppose $f$ is power-function minimal for $\mu$ at weight $d$.
	\begin{enumerate}
		\item If $d=0$, then $f W^d(\cdot,\mu,\psi_{1,1})$ is identically zero iff $f=0$.
		\item If $d=1$ and all $\mu_i$ are distinct, then $f W^d(\cdot,\mu,\psi_{1,1})$ is identically zero iff $f=0$.
		\item If $\mu=\paren{\frac{d-1}{2}+i t, -2i t, -\frac{d-1}{2}+i t}$ with $d\ge 1$, and either $d$ even or $t \ne 0$, then $f$ is a linear combination of the vectors $\bv^d_{\pm d} T^d(w_3,\mu^{w_3})$, and:\\
		$f W^d(\cdot,\mu,\psi_{1,1})$ is identically zero iff $f$ is a multiple of $\bv^d_d T^d(w_3,\mu^{w_3})$.
		\item If $\mu=\paren{\frac{d-1}{2}, 0, -\frac{d-1}{2}}$ with $d$ odd, then $f$ is a linear combination of the vectors $g_4^{d,+}$,$g_4^{d,-}$ and $(g_4^{d,+}+ g_4^{d,-})\WigDMat{d}(\vpmpm{--} w_l)$, and:\\
		$f W^d(\cdot,\mu,\psi_{1,1})$ is identically zero iff $f$ is a linear combination of $g_4^{d,+}+g_4^{d,-}$ and $(g_4^{d,+}+g_4^{d,-})\WigDMat{d}(\vpmpm{--} w_l)$.
		\item If $\mu=\paren{d-1, 0, 1-d}$ and $d \ge 1$, then $f$ a linear combination of the four vectors $\bv^d_{\pm d}$ and $\bv^d_{\pm d}\WigDMat{d}(\vpmpm{--} w_l)$, and:\\
		$f W^d(\cdot,\mu,\psi_{1,1})$ is identically zero iff $f$ is a linear combination of $\bv^d_d$ and $\bv^d_d\WigDMat{d}(\vpmpm{--} w_l)$.
	\end{enumerate}
\end{prop}
Note that in case $d=1$, $g_4^\pm = \bu^{1,\pm}_1$, so there is no inconsistency between 4 and 5 in that case.
\begin{proof}[Proof of proposition \ref{prop:MinimalWhittsZero}]
To see case 2, we note that the eigenvalues under $Y^0$ are as follows:
\begin{align}
\label{eq:d1Y0Eigen}
	Y^0_\mu \bu^{1,-}_0 = -2\sqrt{\tfrac{3}{5}}\mu_3 \bu^{1,-}_0, \quad Y^0_\mu \bu^{1,-}_1 = -2\sqrt{\tfrac{3}{5}}\mu_2 \bu^{1,-}_1, \quad Y^0_\mu \bu^{1,+}_1 = -2\sqrt{\tfrac{3}{5}}\mu_1 \bu^{1,+}_1,
\end{align}
and we may see directly from theorem \ref{thm:MinWhitt} that each of the associated Whittaker functions is non-zero.

In case 3, we note that
\[ \bv^d_{-d} T^d(w_3,\mu^{w_3}) W^d(g,\mu,\psi_{1,1}) \ne 0, \qquad \bv^d_d T^d(w_3,\mu^{w_3}) W^d(g,\mu,\psi_{1,1}) = 0, \]
and by proposition \ref{prop:GDWhittFEs}, those two vectors are a basis of the required space.

In case 4, we note that
\begin{align*}
	\bv^d_{-d} T^d(w_3,\mu^{w_3}) W^d(\cdot,\mu,\psi_{1,1}) &\ne 0, \\
	(g_4^{d,+}+g_4^{d,-}) W^d(g,\mu,\psi_{1,1}) =& 0.
\end{align*}
The second equality follows because the vector $g_4^{d,+}+g_4^{d,-}$ is supported on $\bv^d_m$ with
\[ m\equiv\kappa\equiv 1+\mu_1-\mu_2 \pmod{2}, \qquad m \ge \kappa = 1+\mu_1+\mu_2, \]
and again the gamma function in \eqref{eq:classWhittDef} has a pole there (recall the discussion at the end of section \ref{sect:WhittCentralEntry}).
Then by duality, i.e. \eqref{eq:WhittDuality} and \eqref{eq:YamuDuality},
\[ (g_4^{d,+}+g_4^{d,-})\WigDMat{d}(\vpmpm{--} w_l) W^d(g,\mu,\psi_{1,1}) = (g_4^{d,+}+g_4^{d,-})W^d(\vpmpm{--}g^\iota w_l,\mu,\psi_{1,1}) = 0, \]
and $(g_4^{d,+}+g_4^{d,-})\WigDMat{d}(\vpmpm{--} w_l)$ is still power-function minimal.
Again, these are three linearly independent vectors (they have different parities, i.e. live in the rowspace of $\Sigma^d_\chi$ for different characters $\chi$) in a three dimensional space.

In case 5, we have
\[ \bv^d_{-d} W^d(\cdot,\mu,\psi_{1,1}) \ne 0, \qquad \bv^d_d W^d(\cdot,\mu,\psi_{1,1}) = 0, \]
and by duality
\[ \bv^d_{-d}\WigDMat{d}(\vpmpm{--} w_l) W^d(\cdot,\mu,\psi_{1,1}) \ne 0, \qquad \bv^d_d \WigDMat{d}(\vpmpm{--} w_l) W^d(\cdot,\mu,\psi_{1,1}) = 0. \]
It is easy to see that the vectors $\bv^d_{-d}$ and $\bv^d_{-d} \WigDMat{d}(\vpmpm{--}w_l)$ are linearly independent; however, we require something a bit stronger:
We need to know that any non-zero linear combination
\begin{align}
\label{eq:GDWhittCase5LinComb}
	(a_1 \bv^d_{-d}+a_2\bv^d_{-d} \WigDMat{d}(\vpmpm{--}w_l)) W^d(\cdot,\mu,\psi_{1,1})
\end{align}
with $a_1 a_2 \ne 0$ yields a function that is not identically zero.
Notice that
\begin{align}
\label{eq:BadWhittEigen}
	\sqrt{(d+1)(2d+3)} Y^0_\mu \bv^d_{-d}=-(d-1)\sqrt{6d(2d-1)}\bv^d_{-d},
\end{align}
and by duality $\bv^d_{-d} \WigDMat{d}(\vpmpm{--} w_l)$ is also an eigenfunction of $Y^0_\mu$, but its eigenvalue has the opposite sign, and this is sufficient for our purposes.
(Consider applying the operators $\sqrt{(d+1)(2d+3)} Y^0 \pm(d-1)\sqrt{6d(2d-1)}$ to \eqref{eq:GDWhittCase5LinComb}.)

\end{proof}

\begin{prop}
\label{prop:GDWhitt}
	If $f\in\C^{2d+1}$ is such that
	\begin{align}
		\label{eq:GDWhittLowering}
		f W^d(g,\mu,\psi_{1,1})\ne0, \qquad Y^{-1} f W^d(g,\mu,\psi_{1,1})=0, \qquad Y^{-2} f W^d(g,\mu,\psi_{1,1})=0,
	\end{align}
	subject to \eqref{eq:GDAssume}, then $f$ is power-function minimal.
\end{prop}
\begin{proof}
The proposition is trivial by lemma \ref{eq:WhittLinIdepn} when none of the differences $\mu_i-\mu_j, i\ne j$ are integers, and the case $\mu=\paren{\frac{\kappa-1}{2}+it,-2it,-\frac{\kappa-1}{2}+it}$ for some $\kappa \ge 1$ with $\kappa$ even or $t \ne 0$, is also relatively simple, after applying the $w_3$ functional equation of the Whittaker functions (as we may).
So suppose $\mu=(\kappa-1,0,1-\kappa)$ for some $\kappa \ge 1$.

For any $d$, define the spaces
\begin{align*}
	V^d_{\delta} =& \Span\set{\bv^{d}_j \setdiv j\equiv\delta\pmod{2}}, \\
	V^d_{\delta,\kappa} =& \Span\set{\bv^{d}_j \setdiv \abs{j} \ge \kappa, j\equiv\delta\pmod{2}}, \\
	V^d_{\delta,\kappa,\pm} =& \Span\set{\bv^{d}_j \setdiv \pm j \ge \kappa, j\equiv\delta\pmod{2}}.
\end{align*}
It follows immediately from the argument of proposition \ref{prop:GDPower} that for $f$ which is not already power-function minimal,
\begin{align}
\label{eq:GDonVkappa}
	f \in V^d_\kappa, \quad Y^{-1}_\mu f \in V^{d-1}_{\kappa,\kappa}, \quad Y^{-2}_\mu f \in V^{d-2}_{\kappa,\kappa} \; \Rightarrow \; f \in V^d_{\kappa,\kappa},
\end{align}
and noticing that the $Y^a_\mu$ operators act on $\bv^d_j$ precisely the same as on $\bu^{d,\pm}_j$ for $j \ge 3$, we also have
\begin{align}
\label{eq:GDonVkappapm}
	f \in V^d_\kappa, \quad Y^{-1}_\mu f \in V^{d-1}_{\kappa,\kappa,\pm}, \quad Y^{-2}_\mu f \in V^{d-2}_{\kappa,\kappa,\pm} \; \Rightarrow \; f \in V^d_{\kappa,\kappa,\pm}.
\end{align}

We know that $\bv^\kappa_\kappa W^\kappa(\cdot,\mu,\psi_{1,1})=0$, and the argument of section \ref{sect:GoingUp} applied to $\bv^\kappa_\kappa$ (in place of $\bu^{\kappa,\pm}_\kappa$) implies that $v W^d(\cdot,\mu,\psi_{1,1})=0$ for all $v\in V^d_{\kappa,\kappa,+}$ for all $d$.
Similarly, since $V^d_{\kappa,\kappa,-}$ is closed under the $Y^a_\mu$ operators (because $\mu_1-\mu_2+1-\kappa=0$ in equations \eqref{eq:ThreeTermRecursion}-\eqref{eq:FourTermReduction2}), and
\[ \bv^\kappa_{-\kappa} W^\kappa(\cdot,\mu,\psi_{1,1})\ne 0, \qquad \paren{g_4^{2\kappa-1,+}-g_4^{2\kappa-1,-}} W^{2\kappa-1}(\cdot,\mu,\psi_{1,1})\ne 0, \]
proposition \ref{prop:GDPower} implies $v W^d(\cdot,\mu,\psi_{1,1})\ne 0$ for all $v\in V^d_{\kappa,\kappa}$, $v\notin V^d_{\kappa,\kappa,+}$ for all $d$.
Even stronger, by \eqref{eq:GDonVkappa}, proposition \ref{prop:GDPower}, and proposition \ref{prop:MinimalWhittsZero} at $d=0,1$, we know $v W^d(\cdot,\mu,\psi_{1,1})\ne 0$ for all $v\in V^d_\kappa$, $v\notin V^d_{\kappa,\kappa,+}$ for all $d$.

Suppose
\begin{align*}
	Y^{-1} f W^d(g,\mu,\psi_{1,1})=0, \qquad Y^{-2} f W^d(g,\mu,\psi_{1,1})=0,
\end{align*}
and $f$ is not power-function minimal.
Set $\varepsilon=(-1)^\kappa$ and $\delta\in\set{0,1}$, $\delta\equiv\kappa\pmod{2}$.
Since the zero function cannot descend to a non-zero Whittaker function, $f$ must descend, via $Y^{-1}_\mu$ and $Y^{-2}_\mu$ to a non-zero linear combination of either
\[  g_4^{2\kappa-1,+}+g_4^{2\kappa-1,-} \text{ and } (g_4^{2\kappa-1,+}+g_4^{2\kappa-1,-})\WigDMat{d}(\vpmpm{--} w_l), \]
or $\bv^\kappa_\kappa$ and $\bv^\kappa_\kappa \WigDMat{d}(\vpmpm{--} w_l)$.

Set
\[ h = \sum_{\substack{j \ge \kappa\\j\equiv\kappa\summod{2}}} \frac{f_j+(-1)^{d-\kappa} f_{-j}}{2} \bv^d_j \in V^d_{\kappa,\kappa,+}, \]
then $h W^d(g,\mu,\psi_{1,1})$ is zero.
Subtracting $h$ removes the projection onto $\chi_{\varepsilon,\varepsilon}$, i.e. $(f-h)\Sigma^d_{\varepsilon,\varepsilon}=0$, and so by the above arguments, $f-h$ must descend to a multiple of either
\[ (g_4^{2\kappa-1,+}+g_4^{2\kappa-1,-})\WigDMat{d}(\vpmpm{--} w_l) \text{ or } \bv^\kappa_\kappa \WigDMat{d}(\vpmpm{--} w_l). \]

Now we switch to the dual Whittaker function (note $-\mu^{w_l}=\mu$):
Set $\wtilde{f}=(f-h)\WigDMat{d}(\vpmpm{--} w_l)$, then $\wtilde{f}$ must descend to a multiple of either $g_4^{2\kappa-1,+}+g_4^{2\kappa-1,-} \in V^{2\kappa-1}_{\kappa,\kappa,+}$ or $\bv^\kappa_\kappa \in V^\kappa_{\kappa,\kappa,+}$.
As before, we set 
\[ \wtilde{h} = \sum_{\substack{j \ge \kappa\\j\equiv\kappa\summod{2}}} \frac{\wtilde{f}_j-(-1)^{d-\kappa} \wtilde{f}_{-j}}{2} \bv^d_j \in V^d_{\kappa,\kappa,+}. \]
Now consider the vector $v=f-h-\wtilde{h}\WigDMat{d}(\vpmpm{--}w_l)$ which has the projections $v\Sigma^d_{\varepsilon,\varepsilon}=v\Sigma^d_{-\varepsilon,\varepsilon}=0$.
The projection of $v$ onto $V^d_{\kappa}$ is not contained in $V^d_{\kappa,\kappa,+}$ unless it is zero, and the same is true for the projection of $v\WigDMat{d}(\vpmpm{--} w_l)$.
We still have
\begin{align*}
	Y^{-1} v W^d(g,\mu,\psi_{1,1})=0, \qquad Y^{-2} v W^d(g,\mu,\psi_{1,1})=0,
\end{align*}
and applying proposition \ref{prop:GDPower} to $v$, we conclude $v=0$ because any non-zero minimal descendant could only meet the conclusions 4-6, and we have constructed $v$ so this is impossible.
Therefore,
\[ f W^d(g,\mu,\psi_{1,1})=(h+\wtilde{h}\WigDMat{d}(\vpmpm{--} w_l))W^d(g,\mu,\psi_{1,1})=0. \]

\end{proof}

Using the bases for the spaces of power-function minimal vectors given in proposition \ref{prop:MinimalWhittsZero}, together with propositions \ref{prop:GDPower} and \ref{prop:GDWhittFEs}, and theorem \ref{thm:MinWhitt} gives:
\begin{cor}
\label{cor:WhittFinalFEs}
	Suppose $d \ge 2$, $\mu$ and $f\in\C^{2d+1}$ are such that \eqref{eq:GDWhittLowering} holds and further that $f W^d(g,\mu,\psi_{1,1})$ is an eigenfunction of $Y^0$ if $d=2$.
	Then $f W^d(g,\mu,\psi_{1,1}) = C f' W^d(g,\mu', \psi_{1,1})$, for some non-zero constant $C$, where
	\begin{enumerate}
		\item $\mu'=\paren{\frac{d-1}{2}+i t, -\frac{d-1}{2}+i t, -2i t}$ with $f'=\bu^{d,+}_d$, or
		\item $\mu'=\paren{d-1, 0, 1-d}$ with $f'$ some linear combination of $\bv^d_{-d}$ and $\bv^d_{-d} \WigDMat{d}(\vpmpm{--} w_l)$.
	\end{enumerate}
\end{cor}
This follows, for example, when $\mu=\paren{\frac{d-1}{2},0,-\frac{d-1}{2}}$ with $d \equiv 3 \pmod{4}$ by writing
\[ f=a_1 g_2^{d,0,-}+a_2 g_4^{d,+}+a_3 g_4^{d,-} = \paren{a_1 C_1 \bu^{d,-}_d+(a_2-a_3)C_2 \bu^{d,+}_d} T^d(w_3,\mu^{w_3})+a_3(g_4^{d,+}+g_4^{d,-}), \]
and using \eqref{eq:WhittFEs} and
\[ \bv^d_d W^d(g,\mu^{w_3},\psi_{1,1}) = (g_4^{d,+}+g_4^{d,-}) W^d(g,\mu,\psi_{1,1}) = 0. \]

In case $d=2$, the additional assumption about the behavior under $Y^0$ rules out the false positive posed by $g_1$, as in the discussion following proposition \ref{prop:GDPower}.
As mentioned in section \ref{sect:BadWhitt}, the second case will not occur as the Whittaker function of a cusp form.

\subsection{For cusp forms}
\label{sect:GDCuspForms}
We may now prove theorem \ref{thm:GDCusp}.
As mentioned in the discussion preceeding the theorem, we take condition 2 of proposition \ref{prop:EquivMinimalKTypes} as our working definition of minimal-weight forms.
Suppose $\phi$ has minimal weight $d$ and spectral parameters $\mu$.

The $n$-th Fourier coefficient of $\phi$ is of the form $f W^d(\cdot, \mu, \psi_n), f\in\C^{2d+1}$ by theorem \ref{thm:WhittMultOne} (and the discussion of section \ref{sect:KProjOps}), and proposition \ref{prop:GDWhitt} gives the allowed values of $d$, $f$ and $\mu$.
If $d=0$, there is nothing to do, and if $ d \ge 2$, we apply corollary \ref{cor:WhittFinalFEs} to arrive at the parameter set described in theorem \ref{thm:GDCusp}.
There is a minor caveat that this corollary is given in terms of the character $\psi_{1,1}$ and not $\psi_n$, but this is readily fixed:
If $n \in \Z^2, n_1 n_2 \ne 0$, we may use the isomorphism $n \in (\R^\times)^2 \cong V Y^+$ to write $n=v \tilde{n}$ with $v\in V$, and $\tilde{n}\in Y^+$, then using (I.3.7) and (I.3.9), we have
\begin{align*}
	W^d(g,\mu,\psi_n) =& p_{-\rho-\mu^{w_l}}(\tilde{n}) \WigDMat{d}(w_l v w_l) W^d(n g,\mu,\psi_{1,1}),
\end{align*}
and \eqref{eq:buDv} shows the extra matrix $\WigDMat{d}(w_l v w_l)$ at worst alters the sign of the coefficient.

A small bit more need be said about the cases $\mu=(d-1,0,1-d)$ and $d=1$.
When $\mu=(d-1,0,1-d)$, we have already pointed out in section \ref{sect:BadWhitt} that the asymptotics of $W^d_{-d}(g,\mu,\psi_n)$ are not compatible with the boundedness of cusp forms.
Another way to see these don't occur is to note that such a cusp form would have a real eigenvalue for the skew-symmetric operator $Y^0$ (recall \eqref{eq:BadWhittEigen} and the following discussion).

Now suppose $d=1$ and $\mu$ is arbitrary.
The eigenvalues under $Y^0$ are given in \eqref{eq:d1Y0Eigen}; if the components of $\mu$ are distinct, these choices of $f$ give three linearly independent Whittaker functions.
In case $\mu=(x+it,-2it,-x+it)$, $x\ne 0$, only $\bu^{1,-}_1$ may give the Whittaker function of a cusp form since the others are again eigenfunctions of a skew-symmetric operator with eigenvalues that are not purely imaginary.
Further, the functional equations of the Whittaker function yield
\begin{align*}
	\bu^{1,-}_1 T(w_3,\mu^{w_3}) =& C \bu^{1,-}_0, \\
	\bu^{1,+}_1 T(w_5,\mu^{w_4}) =& C \bu^{1,-}_0.
\end{align*}
So it suffices to take $\bu^{1,-}_0 W^1(y, \mu, \psi_{1,1})$ to be the $d=1$ Whittaker function with either $\mu=(it_1,it_2,-i(t_1+t_2))$ or $\mu=(x+it,-x+it,-2it)$.

\section{Going Up}
\label{sect:GoingUp}
Throughout this section, we fix a triple of spectral parameters $\mu$ and a character $\chi=\chi_{(-1)^\delta,\varepsilon}$ with its two parities $\delta\in\set{0,1}$ and $\varepsilon=\pm1$, as in \eqref{eq:chipmpmdef}.
For any $0 \le \kappa \in \Z$, let $\mathcal{V}^d_{\kappa,\chi} = \Span_{j \ge \kappa} \bu^{d,\varepsilon}_{2j+\delta} \subset \C^{2d+1}$.
As in the previous section, we may safely assume $-\wbar{\mu}$ is a permutation of $\mu$, i.e. that $\mu$ is a permutation of either
\[ (i t_1,i t_2,-i(t_1+t_2)) \qquad \text{or} \qquad (x+i t,-x+i t,-2i t), \qquad t_i,x,t\in\R,  t_1 > t_2 > -t_1 -t_2, x \ge 0. \]
We choose the trivial permutation.
In case $\mu_1-\mu_2+1 \in (2\Z+\delta)$, we define $\kappa=\mu_1-\mu_2+1$, and otherwise we set $\kappa=0$.
For convenience, we define $d_0=\kappa$ when $\kappa > 0$ and
\[ d_0=\piecewise{1 & \If \delta=1 \text{ or } \varepsilon = -1, \\ 0 & \Otherwise,} \]
when $\kappa=0$.
Lastly and for this section only, we apply the shorthand $\bu^d_j=\bu^{d,\varepsilon}_j$ and $\mathcal{V}^d = \mathcal{V}^d_{\kappa,\chi}$.

For each dimension/weight $d$, we have a ``minimal-weight vector''
\[ \bu^d_\text{min} = \bu^d_{j_\text{min}}, \qquad j_\text{min} = \piecewise{2 & \If \kappa=\delta=0, d\not\equiv d_0\pmod{2}, \\ 1 & \If \delta=1 \text{ and } \kappa=0, \\ \kappa & \Otherwise.} \]
We wish to show for $d \ge d_0$ that $\mathcal{V}^d$ can be generated by applying the $Y^a_\mu$ operators to $\bu^{d_0}_\text{min}$.
We accomplish this through induction by showing that suitable combinations of operators applied to $\bu^d_\text{min}$ will give $\bu^{d+1}_\text{min}$ or $\bu^{d+2}_\text{min}$; we call this the $d \to d+1$ or $d \to d+2$ step.
We can then fill out the remainder of $\mathcal{V}^{d+1}$ or $\mathcal{V}^{d+2}$ by repeatedly applying \eqref{eq:ThreeTermRecursion}.
Precisely, we show
\begin{prop}
\label{prop:GoingUpPowFun}
	\[ \C[Y^0_\mu, Y^1_\mu, Y^2_\mu] \bu^{d_0}_\text{min} = \bigcup_{d \ge d_0} \mathcal{V}^d, \]
	where $\C[Y^0_\mu, Y^1_\mu, Y^2_\mu]$ is the complex algebra generated by $Y^0_\mu, Y^1_\mu, Y^2_\mu$ under composition, and the $Y^a_\mu$ operators are viewed in the sense \eqref{eq:YmuDom}.
\end{prop}

All of the raising operators described in the induction argument below follow from either \eqref{eq:FourTermInduction1} and \eqref{eq:FourTermInduction2} directly or from the following two linear combinations:
Set
\begin{align}
\label{eq:RPlus1}
	R^{d,1}_{\mu,j} =&\sqrt{d(d+2)^2(2d+1)^2(2d+5)} Y^0_\mu Y^1_\mu\\
	& \qquad +2\sqrt{6d(d+1)(d+2)(2d+1)}(d-1-2j-2\mu_3) Y^1_\mu \nonumber \\
	& \qquad -\sqrt{d^2(d+2)(2d-1)(2d+1)(2d+3)} Y^1_\mu Y^0_\mu \nonumber,
\end{align}
\begin{align}
\label{eq:RPlus2}
	R^{d,2}_{\mu,j} = &\sqrt{(d+1)(d+2)^2(d+3)(2d+1)(2d+7)} Y^0_\mu Y^2_\mu\\
	& \qquad +2\sqrt{6(d+1)(d+2)(2d+1)(2d+3)}(2d+1-2j-2\mu_3) Y^2_\mu \nonumber \\
	& \qquad -\sqrt{d(d+1)^2(d+2)(2d-1)(2d+1)} Y^2_\mu Y^0_\mu, \nonumber
\end{align}
then
\begin{align}
\label{eq:Plus1}
	R^{d,1}_{\mu,j} \bu^d_j &= 8j\sqrt{3(d+1-j)(d+2-j)(d+3-j)(d+j)}(\mu_1-\mu_2+1-j) \bu^{d+1}_{j-2} \\
	& \qquad -8j\sqrt{3(d+1-j)(d+1+j)}\bigl((d-j-2)(j+1+3\mu_3) \nonumber \\
	& \qquad \qquad -2(\mu_1-\mu_3+1)(\mu_2-\mu_3+1)+4(j+1)\bigr) \bu^{d+1}_{j}, \nonumber
\end{align}
\begin{align}
\label{eq:Plus2}
	R^{d,2}_{\mu,j} \bu^d_j &= -4j\sqrt{6(d+1-j)(d+2-j)(d+3-j)(d+4-j)}(\mu_1-\mu_2+1-j) \bu^{d+2}_{j-2} \\
	& \qquad -4\sqrt{6(d+1-j)(d+2-j)(d+1+j)(d+2+j)}\bigl((2d-j)(d+2-3\mu_3) \nonumber \\
	& \qquad \qquad +2(\mu_1-\mu_3+1)(\mu_2-\mu_3+1)-j(d+1-j)\bigr) \bu^{d+2}_{j}. \nonumber
\end{align}
In both operators, it is not too hard to see that the constant multiplying $\bu^{d+a}_j$ is non-zero, but we will show a stronger statement about the norms of the new vectors.
These become true raising operators because in the cases we use them one of the following is true:
\begin{enumerate}
\item $j=0$,
\item $\bu^{d+a}_{j-2}=0$,
\item $\mu_1-\mu_2+1-j=0$,
\item $j=1$ (so $\bu^{d+a}_{j-2} = \pm \bu^{d+a}_{j}$).
\end{enumerate}
Note that this is where the argument would fail if we attempted to use the highest-weight vector.

\begin{prop}
\label{prop:GoingUpInnerProd}
Suppose we have a sequence of sesquilinear forms $\innerprod{\cdot,\cdot}$ on $\mathcal{V}^d \times \mathcal{V}^d_{0,\chi}$, $d \ge d_0$ that satisfy:
\begin{align}
\label{eq:FormDef}
	& \innerprod{su,tv}=s\wbar{t}\innerprod{u,v}, s,t\in\C, \qquad \innerprod{u,v}=\innerprod{u, \proj_{\mathcal{V}^d} v}, \qquad \innerprod{Y^a_\mu u, v} = \innerprod{u, \what{Y}^a_{-\wbar{\mu}} v},
\end{align}
and that, for $d=d_0$, we have
\begin{align}
\label{eq:FormONB}
	\innerprod{\bu^{d}_{2i+\delta}, \bu^{d}_{2i'+\delta}} &= \bu^{d}_{2i+\delta} \trans{(\bu^{d}_{2i'+\delta})}, \qquad 2i+\delta,2i'+\delta \ge \kappa.
\end{align}
Then \eqref{eq:FormONB} continues to hold for all $d \ge d_0$.
\end{prop}

\subsubsection*{Remarks}
\begin{enumerate}
\item By construction $\dim \mathcal{V}^{d_0}=1$, so the assumption that \eqref{eq:FormONB} hold for $d=d_0$ can be reduced to just
\[ \innerprod{\bu^{d_0}_\text{min}, \bu^{d_0}_\text{min}} = \piecewise{1 & \If \kappa=\delta=0, \varepsilon = +1, \\ \tfrac{1}{2} & \Otherwise.} \]

\item The projection assumption $\innerprod{u,v}=\innerprod{u, \proj_{\mathcal{V}^d} v}$ is necessary since the $Y^a_{-\wbar{\mu}}$ operators do not respect the spaces $\mathcal{V}^d$, even though the $Y^a_\mu$ do.
In practice, this assumption is met since the irksome rows of the incomplete Whittaker function $W^d(g,-\wbar{\mu},\psi)$ on the right-hand side of our inner product on Maass forms will be zero.

\item The actual sequence of sesquilinear forms we will use is given by the left-hand side of \eqref{eq:SesquiActual}; that is,
\begin{align}
\label{eq:ActualInnerProdDef}
	\innerprod{v,v'} = \int_{\Gamma\backslash G} (v \Phi^d(g))\wbar{\trans{\paren{v'\wtilde{\Phi}^d(g)}}}dg, \qquad v\in\mathcal{V}^d, v'\in\mathcal{V}^d_{0,\chi},
\end{align}
where $\Phi^d$ and $\wtilde{\Phi}^d$ are constructed from the Fourier expansion of a single minimal-weight form as in \eqref{eq:PhiFourierWhitt} and the comment that follows.
\end{enumerate}

\begin{proof}[Proof of propositions \ref{prop:GoingUpPowFun} and \ref{prop:GoingUpInnerProd}]
With respect to the sesquilinear forms of proposition \ref{prop:GoingUpInnerProd}, the adjoints of the operators $R^{d,a}_{\mu,j}$ act on the particular vectors $\bu^{d+a}_j$ as
\begin{align}
\label{eq:AdjPlus1}
	& \wbar{\what{R}^{d,1}_{\mu,j} \bu^{d+1}_j} = \\
	& -8\sqrt{3(d+2-j)(d-1+j)(d+j)(d+1+j)}(\mu_1-\mu_2-1+j) \bu^{d}_{j-2} \nonumber \\
	& -8j\sqrt{3(d+1-j)(d+1+j)}\bigl((d-j-2)(j+1+3\mu_3) \nonumber \\
	& \qquad -2(\mu_1-\mu_3+1)(\mu_2-\mu_3+1)+4(j+1)\bigr) \bu^{d}_{j} \nonumber \\
	& +8(j+1)\sqrt{3(d-1-j)(d-j)(d+1-j)(d+2+j)}(\mu_1-\mu_2-1-j) \bu^{d}_{j+2}, \nonumber
\end{align}
\begin{align}
\label{eq:AdjPlus2}
	& \wbar{\what{R}^{d,2}_{\mu,j} \bu^{d+2}_j} = \\
	& -4\sqrt{6(d-1+j)(d+j)(d+1+j)(d+2+j)}(\mu_1-\mu_2-1+j) \bu^{d}_{j-2} \nonumber \\
	& -4\sqrt{6(d+1-j)(d+2-j)(d+1+j)(d+2+j)}\bigl((2d-j)(d+2-3\mu_3) \nonumber \\
	& \qquad +2(\mu_1-\mu_3+1)(\mu_2-\mu_3+1)-j(d+1-j)\bigr) \bu^{d}_{j} \nonumber \\
	& -4(1+j)\sqrt{6(d-1-j)(d-j)(d+1-j)(d+2-j)}(\mu_1-\mu_2-1-j) \bu^{d}_{j+2} \nonumber
\end{align}
Note the coefficient on $\bu^d_j$ matches those on $\bu^{d+a}_j$ in \eqref{eq:Plus1} and \eqref{eq:Plus2}.
It follows that if $\innerprod{\bu^{d}_{j}, \bu^{d}_{j+2}} = 0$ (which will be the induction assumption), then
\begin{align}
\label{eq:ONBPlus1}
	& j\sqrt{(d+1-j)(d+1+j)}\bigl((d-j-2)(j+1+3\mu_3)-2(\mu_1-\mu_3+1)(\mu_2-\mu_3+1) \\
	& \qquad+4(j+1)\bigr) \paren{\innerprod{\bu^{d+1}_{j},\bu^{d+1}_{j}}- \innerprod{\bu^d_j, \bu^{d}_{j}}}, \nonumber \\
	&= \sqrt{(d+2-j)(d+j)(d-1+j)(d+1+j)} (\mu_1-\mu_2-1+j) \innerprod{\bu^d_j, \bu^{d}_{j-2}} \nonumber \\
	& \qquad +j\sqrt{(d+2-j)(d+j)(d+1-j)(d+3-j)}(\mu_1-\mu_2+1-j) \innerprod{\bu^{d+1}_{j-2}, \bu^{d+1}_{j}} \nonumber
\end{align}
\begin{align}
\label{eq:ONBPlus2}
	& \sqrt{(d+1-j)(d+2-j)(d+1+j)(d+2+j)}\bigl((2d-j)(d+2-3\mu_3) \\
	& \qquad +2(\mu_1-\mu_3+1)(\mu_2-\mu_3+1)-j(d+1-j)\bigr) \paren{\innerprod{\bu^{d+2}_{j}, \bu^{d+2}_{j}}-\innerprod{\bu^{d}_{j}, \bu^{d}_{j}}} \nonumber \\
	&= \sqrt{(d-1+j)(d+j)(d+1+j)(d+2+j)}(\mu_1-\mu_2-1+j) \innerprod{\bu^{d}_{j}, \bu^{d}_{j-2}} \nonumber \\
	& -j\sqrt{(d+1-j)(d+2-j)(d+3-j)(d+4-j)}(\mu_1-\mu_2+1-j) \innerprod{\bu^{d+2}_{j-2}, \bu^{d+2}_{j}}. \nonumber
\end{align}

We wish to show for $d \ge d_0$ that $\mathcal{V}^d$ is in the image of the raising operators, i.e.
\[ \mathcal{V}^d = \C[Y^0_\mu] \piecewise{\paren{\C Y^1_\mu \mathcal{V}^{d-1} + \C Y^2_\mu\mathcal{V}^{d-2}} & \If d \ge d_0+2, \\ Y^1_\mu \mathcal{V}^{d-1} & d=d_0+1, \\ \bu^d_\text{min} & d=d_0,} \]
and that \eqref{eq:FormONB} continues to hold in the higher weight.
We prove this in two steps:
First, we show that $\bu^d_\text{min}$ is in the image of the raising operators and \eqref{eq:FormONB} holds for $\bu^{d}_{2i+\delta} = \bu^{d}_{2i'+\delta} = \bu^d_\text{min}$ for all $d \ge d_0$ (the base case of the double induction).
Second, we extend this to all of $\mathcal{V}^d$, for all $d \ge d_0$ (the induction step of the double induction).

The first step, itself an induction argument on $d$, proceeds by cases.
Note that the base case of the induction is the statement that $\bu^{d_0}_\text{min}$ itself is in the image of $\bu^{d_0}_\text{min}$ under the raising operators, i.e. elements of $\C[Y^0_\mu,Y^1_\mu,Y^2_\mu]$, but this is obvious.
The cases are
\begin{align*}
	\begin{array}{r|c|c|c|l}
	\text{Case}&\text{Conditions}&\text{Step}\\
	\hline
	\text{Ia}&\kappa > 1&d_0 \to d_0+1\\
	\text{Ib}&\kappa > 1&d \to d+2\\
	\text{IIa}&\kappa=\delta=0, \varepsilon=+1&0 \to 2\\
	\text{IIb}&\kappa=\delta=0, \varepsilon=+1&2 \to 3\\
	\text{IIc}&\kappa=\delta=0, \varepsilon=+1&d \to d+2\\
	\text{IIIa}&\Max{\kappa,\delta}=1, \varepsilon=(-1)^d&d \to d+1\\
	\text{IIIb}&\Max{\kappa,\delta}=1, \varepsilon=-(-1)^d&d \to d+1
	\end{array}
\end{align*}

For the cases incrementing $d$ by 1, we use the raising operator \eqref{eq:Plus1} and apply \eqref{eq:ONBPlus1} for the orthonormality.
For the cases incrementing $d$ by 2, we use the raising operator \eqref{eq:Plus2} and apply \eqref{eq:ONBPlus2} for the orthonormality.
The need for the separation of cases are:
firstly, when $j_\text{min}=1$, we have $\bu^d_{j_\text{min}-2} = \varepsilon (-1)^d \bu^d_{j_\text{min}}$, so the form of the raising operator changes (slightly) with the parity of $d$; secondly, if $j_\text{min}=2$ for some $d$, then $j_\text{min}=0$ for $d+1$, and the raising operator cannot lift from $j=2$ at $d$ to $j=0$ at $d+1$ (but lifting from $j=2$ at $d$ to $j=2$ at $d+2$ is fine); lastly $\dim \mathcal{V}^1=0$ when $\kappa=\delta=0$,$\varepsilon=+1$.

We now prove case IIIa, the others are similar:
At $j=1$, \eqref{eq:Plus1} becomes
\begin{align*}
	R^{d,1}_{\mu,1} \bu^d_1 =& 16\sqrt{3d(d+2)}(\mu_1-\mu_3+d)(\mu_2-\mu_3-1) \bu^{d+1}_{1}.
\end{align*}
For the base case of the orthonormality condition \eqref{eq:FormONB}, we apply \eqref{eq:ONBPlus1} with $j=1$ and $\bu^{d+a}_{-1}=(-1)^a \bu^{d+a}_1$ (by assumption) and put everything on the same side, giving
\begin{align*}
	& 2\sqrt{d(d+2)}(\mu_1-\mu_3+d)(\mu_2-\mu_3-1) \paren{\innerprod{\bu^{d+1}_1,\bu^{d+1}_1}- \innerprod{\bu^d_1, \bu^{d}_1}} = 0.
\end{align*}
In both equations, we know $\sqrt{d(d+2)} (\mu_1-\mu_3+d)(\mu_2-\mu_3-1) \ne 0$, and this completes the induction step in this case.

We now proceed to the second step.
Again, the base case that $\mathcal{V}^{d_0} = \C \bu^{d_0}_\text{min}$ is in the image of $\bu^{d_0}_\text{min}$ under the raising operators and \eqref{eq:FormONB} holds at $d=d_0$ is obvious.
As mentioned above, $\mathcal{V}^d$ can be generated from $\bu^d_\text{min}$ by repeatedly applying \eqref{eq:ThreeTermRecursion}, though there is a little extra work going from $j=j_\text{min}$ to $j+2$ in the case where $j_\text{min} = 0$ since then $\bu^d_{-2} = \bu^d_2$ (note that $\bu^d_{-2} = -\bu^d_2$ cannot occur for $\delta=j_\text{min}=0$).
Even in that case, it is easy to see that the coefficient of $\bu^d_{j+2}$ in \eqref{eq:ThreeTermRecursion} is non-zero, and this is enough to conclude that $\mathcal{V}^d = \C[Y^0_\mu]\bu^d_\text{min}$ and hence is in the image of $\bu^{d_0}_\text{min}$ under the raising operators (by the conclusion of the first step).

Now we show \eqref{eq:FormONB} holds for $d > d_0$, assuming it holds for for all $d_0 \le d' < d$ and that it holds at $2i+\delta=2i'+\delta=j_\text{min}$.
By symmetry, it suffices to assume $i \ge i'$, since for $i' \ge i+1$ all of the relevant adjoint operators will respect the spaces $\mathcal{V}^{d+a}$ (which was the only asymmetry in the hypotheses on the sesquilinear forms).

For convenience, let $j=2i+\delta$ and $j'=2i'+\delta$.
We now proceed by induction on $j$.
The base cases are $j\in\set{0,1,\kappa}$, which necessarily imply $j=j'=j_\text{min}$, and this case is implied by the induction assumption above.
Note that $j=\kappa+1$ implies $\kappa=0$ so that again $j=1$.

Assume $j \ge \Max{3,\kappa+2}$.
First assume $j > j'$, then by \eqref{eq:FourTermInduction1}, we have
\begin{equation}
\begin{aligned}
	&\sqrt{2d(d+1)(d+2)(2d+1)} \innerprod{Y^1_\mu \bu^{d}_{j-2},\bu^{d+1}_{j'}} \\
	&=\sqrt{(d+2-j)(d-1+j)(d+j)(d+1+j)} (\mu_1-\mu_2-1+j) \innerprod{\bu^{d+1}_{j},\bu^{d+1}_{j'}} \\
	&\qquad-\sqrt{(d+3-j)(d+4-j)(d+5-j)(d+j-2)} (\mu_1-\mu_2+3-j) \tfrac{1}{2}\delta_{j-4=j'} \\
	&\qquad-2(j-2)\sqrt{(d+3-j)(d-1+j)} (3\mu_3-d-1) \tfrac{1}{2}\delta_{j-2=j'}.
\end{aligned}
\end{equation}
From the definition \eqref{eq:FormDef} and using \eqref{eq:AdjYmuDef} with our induction assumption and \eqref{eq:FourTermReduction1}, this may also be written
\begin{equation}
\begin{aligned}
	&\sqrt{2d(d+1)(d+2)(2d+3)} \innerprod{\bu^{d}_{j-2}, Y^{-1}_{-\wbar{\mu}} \bu^{d+1}_{j'}} \\
	&=-\sqrt{(d+2-j')(d-1+j')(d+j')(d+1+j')} (-\mu_1+\mu_2+1-j') \tfrac{1}{2}\delta_{j-2=j'-2} \\
	&\qquad+2j'\sqrt{(d+1-j')(d+1+j')} (-3\mu_3+d+1) \tfrac{1}{2}\delta_{j-2=j'} \\
	&\qquad+\sqrt{(d-1-j')(d-j')(d+1-j')(d+2+j')} (-\mu_1+\mu_2+1+j') \tfrac{1}{2}\delta_{j-2=j'+2},
\end{aligned}
\end{equation}
from which it follows
\[ \innerprod{\bu^{d+1}_{j},\bu^{d+1}_{j'}} = \tfrac{1}{2}\delta_{j=j'}. \]
Now having shown the case $j=j'+2$ (and by symmetry $j=j'-2$), the proof applies verbatim at $j=j'$.

The remaining case is $j=2$, $j'\in\set{0,2}$, with $\kappa=0$, $\varepsilon=(-1)^d$ and $d \ge d_0+2$, and this has a proof identical to the previous using \eqref{eq:FourTermInduction2}.
\end{proof}

\section{The structure of the cusp forms}
\label{sect:CuspStructure}

\subsection{The orthogonal and harmonic descriptions of minimal $K$-types}
\label{sect:MinKProof}
We now finish the proof of proposition \ref{prop:EquivMinimalKTypes}.

First, we show condition 1 implies condition 2:
If $\phi$ is orthogonal to the raises of all lower weight forms, then some word $L$ in $Y^{-1}$ and $Y^{-2}$, i.e. $L=Y^{a_1} Y^{a_2} \cdots Y^{a_k}$ for some $k \ge 0$ and $a_i\in\set{-1,-2}$, makes $0 \ne L\phi =: \phi^{d_0}$ minimal in the sense that $Y^{-1} \phi^{d_0}=0$ and $Y^{-2} \phi^{d_0} = 0$.
But if $L \ne 1$, then
\[ 0=\innerprod{\what{L} \phi^{d_0}, \phi}=\innerprod{\phi^{d_0}, L\phi}=\innerprod{\phi^{d_0},\phi^{d_0}} \ne 0, \]
a contradiction, so $\phi$ must already satisfy $Y^{-1} \phi = 0$ and $Y^{-2} \phi = 0$.

As mentioned in the discussions following corollary \ref{cor:WhittFinalFEs} and proposition \ref{prop:GDPower}, this is already sufficient to say $\phi$ is an eigenfunction of $Y^0$, except possibly when $d=2$.
If $d=2$ and $\phi$ is not an eigenfunction of $Y^0$ -- i.e. when the Fourier coefficients of $\phi$ are multiples of $g_1 W^2(\cdot,\mu,\psi_{1,1})$ with all $\mu_i-\mu_j\ne 1$ -- then $Y^{-2} Y^0 \phi$ is not zero (recall \eqref{eq:Ym2Y0g1}), and the above argument works taking $L$ to be $Y^{-2} Y^0$.
So $\phi$ is minimal in the sense of condition 2.

Condition 1 follows from condition 2 by considering minimal-weight ancestors:
We use the decomposition $\mathcal{A}^d = \bigoplus_\mu \mathcal{A}^d_\mu$ implied by theorem \ref{thm:SpectralDecomp}.
Suppose $\phi\in\mathcal{A}^d_\mu$ is minimal in the sense of condition 2.
Now if $\phi' \in \mathcal{A}^{d-a}_{\mu'}$ for some $a\in\set{1,2}$, then by applying $Y^{-1}$ and $Y^{-2}$ (and possibly $Y^0$) to $\phi'$, we must arrive at some cusp form which satisfies condition 2, but by theorem \ref{thm:GDCusp} (and the fact that $Y^1:\mathcal{A}^0\to\mathcal{A}^1$ is the zero operator), we know $\mu \ne \mu'$, and hence $\phi$ is orthogonal to $Y^a \phi'$ as they belong to different eigenspaces of the Casimir operators.

It remains to prove the equivalence of conditons 2 and 3.
Since $\Lambda_X$ necessarily commutes with $\Delta_1$ and $\Delta_2$, we may assume the cusp forms in question are eigenfunctions of the latter two operators.
We know that such forms have spectral parameters of the form $(i t_1, i t_2,-i(t_1+t_2))$ or $(x+i t, -x+i t, -2i t)$, up to permutation.
In the first case, the $\Lambda_X$ eigenvalue is computed to be
\[ \paren{(t_1-t_2)^2+4X^2}\paren{(2t_1+t_2)^2+4X^2}\paren{(t_1+2t_2)^2+4X^2}, \]
and for $X \ne 0$, this is not zero.
In the second case, the eigenvalue is
\[ 4(X-x)(X+x)\paren{9t^2+(2X-x)^2}\paren{9t^2+(2X+x)^2}, \]
and this is zero exactly when $x=\pm X$ or $(x,t)=(\pm 2X,0)$.
For $d \ge 2$, since we have shown the cusp forms having minimal $K$-type $\WigDMat{d}$ are exactly those with spectral parameters of the form
\[ \paren{\tfrac{d-1}{2}+i t, -\tfrac{d-1}{2}+i t, -2i t} \]
and there are no cusp forms of $K$-type $\WigDMat{d}$ with spectral parameters of the form $(d-1,1-d,0)$, this gives the claim.

\subsection{The cuspidal spectral expansion}
We have some final calculations to complete the proof of theorem \ref{thm:CuspSpectralExpand}.
Consider $\Phi\in\mathcal{S}_3$ and $\Phi^d$, $\wtilde{\Phi}^d$ as in section \ref{sect:MainThms} and $\mathcal{V}^d_{\kappa,\chi}$ as in section \ref{sect:GoingUp}.
Then $\Phi \in \mathcal{A}_\mu^{d_0*}$ for some $\mu$ and $d_0$, and we determine $\chi$ and $\kappa$ from $d_0$ by \eqref{eq:chifromphi} and
\begin{align}
	\kappa =& \piecewise{d_0 & \If d_0 \ge 2, \\ 0 & \Otherwise.}
\end{align}

First, we point out that in the proof of proposition \ref{prop:GDWhitt}, we showed $v\wtilde{\Phi}^d(g) = (\proj_{\mathcal{V}^d_{\kappa,\chi}} v)\wtilde{\Phi}^d(g)$ for $v \in \mathcal{V}^d_{0,\chi}$, since all of the Whittaker functions of $v\wtilde{\Phi}^d(g)$ are zero when also $v \in (\mathcal{V}^d_{\kappa,\chi})^\perp$; in particular, the projection assumption of \eqref{eq:FormDef} is met for the inner product \eqref{eq:ActualInnerProdDef}.
Moreover, in proposition \ref{prop:GoingUpPowFun}, we showed that all vectors $v\wtilde{\Phi}^d(g)$ with $v \in \mathcal{V}^d_{\kappa,\chi}$ are obtainable by applying suitable combinations of the $Y^a$ operators to $\Phi=\bu^{d_0,\varepsilon}_\text{min} \Phi^{d_0}$ itself, and hence they are all true cusp forms by the left-translation invariance of the $Y^a$ operators.
So the rows of $\Phi^d(g)$ which do not correspond to cusp forms also do not contribute to $\Tr\paren{\Phi^d(g) \wbar{\trans{\wtilde{\Phi}^d(g')}}}$.

One might wonder about the need for $\wtilde{\Phi}$ when a similar situation does not happen for the maximal parabolic Eisenstein series.
The simple answer is that for the Eisenstein series, our choice of normalization constants effectively completes the Whittaker function under the $\mu \mapsto \mu^{w_2}$ functional equation.
(We cannot formulate a matrix-valued Whittaker function which is complete under all of the functional equations, because the $w_3$ functional equation, in particular, is not a diagonal matrix, while the $w_2$ functional equation is diagonal with distinct entries, so they cannot be simultaneously diagonalized.)

We normalize the Fourier-Whittaker coefficients as follows:
The space $\mathcal{V}^{d_0}_{\kappa,\chi}$ has dimension one, and hence the matrix $\wtilde{\Phi}^{d_0}(g)$ has exactly one distinct non-zero row (which may occur twice, due to the symmetry of $\Sigma_\chi$), call it $\wtilde{\Phi}$.
If we insist that
\begin{align}
\label{eq:PhiNormalization}
	\int_{\Gamma\backslash G} \Phi(g) \wbar{\trans{\wtilde{\Phi}(g)}} dg = \piecewise{1 & \If \kappa=\delta=0, \\ \frac{1}{2} & \Otherwise,}
\end{align}
then proposition \ref{prop:GoingUpInnerProd} implies that the rows of $\Phi^d$ and $\wtilde{\Phi}^d$ have the desired orthonormality.
That is,
\begin{align}
\label{eq:SesquiActual}
	\int_{\Gamma\backslash G} (v \Phi^d(g)) \wbar{\trans{\paren{v'\wtilde{\Phi}^d(g)}}} dg = v\,\wbar{\trans{v'}},
\end{align}
for $v,v' \in \mathcal{V}^d_{\kappa,\chi}$.
Note that the rows of $\Phi^d(g)$ which apparently have norm $\frac{1}{\sqrt{2}}$ actually occur twice in $\Phi^d(g)$, so there is no discrepancy.
Since $-\wbar{\mu}$ is either $\mu$ or $\mu^{w_2}$ and the $w_2$ functional equation of the Whittaker function acts by a non-zero scalar on the minimal Whittaker function, it is always possible to arrange \eqref{eq:PhiNormalization}.

The final step is to convert back to scalar-valued forms.
We notice that
\begin{align*}
	& (2d+1) \int_{\Gamma\backslash G} \Phi^d_{i,j}(g) \wbar{\wtilde{\Phi}^d_{k,\ell}(g)} dg \\
	&= (2d+1) \sum_{m,n} \int_{\Gamma\backslash G/K} \Phi^d_{i,m}(z) \wbar{\wtilde{\Phi}^d_{n,k}(z)} \int_K \WigD{d}{m}{j}(k) \wbar{\WigD{d}{n}{\ell}(k)} dk \, dz \\
	&= \delta_{j=\ell} \int_{\Gamma\backslash G} \Phi^d_i(g) \wbar{\trans{\wtilde{\Phi}^d_k(g)}} dg,
\end{align*}
and this produces the factor $2d+1$ in theorem \ref{thm:CuspSpectralExpand}.

\appendix

\section{Shalika's multiplicity one theorem}
\label{app:MultOne}
We now show that theorem \ref{thm:ext}, part 3 follows from Shalika's local multiplicity one theorem \cite[theorem 3.1]{Shal01}; for this we follow the notation of Knapp \cite{Knapp}.
Assume that $n_1 n_2 \ne 0$, since otherwise $c_n=0$ works.
We note that $\phi$ generates a unitary, admissible representation (finite dimensionality of the $K$-types is given by part 1 of the current theorem) of $G$ via right translation
\[ \mathcal{R}_\phi := \Span \set{\phi(\cdot g) \setdiv g\in G} \subset L^2(\Gamma\backslash G). \]
We have assumed that $\mathcal{R}_\phi$ is irreducible; in general, $\mathcal{A}^d_\mu$ is a finite (by part 1 of the theorem) span of elements of irreducible representations, so that the $c_n$ and $f$ of the theorem will be replaced by linear combinations of some finite set $f_1, \ldots, f_k$.

This representation is infinitesimally equivalent\footnote{That is, isomorphic by a bounded, unitary intertwining operator that commutes with the action of the Lie algebra $\frak{g}_\C$, see \cite[corollary 9.2]{Knapp}.} to a subrepresentation \cite[theorem 8.37]{Knapp} of a principal series representation
\[ \mathcal{P}_{\mu,\chi} = \set{f:G\to\C\setdiv f(xyvk)=p_{\rho+\mu}(y) \chi(v) f(k), \int_K \abs{f(k)}^2 dk < \infty} \]
for some $\chi$ (a character of the diagonal, orthogonal matrices $V$); denote the isomorphism by $\mathcal{L}:\mathcal{R}_\phi \to \mathcal{P}_{\mu,\chi}$.

Denote the subspace of smooth functions in $\mathcal{R}_\phi$, resp. $\mathcal{P}_{\mu,\chi}$, as $\mathcal{R}_\phi^\infty$, resp. $\mathcal{P}_{\mu,\chi}^\infty$.
We give these spaces the Fr\'echet topology generated by the semi-norms $\norm{f}^2_{R,X} := \int_{\Gamma\backslash G} \abs{(X f)(g)}^2 dg$, $f \in \mathcal{R}_\phi^\infty$ and $\norm{f}^2_{P,X} := \int_{K} \abs{(X f)(k)}^2 dk$, $f \in \mathcal{P}_{\mu,\chi}^\infty$ for all $X \in \frak{g}_\C$.
The property of infinitesimal equivalence means that the isomorphism $\mathcal{L}$ preserves the $(\frak{g},K)$-module structure of the admissible representations; in particular, the action of the Lie algebra $\frak{g}_\C$, and hence the generated Fr\'echet topology is preserved.
Thus $\mathcal{L}$ restricts to $\mathcal{L}^\infty:\mathcal{R}_\phi^\infty\to\mathcal{P}_{\mu,\chi}^\infty$.

We have the Whittaker model
\[ \mathcal{W}_{n,\mu} = \set{f \in C^\infty(G)\setdiv f(xg) = \psi_n(x) f(g), \Delta_i f = \lambda_i(\mu) f, i=1,2}, \]
which is once again given the Fr\'echet topology generated by the action of the Lie algebra.
The operator $\mathcal{F}_n:\mathcal{R}_\phi^\infty \to \mathcal{W}_{n,\mu}$, which takes a cusp form to its $n$-th Fourier coefficient,
\[ (\mathcal{F}_n f)(g) := \int_{U(\Z)\backslash U(\R)} f(ug) \wbar{\psi_n(u)} du, \qquad f \in \mathcal{R}_\phi^\infty, \]
has image in the Whittaker model, as does the Jacquet integral,
\[ (\mathcal{J}_n f)(g) := \int_{U(\R)} f(w_l u g) \wbar{\psi_n(u)} du, \qquad f \in \mathcal{P}_{\mu,\chi}^\infty, \]
viewed as an operator $\mathcal{J}_n:\mathcal{P}_{\mu,\chi}^\infty \to \mathcal{W}_{n,\mu}$.
Both operators commute with the action of the Lie algebra, hence they are continuous with respect to the Fr\'echet toplology.

The convergence and analytic continuation of the Jacquet integral was originally studied in \cite{Jac01}, but the necessary extension to $\Re(\mu_i)=\Re(\mu_j)$, $i \ne j$ can instead be deduced from propositions I.3.1 and I.3.3.
Indeed, the proposition I.3.1 gives the analytic continuation, and the functional equations, proposition I.3.3, plus the usual Phragm\'en-Lindel\"of argument shows the entries of $W^d(g,\mu,\psi_n)$ are polynomially bounded in $d$.
Then on $\Re(\mu_i)=\Re(\mu_j)$, we may define $\mathcal{J}_n f$ by the expansion (I.2.10) of $f$ into Wigner $\WigDName$-matrices,
\[ (\mathcal{J}_n f)(g) := \sum_{d \ge 0} (2d+1) \Tr\paren{\int_K f(k) \wbar{\trans{\WigDMat{d}(k)}} dk \, W^d(g,\mu,\psi_n)}, \]
since the entries of $\int_K f(k) \wbar{\trans{\WigDMat{d}(k)}} dk$ have super-polynomial decay in $d$.
This is the natural, continuous extension which gives $(\mathcal{J}_n f)(g) = W^d_{m',m}(g,\mu,\psi_n)$ for $f(xyk) = p_{\rho+\mu}(y) \WigD{d}{m'}{m}(k)$.

As an aside, we note that the extension of Jacquet integral to $\Re(\mu_i)=\Re(\mu_j)$ may instead be accomplished by interpretting the integral in the Riemannian sense $\int_{U(\R)}=\lim_{R\to\infty} \int_{[-R,R]^2\times\R}$, by integration by parts; this is done very explicitly for $f=1$ in the analysis of $X_3$ in \cite[section 4.3]{Me01}, but easily extends to non-trivial $f$.

Shalika's local multiplicity one theorem \cite[theorem 3.1]{Shal01} states the operators $\mathcal{F}_n$ and $\mathcal{J}_n \circ \mathcal{L}^\infty$ are identical, up to a constant.
In particular, $\mathcal{F}_n(\phi) = c_n \mathcal{J}_n(\mathcal{L}^\infty(\phi))$ for some constant $c_n\in\C$, and we take the $f$ in the theorem to be the restriction to $K$ of $\mathcal{L}^\infty(\phi)$.

In terms of Shalika's notation, we notice that Shalika's $\mathcal{D}(\Pi)$ is a dense subspace of the smooth vectors $\mathcal{R}_\phi^\infty$, the continuity condition on elements of Shalika's $\mathcal{D}'(\Pi)$ is with respect to the Fr\'echet topology described above, and the images of both $\mathcal{F}_n$ and $\mathcal{J}_n \circ \mathcal{L}^\infty$ trivially lie in Shalika's $\mathcal{D}'_{\psi_n}(\Pi) \subset \mathcal{W}_{n,\mu}$.

\section{Computing the Casimir operators}
\label{app:CasCompute}
We now prove lemma \ref{lem:CasCompute}.
The first step is to write the $E_{i,j}$ basis in terms of the $X_j,K_j$ basis:
As operators on smooth functions of $G$ (that is, we drop the identity matrix which acts as the zero operator),
\begin{align}
\label{eq:EijToXK}
	4E_{i,j} = \sum_{\abs{k} \le 2} \wbar{[X_k]_{i,j}} X_k+\sum_{\abs{k} \le 1} \wbar{[K_{-k}]_{i,j}} K_k,
\end{align}
where, as in \eqref{eq:XKKilling}, $[X_k]_{i,j}$ means the entry at index $i,j$ of the matrix $X_k$ and similarly for $[K_k]_{i,j}$.

By the symmetries of the matrices $X_k$ and $K_k$, the coefficients satisfy
\begin{align*}
	\wbar{[X_k]_{i,j}} =& (-1)^k [X_{-k}]_{i,j}, & \wbar{[K_k]_{i,j}} =& (-1)^k [K_{-k}]_{i,j}, \\
	[X_k]_{j,i} =& [X_k]_{i,j} & [K_k]_{j,i} =& -[K_k]_{i,j}.
\end{align*}
Inserting these expressions into \eqref{eq:CasimirDef} gives
\begin{align*}
	-32\Delta_1 =& \sum_{i,j} \paren{\sum_{\abs{k_1} \le 2} (-1)^{k_1} [X_{-k_1}]_{i,j} X_{k_1}+\sum_{\abs{k_1} \le 1} (-1)^{k_1} [K_{-k_1}]_{i,j} K_{k_1}} \\
	& \qquad \circ \paren{\sum_{\abs{k_2} \le 2} \wbar{[X_{k_2}]_{i,j}} X_{k_2}-\sum_{\abs{k_2} \le 1} \wbar{[K_{k_2}]_{i,j}} K_{k_2}}.
\end{align*}

Now we interchange the sums and use the orthonormality \eqref{eq:XKKilling}, so we have
\begin{align}
\label{eq:Delta1XK}
	-8 \Delta_1 =& \sum_{\abs{j}\le 2} (-1)^j X_j \circ X_{-j}+2\Delta_K.
\end{align}

We reduce to the $\wtilde{Z}_j$ operators by lemma \ref{eq:XjOp}, so that
\begin{align}
\label{eq:Delta1XK}
	-8 \Delta_1-2\Delta_K =& \sum_{\abs{\ell_1},\abs{\ell_2}\le 2} \paren{\sum_{\abs{j}\le 2} (-1)^j \WigD{2}{\ell_1}{j}(k) \WigD{2}{\ell_2}{-j}(k)} \wtilde{Z}_{\ell_1} \circ \wtilde{Z}_{\ell_2} \\
	& + \sum_{\abs{\ell_1},\abs{\ell_2}\le 2} \paren{\sum_{\abs{j}\le 2} (-1)^j \WigD{2}{\ell_1}{j}(k) \wtilde{Z}_{\ell_1} \WigD{2}{\ell_2}{-j}(k)} \wtilde{Z}_{\ell_2}. \nonumber
\end{align}

The symmetry \eqref{eq:WigdSymms} implies $(-1)^{\ell+j} \WigD{2}{\ell}{-j}(k) = \wbar{\WigD{2}{-\ell}{j}(k)}$ and so the orthogonality of the rows of $\WigDMat{2}(k)$ gives
\begin{align}
\label{eq:WDSymOrtho}
	\sum_{\abs{j}\le 2} (-1)^j \WigD{2}{\ell_1}{j}(k) \WigD{2}{\ell_2}{-j}(k) =& (-1)^{\ell_2} \delta_{\ell_1=-\ell_2}.
\end{align}
Also, we may compute from \eqref{eq:tildeZDef} and \eqref{eq:KLeftactWD} that
\begin{align}
\label{eq:tildeZactWDpm2}
	\wtilde{Z}_{\pm2} \WigD{d}{\ell}{j}(k) =& \pm\ell \WigD{d}{\ell}{j}(k), \\
\label{eq:tildeZactWDpm1}
	\wtilde{Z}_{\pm1} \WigD{d}{\ell}{j}(k) =& -\sqrt{d(d+1)-\ell(\ell\mp1)} \WigD{d}{\ell\mp1}{j}(k), \\
\label{eq:tildeZactWD0}
	\wtilde{Z}_0 \WigD{d}{\ell}{j}(k) =& 0.
\end{align}
Applying these two facts to \eqref{eq:Delta1XK}, we have
\begin{align*}
	-8 \Delta_1-2\Delta_K =& \sum_{\abs{\ell}\le 2} (-1)^\ell \wtilde{Z}_{\ell} \circ \wtilde{Z}_{-\ell} -2 \wtilde{Z}_{-2}+2\sqrt{6} \wtilde{Z}_0 -2 \wtilde{Z}_2.
\end{align*}

From \eqref{eq:tildeZDef}, \eqref{eq:DeltaKLeft} and the commutativity of all $\KLeft{j}$ and $Z_j$, we have
\begin{equation}
\begin{aligned}
\label{eq:Delta1ZK}
	-8 \Delta_1 =& \sum_{\ell=-2}^2 (-1)^\ell Z_{\ell} \circ Z_{-\ell}-2Z_2+2\sqrt{6}Z_0-2 Z_{-2} \\
	& +\sqrt{2}i \KLeft{0} \circ \paren{Z_2-Z_{-2}}+2\KLeft{1} \circ Z_{-1}+2\KLeft{-1} \circ Z_1,
\end{aligned}
\end{equation}
Since the operators $\KLeft{j}$ are zero on spherical functions, we see that
\begin{align}
	-8 \Delta_1^\circ =&\sum_{\ell=-2}^2 (-1)^\ell Z_{\ell} \circ Z_{-\ell}-2Z_2-2 Z_{-2}+2\sqrt{6}Z_0,
\end{align}
and \eqref{eq:FullDelta1} follows.

The degree three operator is somewhat more complicated, so we increase the formalism a little:
We collect the $X_\ell$ and $K_\ell$ operators and the $\wtilde{Z}_\ell$ and $\KLeft{\ell}$ operators by defining
\[ \mathcal{X}^d_\ell = \piecewise{X_\ell & \If d=2,\\ K_\ell & \If d=1,} \qquad \mathcal{Z}^d_\ell = \piecewise{\wtilde{Z}_\ell & \If d=2,\\ \KLeft{\ell} & \If d=1.} \]
In the same manner as we derived \eqref{eq:SimpkXjkinv}, we have
\begin{align}
\label{eq:SimpkKjkinv}
	k K_j k^{-1} = \sum_{\abs{\ell}\le 1} i^{j-\ell} \WigD{1}{\ell}{j}(k) K_j, \qquad k\in K,
\end{align}
which follows from
\begin{align}
	k(\alpha,0,0) K_j k(-\alpha,0,0) =& e^{-ij\alpha} K_j, & w_3 K_j w_3^{-1} =& \sum_{\abs{\ell}\le 1} i^{j-\ell} \WigD{1}{\ell}{w_3}(k) K_j,
\end{align}
and this may be checked using
\[ \WigDMat{1}(w_3)=\frac{1}{2}\Matrix{-1&-i\sqrt{2}&1\\i\sqrt{2}&0&i\sqrt{2}\\1&-i\sqrt{2}&-1}. \]
Then the trick \eqref{eq:LieAlgTrick} implies
\begin{align}
\label{eq:KtoKLeft}
	K_j =& \sum_{\abs{\ell}\le 1} i^{j-\ell} \WigD{1}{\ell}{j}(k) \KLeft{j},
\end{align}
as differential operators in the Iwasawa coordinates.
Collectively, we may now write
\begin{align}
\label{eq:mcXtomcZ}
	\mathcal{X}^d_j = \sum_{\abs{\ell}\le d} i^{d^2(j-\ell)} \mathcal{Z}^d_\ell.
\end{align}
(The factor $d^2$ here should not be confused with the Wigner $\WigdName$-polynomial.)

Now the orthonormality relations \eqref{eq:XKKilling} are replaced with properties of Clebsch-Gordan coefficients.
The identities are best expressed in terms of the Wigner three-$j$ symbols (recall \eqref{eq:CGtoWig3j}), so we define
\[ \frak{D}^{d_1,d_2,d_3}_{\ell_1,\ell_2,\ell_3} := \delta_{\ell_1+\ell_2+\ell_3=0} i^{d_1\ell_1+d_2\ell_2+d_3\ell_3} \Matrix{d_1&d_2&d_3\\ \ell_1&\ell_2&\ell_3} \times \piecewise{-\sqrt{105} & \If d_1+d_2+d_3=6,\\ -3\sqrt{15}i & \If d_1+d_2+d_3=5,\\ 3\sqrt{5} & \If d_1+d_2+d_3=4,\\ -3\sqrt{3}i & \If d_1+d_2+d_3=3.} \]
Then
\begin{align}
\label{eq:CastomcX}
	\mathcal{U} :=& 48 \sum_{i,j,k} E_{i,j}\circ E_{j,k} \circ E_{j,i} = \sum_{d_1,d_2,d_3\in\set{1,2}} \sum_{\substack{\abs{j_1}\le d_1,\abs{j_2}\le d_2,\abs{j_3} \le d_3\\ j_1+j_2+j_3=0}} \frak{D}^{d_1,d_2,d_3}_{j_1,j_2,j_3} \mathcal{X}^{d_1}_{j_1} \circ \mathcal{X}^{d_2}_{j_2} \circ \mathcal{X}^{d_3}_{j_3},
\end{align}
which follows from
\[ \frac{3}{4} \sum_{i,j,k} \wbar{[\mathcal{X}^{d_1}_{\ell_1}]_{i,j}} \wbar{[\mathcal{X}^{d_2}_{\ell_2}]_{j,k}} \wbar{[\mathcal{X}^{d_3}_{\ell_3}]_{k,i}} = \frac{3}{4} \wbar{\Tr\paren{\mathcal{X}^{d_2}_{\ell_2} \mathcal{X}^{d_2}_{\ell_2} \mathcal{X}^{d_3}_{\ell_3}}} = \frak{D}^{d_1,d_2,d_3}_{\ell_1,\ell_2,\ell_3}, \]
and this may be checked directly.

To \eqref{eq:CastomcX}, we apply \eqref{eq:mcXtomcZ}, and carefully interchange summations.
For $d,j,\ell\in\Z^3$, let
\[ \frak{E}^d_{j,\ell} = i^{d_1^2(j_1-\ell_1)+d_2^2(j_2-\ell_2)+d_3^2(j_3-\ell_3)} \frak{D}^{d_1,d_2,d_3}_{j_1,j_2,j_3}, \]
then the key identity is
\begin{align}
\label{eq:TripWDCGOrtho}
	\sum_{\substack{\abs{j_1}\le d_1,\abs{j_2}\le d_2,\abs{j_3} \le d_3\\ j_1+j_2+j_3=0}} \frak{E}^d_{j,\ell} \WigD{d_1}{\ell_1}{j_1}(k) \WigD{d_2}{\ell_2}{j_2}(k) \WigD{d_3}{\ell_3}{j_3}(k) =& \frak{D}^{d_1,d_2,d_3}_{\ell_1,\ell_2,\ell_3}.
\end{align}
This can be seen by the definition \eqref{eq:CGExpand} and orthogonality \cite[eq. 34.3.16]{DLMF} of the Clebsch-Gordon coefficients as follows:
Suppose, for convenience, that $d_1=d_2=d_3=2$, then
\begin{align*}
	&\sum_{j_1+j_2+j_3=0} \Matrix{2&2&2\\j_1&j_2&j_3} \WigD{2}{\ell_1}{j_1}(k) \WigD{2}{\ell_2}{j_2}(k) \WigD{2}{\ell_3}{j_3}(k) \\
	&=(-1)^{\ell_3} \sum_{d_4=0}^4 \Matrix{2&2& d_4\\ \ell_1&\ell_2&-\ell_1-\ell_2} \sum_{\abs{j_3}\le 2} (-1)^{j_3} \WigD{d_4}{\ell_1+\ell_2}{-j_3}(k) \WigD{2}{\ell_3}{j_3}(k) \\
	& \qquad \times \sum_{j_1+j_2=-j_3} (2d_4+1) \Matrix{2&2&2\\j_1&j_2&j_3} \Matrix{2&2&d_4\\j_1&j_2&j_3},
\end{align*}
and the inner sum over $j_1+j_2=-j_3$ is $\delta_{d_4=2}$ by orthogonality.
The identity follows by applying \eqref{eq:WDSymOrtho} on the $j_3$ sum.

We have $\mathcal{U} = \mathcal{U}_3+\mathcal{U}_2+\mathcal{U}_1$, where
\begin{align*}
	\mathcal{U}_3 =& \sum_{\substack{d_1,d_2,d_3\\ \ell_1,\ell_2,\ell_3}} \paren{\sum_{j_1+j_2+j_3=0} \frak{E}^d_{j,\ell} \WigD{d_1}{\ell_1}{j_1}(k) \WigD{d_2}{\ell_2}{j_2}(k) \WigD{d_3}{\ell_3}{j_3}(k)} \mathcal{Z}^{d_1}_{\ell_1} \circ \mathcal{Z}^{d_2}_{\ell_2} \circ \mathcal{Z}^{d_3}_{\ell_3},
\end{align*}
\begin{align*}
	\mathcal{U}_2 =& \sum_{\substack{d_1,d_2,d_3\\ \ell_1,\ell_2,\ell_3}} \paren{\sum_{j_1+j_2+j_3=0} \frak{E}^d_{j,\ell} \WigD{d_1}{\ell_1}{j_1}(k) \paren{\mathcal{Z}^{d_1}_{\ell_1} \WigD{d_2}{\ell_2}{j_2}(k)} \WigD{d_3}{\ell_3}{j_3}(k)} \mathcal{Z}^{d_2}_{\ell_2} \circ \mathcal{Z}^{d_3}_{\ell_3} \\
	&+\sum_{\substack{d_1,d_2,d_3\\ \ell_1,\ell_2,\ell_3}} \paren{\sum_{j_1+j_2+j_3=0} \frak{E}^d_{j,\ell} \WigD{d_1}{\ell_1}{j_1}(k) \WigD{d_2}{\ell_2}{j_2}(k) \paren{\mathcal{Z}^{d_1}_{\ell_1}\WigD{d_3}{\ell_3}{j_3}(k)}} \mathcal{Z}^{d_2}_{\ell_2} \circ \mathcal{Z}^{d_3}_{\ell_3} \\
	&+\sum_{\substack{d_1,d_2,d_3\\ \ell_1,\ell_2,\ell_3}} \paren{\sum_{j_1+j_2+j_3=0} \frak{E}^d_{j,\ell} \WigD{d_1}{\ell_1}{j_1}(k) \WigD{d_2}{\ell_2}{j_2}(k) \paren{\mathcal{Z}^{d_2}_{\ell_2}\WigD{d_3}{\ell_3}{j_3}(k)}} \mathcal{Z}^{d_1}_{\ell_1} \circ \mathcal{Z}^{d_3}_{\ell_3},
\end{align*}
\begin{align*}
	\mathcal{U}_1 =& \sum_{\substack{d_1,d_2,d_3\\ \ell_1,\ell_2,\ell_3}} \paren{\sum_{j_1+j_2+j_3=0} \frak{E}^d_{j,\ell} \WigD{d_1}{\ell_1}{j_1}(k) \WigD{d_2}{\ell_2}{j_2}(k) \paren{\mathcal{Z}^{d_1}_{\ell_1} \circ \mathcal{Z}^{d_2}_{\ell_2}\WigD{d_3}{\ell_3}{j_3}(k)}} \mathcal{Z}^{d_3}_{\ell_3} \\
	&+\sum_{\substack{d_1,d_2,d_3\\ \ell_1,\ell_2,\ell_3}} \paren{\sum_{j_1+j_2+j_3=0} \frak{E}^d_{j,\ell} \WigD{d_1}{\ell_1}{j_1}(k) \paren{\mathcal{Z}^{d_1}_{\ell_1} \WigD{d_2}{\ell_2}{j_2}(k)} \paren{\mathcal{Z}^{d_2}_{\ell_2} \WigD{d_3}{\ell_3}{j_3}(k)}} \mathcal{Z}^{d_3}_{\ell_3}.
\end{align*}

Though it is far more involved than the degree two operator, from \eqref{eq:tildeZactWDpm2}-\eqref{eq:tildeZactWD0} and \eqref{eq:TripWDCGOrtho}, we can compute
\begin{align}
	\mathcal{U}^3 =& \sum_{\substack{d_1,d_2,d_3\\ \ell_1+\ell_2+\ell_3=0}} \frak{D}^{d_1,d_2,d_3}_{\ell_1,\ell_2,\ell_3} \mathcal{Z}^{d_1}_{\ell_1} \circ \mathcal{Z}^{d_2}_{\ell_2} \circ \mathcal{Z}^{d_3}_{\ell_3}, \\
	\mathcal{U}^2 =& \sum_{\substack{d_2,d_3\\ \ell_2,\ell_3}} \frak{F}^{d_2,d_3}_{\ell_2,\ell_3} \mathcal{Z}^{d_2}_{\ell_2} \circ \mathcal{Z}^{d_3}_{\ell_3}, \\
	\mathcal{U}^1 =& 48\sqrt{6} \wtilde{Z}_0 = 48\sqrt{6} Z_0,
\end{align}
where the coefficients matrices $\frak{F}^{d_2,d_3}$ are given by
\begin{align}
	\frak{F}^{1,1} =& \Matrix{9&0&15\\0&12&0\\15&0&9}, & \frak{F}^{1,2} =& \Matrix{0&-3&0&3&0\\0&0&0&0&0\\0&3&0&-3&0}, \\
	\frak{F}^{1,2} =& \Matrix{0&6i\sqrt{2}&0\\3&0&-15\\0&0&0\\-15&0&3\\0&-6i\sqrt{2}&0}, & \frak{F}^{2,2}=& \Matrix{0&0&2\sqrt{6}&0&0\\0&-9&0&-27\\4\sqrt{6}&0&36&0&4\sqrt{6}\\0&-27&0&9&0\\0&0&2\sqrt{6}&0&0},
\end{align}
indexing from the center.
We have now completely removed the Wigner $\WigDName$-matrices from the above expressions, and the rest is purely computational.

As before, we have
\[ 144\Delta_2^\circ = \sum_{\ell_1+\ell_2+\ell_3=0} \frak{D}^{2,2,2}_{\ell_1,\ell_2,\ell_3} Z_{\ell_1} \circ Z_{\ell_2} \circ Z_{\ell_3}+\sum_{\ell_2,\ell_3} \frak{F}^{2,2}_{\ell_2,\ell_3} Z_{\ell_2} \circ Z_{\ell_3}+48\sqrt{6} Z_0+144\Delta_1^\circ, \]
and \eqref{eq:FullDelta2} follows by using
\begin{align}
	[\KLeft{0}, \KLeft{\pm1}] = \mp \sqrt{2} i \KLeft{\pm 1}, \qquad [\KLeft{1},\KLeft{-1}]=\sqrt{2}i \KLeft{0},
\end{align}
and
\begin{align}
	[Z_{\pm2},Z_0]=&0 & [Z_{\pm1},Z_0]=&\sqrt{6}Z_{\pm1} & [Z_{-1},Z_1]=&0 \\
	[Z_{\pm2},Z_{\pm1}]=& -Z_{\pm1} & [Z_{\pm2},Z_{\mp1}]=&Z_{\mp1}-2Z_{\pm1} &[Z_{-2},Z_2]=&2Z_2-2Z_{-2}.
\end{align}

\bibliographystyle{amsplain}

\bibliography{HigherWeight}

\end{document}